\documentclass[11pt,reqno]{amsart} 

\usepackage{graphicx, upref}

\usepackage[top=1.75in,bottom=1.75in,left=1.75in,right=1.75in]{geometry}

\usepackage{xcolor}

\usepackage[labelfont=bf,font=md]{caption}

\usepackage[font={sf,small},labelfont=md]{subfig}

\usepackage[noadjust,nosort]{cite}

\usepackage{amsmath, amsthm, amssymb}
\usepackage{enumerate}
\usepackage{tikz}
\usepackage{mathtools}
\usepackage{array}
\usepackage{mathrsfs}
\usepackage{tensor}

\usepackage{polynom}

\usepackage{multicol}

\usepackage{dsfont}
	\newcommand{\one}{\mathds{1}}

\usepackage{framed}

\usepackage{hyperref}
	\hypersetup{colorlinks=true,linkcolor=blue,citecolor=blue,urlcolor=black}
	
\allowdisplaybreaks

\numberwithin{equation}{section}


\newcommand{\eq}[1]{\begin{align*} #1 \end{align*}}
\newcommand{\eeq}[1]{\begin{align} \begin{split} #1 \end{split} \end{align}}

\newcommand{\stackref}[2]{\stackrel{\mbox{\footnotesize{\eqref{#1}}}}{#2}}
\newcommand{\stackrefp}[2]{\stackrel{\phantom{\mbox{\footnotesize{\eqref{#1}}}}}{#2}}

\newcommand{\R}{\mathbb{R}}
\newcommand{\Z}{\mathbb{Z}}

\newcommand{\E}{\mathbb{E}}
\renewcommand{\P}{\mathbb{P}}

\renewcommand{\AA}{\mathcal{A}}
\newcommand{\PP}{\mathcal{P}}
\newcommand{\FF}{\mathcal{F}}

\newcommand{\FFF}{\mathscr{F}}
\newcommand{\EEE}{\mathscr{E}}
\newcommand{\LL}{\mathcal{L}}

\newcommand{\BB}{\mathcal{B}}
\newcommand{\RR}{\mathcal{R}}
\newcommand{\MM}{\mathcal{M}}
\newcommand{\NN}{\mathcal{N}}

\newcommand{\ff}{\mathfrak{f}}

\newtheorem{thm}{Theorem}[section]
\newtheorem{prop}[thm]{Proposition}
\newtheorem{cor}[thm]{Corollary}
\newtheorem{lemma}[thm]{Lemma}
\newtheorem{claim}[thm]{Claim}

\theoremstyle{definition}
\newtheorem{defn}[thm]{Definition}

\newtheorem{remark}[thm]{Remark}

\def\eps{\varepsilon}
\def\vphi{\varphi}

\newcommand{\iprod}[2]{\langle #1,\, #2\rangle}

\newcommand{\vc}[1]{{\boldsymbol #1}}

\newcommand{\wt}[1]{\widetilde{#1}}

\DeclareMathOperator{\Var}{Var}
\DeclareMathOperator{\Cov}{Cov}

\newcommand{\givenk}[3][]{#1[ #2 \: #1| \: #3 #1]} 
\newcommand{\givenp}[3][]{#1( #2 \: #1| \: #3 #1)} 
\newcommand{\givena}[3][]{#1\langle #2 \: #1| \: #3 #1\rangle} 

\DeclareMathOperator{\e}{e}

\newcommand{\dd}{\mathrm{d}}
\newcommand{\cc}{\mathrm{c}}

\DeclarePairedDelimiter\ceil{\lceil}{\rceil}
\DeclarePairedDelimiter\floor{\lfloor}{\rfloor}

        \makeatletter
        \DeclareFontFamily{OMX}{MnSymbolE}{}
        \DeclareSymbolFont{MnLargeSymbols}{OMX}{MnSymbolE}{m}{n}
        \SetSymbolFont{MnLargeSymbols}{bold}{OMX}{MnSymbolE}{b}{n}
        \DeclareFontShape{OMX}{MnSymbolE}{m}{n}{
            <-6>  MnSymbolE5
           <6-7>  MnSymbolE6
           <7-8>  MnSymbolE7
           <8-9>  MnSymbolE8
           <9-10> MnSymbolE9
          <10-12> MnSymbolE10
          <12->   MnSymbolE12
        }{}
        \DeclareFontShape{OMX}{MnSymbolE}{b}{n}{
            <-6>  MnSymbolE-Bold5
           <6-7>  MnSymbolE-Bold6
           <7-8>  MnSymbolE-Bold7
           <8-9>  MnSymbolE-Bold8
           <9-10> MnSymbolE-Bold9
          <10-12> MnSymbolE-Bold10
          <12->   MnSymbolE-Bold12
        }{}
        
        \let\llangle\@undefined
        \let\rrangle\@undefined
        \DeclareMathDelimiter{\llangle}{\mathopen}%
                             {MnLargeSymbols}{'164}{MnLargeSymbols}{'164}
        \DeclareMathDelimiter{\rrangle}{\mathclose}%
                             {MnLargeSymbols}{'171}{MnLargeSymbols}{'171}
        \makeatother

\renewcommand{\thefootnote}{\fnsymbol{footnote}}

\title[Localization in disordered systems]{Localization in Gaussian disordered systems at low temperature}



\author{Erik Bates}
\thanks{E.B. was partially supported by NSF grants DGE-114747 and DMS-1902734} 
\address{\newline Department of Mathematics \newline University of California, Berkeley \newline 1067 Evans Hall \newline Berkeley, CA 94720-3840 \newline \textup{\tt ewbates@berkeley.edu}}

\author{Sourav Chatterjee}
\thanks{S.C. was partially supported by NSF grant DMS-1608249.}
\address{\newline Department of Statistics \newline Stanford University \newline Sequoia Hall, 390 Jane Stanford Way \newline Stanford, CA 94305 \newline \textup{\tt souravc@stanford.edu}}

\begin{document}
\bibliographystyle{acm}

\renewcommand{\thefootnote}{\arabic{footnote}} \setcounter{footnote}{0}

\begin{abstract}
For a broad class of Gaussian disordered systems at low temperature, we show
that the Gibbs measure is asymptotically localized in small neighborhoods of a small number of states. From a single argument, we obtain (i) a  version of ``complete" path localization for directed polymers that is not
available even for exactly solvable models; and (ii) a result about the exhaustiveness of Gibbs states in spin glasses not requiring the Ghirlanda--Guerra identities.
\end{abstract}

\maketitle

\section{Introduction}
A ubiquitous theme in statistical mechanics is to understand how a system
behaves differently at high and low temperatures. In a disordered system,
where the interactions between its elements are governed by random quantities, the strength of the disorder is determined by temperature. Namely,
high temperatures mean the disorder is weak, and the system is likely to
resemble a generic one based on entropy. On the other hand, low temperatures indicate strong disorder, which creates dramatically different behavior
in which the system is constrained to a small set of states that are energetically favorable. In the latter case, this concentration phenomenon is often
called ``localization".

A useful statistic in distinguishing different temperature regimes is the
so-called ``replica overlap". That is, given the disorder, one can study the
similarity of two independently observed states. If the disorder is strong,
then these two states should closely resemble one another with good probability, since we believe the system is bound to a relatively small number
of possible realizations. Some version of this statement has been rigorously
established in a number of contexts, most famously in spin glass theory but
also in the settings of disordered random walks and disordered Brownian motion. Unfortunately, it does not follow that the number of realizable states
is small, but only that there is small number of states that are observed with
positive probability.

In the present study, our entry point to this problem is to consider \textit{conditional} overlap. Whereas previous results in the literature show the overlap
distribution between two independent states has a nonzero component, we
ask whether the same is true even if one conditions on the first state. That is,
does a typical state \textit{always} have positive expected overlap with an independent one? We show that for a broad class of Gaussian disordered systems,
the answer is yes, the key implication being that the \textit{entire} realizable state
space is small. Specifically, there is an $O(1)$ number of states such that all but a negligible fraction of samples from the system will have positive overlap with one of these
states. 

The general setting, notation, motivation, and results are given in Sections~\ref{model}--\ref{results}, respectively. 
The consequences for spin glasses, directed polymers,
and other Gaussian fields are discussed in Sections~\ref{applications} and~\ref{other_fields}.

\subsection{Model and assumptions} \label{model}
Let $(\Omega,\FF,\P)$ be an abstract probability space, and $(\Sigma_n)_{n\geq1}$ a sequence of Polish spaces 
equipped respectively with probability measures $(P_n)_{n\geq1}$. 
For each $n$, we consider a centered Gaussian field $H_n$ indexed by $\Sigma_n$ and defined on $\Omega$.  
Viewing this field as a Hamiltonian, we have the associated Gibbs measure at inverse temperature $\beta$:
\eq{
\mu_{n}^\beta(\dd\sigma) \coloneqq \frac{\e^{\beta H_n(\sigma)}}{Z_n(\beta)}\ P_n(\dd\sigma), \quad \text{where} \quad
Z_n(\beta) \coloneqq \int \e^{\beta H_n(\sigma)}\ P_n(\dd\sigma).
}
%
Our results concern the relationship between the \textit{free energy},
\eq{
F_n(\beta) \coloneqq \frac{1}{n}\log Z_n(\beta),
}
and the covariance structure of $H_n$.
We make the following assumptions:
\begin{itemize}
\item There is a deterministic function $p : \R \to \R$ such that
\[\label{free_energy_assumption} \tag{A1}
\lim_{n\to\infty} F_n(\beta) = p(\beta) \quad \P\text{-}\mathrm{a.s.} \text{ and in $L^1(\P)$, for every $\beta\in\R$}. 
\]
\item For every $\sigma\in\Sigma_n$,
\[ \label{variance_assumption} \tag{A2}
\Var H_n(\sigma) = n.
\]
\item For every $\sigma^1,\sigma^2\in\Sigma_n$,
\[ \label{positive_overlap} \tag{A3}
\Cov(H_n(\sigma^1),H_n(\sigma^2))\geq-n\EEE_n,
\]
where $\EEE_n$ is a nonnegative constant tending to $0$ as $n\to\infty$.
\item For each $n$, there exist measurable real-valued functions $(\vphi_{i,n})_{i=1}^\infty$ on $\Sigma_n$ 
and i.i.d.~standard normal random variables $(g_{i,n})_{i=1}^\infty$ defined on $\Omega$ such that for each $\sigma\in\Sigma_n$, with $\P$-probability $1$, 
\[ \label{field_decomposition} \tag{A4}
H_n(\sigma)  = \sum_{i=1}^\infty g_{i,n}\vphi_{i,n}(\sigma), 
\]
where the series on the right converges in $L^2(\P)$. 
\end{itemize}

\begin{remark}
In all applications of interest (see Section~\ref{applications}), the hypothesis \eqref{positive_overlap} is trivially satisfied with $\EEE_n = 0$.
Nevertheless, we assume throughout only that $\EEE_n\to0$ (at any rate).
This modest relaxation is made so our results can apply to slightly more general models, for instance perturbations of the standard models we will soon describe.
\end{remark}

\begin{remark}
The condition \eqref{field_decomposition} is very mild: For example, it always holds when $\Sigma_n$ is finite. More generally, a sufficient condition for the existence of a representation \eqref{field_decomposition} is that $\Sigma_n$ is compact in the metric defined by $H_n$ (namely, the metric that defines the distance between $\sigma$ and $\sigma'$ as the $L^2$ distance between the random variables $H_n(\sigma)$ and $H_n(\sigma')$). For a proof of this standard result, see \cite[Theorem 3.1.1]{adler-taylor07}. 
Furthermore, in all applications of interest, $H_n$ will actually be explicitly defined using a sum of the form~\eqref{field_decomposition}.


\end{remark}

\subsection{Notation} \label{notation}
Unless stated otherwise, ``almost sure" and ``in $L^\alpha$" statements are with respect to $\P$.
We will use $E_n$ and $\E$ to denote expectation with respect to $P_n$ and $\P$, respectively.
Absent any decoration, $\langle \cdot \rangle$ will always denote expectation with respect to $\mu_{n}^{\beta}$, meaning
\eq{
\langle f(\sigma)\rangle = \frac{E_n(f(\sigma)\e^{\beta H_n(\sigma)})}{E_n(\e^{\beta H_n(\sigma)})}.
}
At various points in the paper, we will decorate $\langle\cdot\rangle$ to denote expectation
with respect to some perturbation of $\mu_{n}^\beta$. 
The type of perturbation will
change between sections.
The symbols $\sigma^{j}$, $j = 1,2,\dots$, shall denote independent samples from $\mu_{n}^\beta$ if appearing within $\langle\cdot\rangle$, or from $P_n$ if appearing within~$E_n(\cdot)$.
We will refer to the vector $\vc g_n = (g_{i,n})_{i=1}^\infty$ as the \textit{disorder} or \textit{random environment}. 
Sometimes we will consider multiple environments at the same time, which will necessitate that we write $\mu_{n,\vc g_n}^\beta$ instead of $\mu_n^\beta$ to emphasize the dependence on the environment $\vc g_n$.

In the sequel, $\sum_i$ will always mean $\sum_{i=1}^\infty$,
and we will condense our notation to $\vphi_i = \vphi_{i,n}(\sigma)$ when we are dealing with some fixed $n$. Similarly, $g_{i,n}$ will be shortened to $g_i$ and $\vc g_n$ will be shortened to $\vc g$. 
Also, $C(\cdot)$ will indicate a positive constant that depends only on the argument(s).
In particular, no such constant depends on $\vc g$ or $n$.
We will not concern ourselves with the
precise value, which may change from line to line.

\subsection{Motivation} \label{motivation}
Our results will be stated in terms of the correlation or \textit{overlap function},
\eq{
\RR(\sigma^1,\sigma^2) &\coloneqq \frac{1}{n}\Cov(H_n(\sigma^1),H_n(\sigma^2)), 
\quad \sigma^1,\sigma^2\in\Sigma_n.
}
Note that \eqref{variance_assumption} and \eqref{positive_overlap} imply
\eq{
-\EEE_n\leq\RR(\sigma^1,\sigma^2) \leq 1.
}
We will often abbreviate $\RR(\sigma^j,\sigma^k)$ to $\RR_{j,k}$. 

The Gaussian process $(H_n(\sigma))_{\sigma\in\Sigma_n}$ naturally defines a
(pseudo)metric $\rho$ on $\Sigma_n$, given by
\eeq{\label{rhodef}
\rho(\sigma^1,\sigma^2)\coloneqq 1-\RR_{1,2}.
}
Given the metric topology, we can study the so-called ``energy landscape"
of $\beta H_n$ on $\Sigma_n$. 
The geometry of this landscape is intimately
related to the free energy. By Jensen's inequality,
\eeq{ \label{basic_jensen_gap}
\E F_n(\beta) \leq \frac{1}{n}\log \E Z_n(\beta) \stackrel{\mbox{\footnotesize(\text{Lemma }\ref{moments_lemma})}}{=}\frac{\beta^2}{2},
}
which in particular implies $p(\beta) \leq \beta^2/2$. In general, whether or not this
inequality is strict determines the nature of the energy landscape: In order
for $p(\beta) = \beta^2/2$, the fluctuations of $\log Z_n(\beta)$ must be relatively small so that
the Jensen gap in \eqref{basic_jensen_gap} is $o(1)$. This behavior arises when the Gaussian deviations
of $\beta H_n(\sigma)$ are washed out by the entropy of $P_n$, creating a more or less flat
landscape. On the other hand, if $p(\beta) < \beta^2/2$, then these deviations will
have overcome the entropy of $P_n$, producing large peaks and valleys where
$\beta H_n(\sigma)$ is exceptionally positive or negative. From a physical perspective,
this latter scenario is more interesting, as these peaks can account for an
exponentially vanishing fraction of the state space even as their union accounts for a non-vanishing fraction of the mass of $\mu_{n}^\beta$.
The primary goal of this paper is to
give a sufficient condition for when (in a sense Theorem~\ref{easy_cor} makes precise)
$\mu_{n}^\beta$ places all of its mass on this union of peaks.

Suppose that $p(\cdot)$ is differentiable at $\beta\geq0$. 
Using Gaussian integration by parts, it is not difficult to show (as we do in Corollary~\ref{overlap_identity_cor}) that
\eeq{ \label{expected_overlap_formula}
\lim_{n\to\infty} \E\langle \RR_{1,2}\rangle = 1 - \frac{p'(\beta)}{\beta}.
}
This identity has been observed before (e.g.~see \cite{aizenman-lebowitz-ruelle87,comets-neveu95,talagrand06III,panchenko08}, \cite[Lemma 7.1]{carmona-hu02}, and \cite[Theorem 6.1]{comets17}).
For this reason, the condition in which we are interested is $p'(\beta)<\beta$.
To improve upon \eqref{expected_overlap_formula}, a first step is to show that if $\E\langle \RR_{1,2}\rangle$ is bounded away
from $0$, then the random variable $\langle \RR_{1,2}\rangle$ is itself stochastically bounded away
from $0$. This is the content of Theorem~\ref{averages_squared}. 
The more substantial contribution
of this paper, however, is to bootstrap this result to a proof of Theorem~\ref{expected_overlap_thm},
which roughly says that $\langle \RR_{1,2}\rangle$ is stochastically bounded away from $0$ even
conditional on $\sigma^1$.

It follows from Corollary~\ref{overlap_identity_cor} that $p'(\beta) < \beta$ implies $p(\beta) < \beta^2/2$,
but it is natural to ask whether the two conditions are equivalent.
This equivalence is true for spin glasses \cite{talagrand06III,panchenko08} and is believed to be true for directed
polymers \cite[Conjecture 6.1]{comets17}. But at the level of generality considered in
this paper, we are not aware of any conjecture. 
In any case, for the examples we consider in Section~\ref{applications}, both conditions will be true for sufficiently large $\beta$.

\subsection{Results} \label{results}
Our main result is Theorem~\ref{easy_cor}, stated below. It says that at low temperatures, one can find a finite number of (random) states such that
almost any sample from the Gibbs measure will have positive overlap with
at least one of them.
To state this precisely, let us define the sets
\eeq{ \label{ball_def}
\BB(\sigma,\delta) \coloneqq \{\sigma'\in\Sigma_n : \RR(\sigma, \sigma')\geq\delta\}, \quad \sigma\in\Sigma_n,\, \delta>0.
}
In terms of the metric $\rho$ defined in \eqref{rhodef}, this is just the ball of radius $1-\delta$ centered at $\sigma$. Typically, such balls have  vanishingly small size under $P_n$ as $n\to\infty$, which should be contrasted with the following behavior of the Gibbs measure.
\begin{thm} \label{easy_cor}
Assume \eqref{free_energy_assumption}--\eqref{field_decomposition}.
If $\beta\geq0$ is a point of differentiability for $p(\cdot)$, and $p'(\beta) < \beta$,
then for every $\eps > 0$, there exist integers $k = k(\beta,\eps)$ and $n_0 = n_0(\beta,\eps)$ and a number $\delta = \delta(\beta,\eps)>0$ such that the following is true for all $n\geq n_0$.
With $\P$-probability at least $1-\eps$, there exist $\sigma^1,\dots,\sigma^k\in\Sigma_n$ such that 
\eq{
\mu_{n}^\beta\Big(\bigcup_{j=1}^k \BB(\sigma^j, \delta)\Big) \geq 1 - \eps.
}
\end{thm}
It is worth noting that in some cases, such as the directed polymer model defined in Section~\ref{directed_polymers}, it is possible (although unproven) that $k$ can be taken equal to $1$ if $\delta$ is chosen sufficiently small.
For other models, however, such as polymers on trees or the Random Energy Model discussed in Section~\ref{other_fields}, $k$ will necessarily diverge as $\eps\to0$.

We will derive Theorem~\ref{easy_cor} as a corollary of Theorem~\ref{expected_overlap_thm}, stated below. In fact,  Theorem~\ref{easy_cor} is actually equivalent to Theorem~\ref{expected_overlap_thm}, although the latter has a less transparent statement, which is why we have stated Theorem~\ref{easy_cor} as our main result.


Theorem~\ref{expected_overlap_thm} concerns the following function on $\Sigma_n$.
For given $\sigma^1\in\Sigma_n$, we will write the conditional expectation of $\RR_{1,2}$ as
\eeq{\label{rdef}
\RR(\sigma^1) \coloneqq \givena{\RR_{1,2}}{\sigma^1}
= \frac{1}{n}\sum_{i=1}^\infty \vphi_{i,n}(\sigma^1) \langle\vphi_{i,n}(\sigma^2)\rangle.
}
(Note that the expectation $\givena{\cdot}{\sigma^1}$ can be exchanged with the sum because of Fubini's theorem, in light of \eqref{variance_assumption}.)
Given $\delta>0$, we consider the set
\eeq{ \label{A_def}
\AA_{n,\delta} \coloneqq \{\sigma\in\Sigma_n : \RR(\sigma) \leq \delta\}. 
} 
With this notation, the quantity $\langle \one_{\AA_{n,\delta}}\rangle$ 
is the probability that a state sampled from $\mu_{n}^\beta$ has expected overlap at most $\delta$ with an independent sample from $\mu_{n}^\beta$.
Theorem~\ref{expected_overlap_thm} says that at low temperatures and for small $\delta$, this probability is typically small.

\begin{thm} \label{expected_overlap_thm}
Assume \eqref{free_energy_assumption}--\eqref{field_decomposition}.
If $\beta\geq0$ is a point of differentiability for $p(\cdot)$, and $p'(\beta) < \beta$, then for every $\eps>0$, there exists $\delta = \delta(\beta,\eps) > 0$ sufficiently small that
\eeq{ \label{expected_overlap_thm_eq}
\limsup_{n\to\infty} \E\langle\one_{\AA_{n,\delta}}\rangle \leq \eps.
}
\end{thm}

To prove Theorem~\ref{expected_overlap_thm}, we first have to prove a weaker theorem stated below. This result considers the following event in the $\sigma$-algebra $\FF$,
\eq{
B_{n,\delta} \coloneqq \{\langle \RR_{1,2}\rangle \leq \delta\},
}
and shows that its probability is small at low temperature. 

\begin{thm} \label{averages_squared}
Assume \eqref{free_energy_assumption}--\eqref{field_decomposition}.
If $\beta\geq0$ is a point of differentiability for $p(\cdot)$, and $p'(\beta) < \beta$, then for every $\eps > 0$, there exists $\delta = \delta(\beta,\eps) > 0$ sufficiently small such that
\eeq{ \label{averages_squared_eq}
\limsup_{n\to\infty} \P(B_{n,\delta}) \leq \eps.
}
\end{thm}


Theorem~\ref{averages_squared} is proved in Section~\ref{proof_1}, Theorem~\ref{expected_overlap_thm} in Section~\ref{proof_2}, and the equivalence of Theorems~\ref{easy_cor} and~\ref{expected_overlap_thm} in Section~\ref{cor_proof}.
In Section~\ref{prep_section}, we provide some general facts that are needed in the main arguments. 
A detailed sketch of the proof technique is given  in Section~\ref{sketchsec}. 
We will often simplify notation by writing $\AA_\delta$ and $B_\delta$, where the dependence on $n$ is understood and will not be a source of confusion.

\subsection{Applications} \label{applications}
For many applications, it would suffice to consider $\Sigma_n$ which is finite for every $n$.
Other applications, however, such as spherical spin glasses or directed polymers with a reference walk of unbounded support, require $\Sigma_n$ to be infinite.
It is for this reason that we have stated the setting and results in the generality seen above.
Now we discuss specific models of interest.

\subsubsection{Spin glasses} \label{spin_glasses}
Let $\Sigma_n = \{\pm 1\}^n$ (Ising case) or $\Sigma_n = \{\sigma\in\R^n :\|\sigma\|_2 = \sqrt{n}\}$ (spherical case), and take $P_n$ to be uniform measure on $\Sigma_n$. 
In the mean-field models, the Hamiltonian is of the form
\eeq{ \label{sg_hamiltonian}
H_n(\sigma) = \sum_{p\geq2} \frac{\beta_p}{n^{(p-1)/2}}\sum_{i_1,\dots,i_p = 1}^n g_{i_1,\dots,i_p}\sigma_{i_1}\cdots\sigma_{i_p}.
}
We will assume
\eeq{ \label{sg_assumption_1}
\sum_{p\geq2}\beta_p^2(1+\eps)^p < \infty \quad \text{for some $\eps>0$},
}
which is more restrictive than what we require but standard in the literature.
Standard applications of Gaussian concentration show that $|F_n(\beta) - \E F_n(\beta)|\to0$ almost surely and in $L^1$.
Assumption \eqref{free_energy_assumption} then follows from the convergence of $\E F_n(\beta)\to p(\beta)$, where $p(\beta)$ is given by a formula depending on the model.
In the Ising case, there is the celebrated Parisi formula \cite{parisi79,parisi80}, proved by
Talagrand \cite{talagrand06} for even-spin models, building on the seminal work of Guerra \cite{guerra03}.
It was later extended by Panchenko \cite{panchenko14} to general mixed $p$-spins.
For the spherical model, there is a simpler and elegant formula predicted
by Crisanti and Sommers \cite{crisanti-sommers92}, and proved by Talagrand \cite{talagrand06II} and Chen \cite{chen13}.

To accommodate assumptions \eqref{variance_assumption} and \eqref{positive_overlap}, one should assume the function $\xi(q) \coloneqq \sum_{p\geq2}\beta_p^2q^p$ satisfies
\eeq{ \label{sg_assumption_2}
\xi(1) = 1 \quad \text{and} \quad \xi(q) \geq 0 \quad \text{for all $q\in[-1,1]$.}
}
This is because 
\eq{
\RR_{j,k} = \xi(R_{j,k}), \quad \text{where} \quad R_{j,k} \coloneqq \frac{1}{n}\sum_{i=1}^n \sigma^j_i\sigma^k_i \in [-1,1].
}
Note that the second assumption in \eqref{sg_assumption_2} is automatic if $\beta_p = 0$ for all odd $p$.
When $\xi(q) = q^2$, \eqref{sg_hamiltonian} is the classical Sherrington--Kirkpatrick (SK) model~\cite{sherrington-kirkpatrick75} if $\Sigma_n = \{\pm1\}^n$, or the spherical SK model~\cite{kosterlitz-thouless-jones76} if $\Sigma_n = \{\sigma\in\R^n : \|\sigma\|_2 = \sqrt{n}\}$.

In the spin glass literature, $R_{1,2}$ is the usual \textit{replica overlap} that is studied as an order parameter for the system \cite{talagrand03}.
Roughly speaking, $R_{1,2}$ converges
to $0$ when $p(\beta) = \beta^2/2$, but converges in law to a non-trivial distribution
when $p(\beta) < \beta^2/2$. 
In the latter case, the model exhibits what is known as
replica symmetry breaking (RSB). 
If the limiting distribution of $R_{1,2}$, called
the Parisi measure, contains $k + 1$ distinct atoms (one of which must be $0$
\cite{auffinger-chen15I}), then $\xi$ is said to be $k$RSB. 
For instance, spherical pure $p$-spin models
are $1$RSB for large $\beta$ \cite{panchenko-talagrand07}, and it was recently shown that some spherical mixed spin models are $2$RSB at zero temperature \cite{auffinger-zeng19}.
In the Ising case, however, the Parisi measure is
expected to have an infinite support throughout the low-temperature phase
(with $0$ in the support but not as an atom; see \cite[Page 15]{bolthausen07}), a behavior
referred to as \textit{full}-RSB (FRSB). Proving such a statement is a problem of
great interest and has been solved at zero temperature \cite{auffinger-chen-zeng20}.
For spherical models, the situation is somewhat clearer; in \cite{chen-sen17}, sufficient conditions were given for both $1$RSB and FRSB, again at zero temperature.

The simplest type of symmetry breaking, $1$RSB, admits the following
heuristic picture. The state space $\Sigma_n$ is (from the perspective of $\mu_{n}^\beta$)
separated into many orthogonal parts called ``pure states", within which the intra-cluster overlap concentrates on some positive value $q > 0$. 
In the $2$RSB picture, the pure states are not necessarily orthogonal, but rather grouped
together into larger clusters which are themselves orthogonal. In this case,
the overlap could be $q$ (same pure state), $q'\in (0, q)$ (same cluster but different pure state), or $0$ (different clusters). 
The complexity increases in the same
fashion for general $k$RSB. 
In FRSB, the clusters become infinitely nested,
yielding a continuous spectrum of possible overlaps while maintaining ``ultrametric" structure \cite{panchenko13II}. 
In any case, though, there should be asymptotically
no part of the state space which is orthogonal to everything; that is, the pure
states exhaust $\mu_{n}^\beta$.

Absent the intricate hierarchical picture described above, the following rephrasing of Theorem~\ref{easy_cor} confirms this idea.

\begin{thm} \label{sg_thm}
Assume \eqref{sg_assumption_1} and \eqref{sg_assumption_2}, and that $\beta\geq0$ is a point of differentiability for $p(\cdot)$ such that $p'(\beta) < \beta$.
Then for every $\eps > 0$, there exist integers $k = k(\beta,\eps)$ and $n_0 = n_0(\beta,\eps)$ and a number $\delta = \delta(\beta,\eps)>0$ such that the following is true for all $n\geq n_0$.
With $\P$-probability at least $1-\eps$, there exist $\sigma^1,\dots,\sigma^k\in\Sigma_n$ such that
\eq{ 
\mu_{n}^\beta\Big(\bigcup_{j=1}^k \{\sigma^{k+1}\in\Sigma_n : |R_{j,k+1}|\geq\delta\}\Big) \geq 1 - \eps.
}
\end{thm}

%

The proof of the above Theorem follows simply from Theorem~\ref{easy_cor} and the observation that by \eqref{sg_assumption_1}, $\xi$ is continuous at $0$. 

Under strong assumptions on $\xi$ and the overlap distribution, namely the (extended) Ghirlanda--Guerra identities, much more precise results were proved by \mbox{Talagrand} \cite[Theorem 2.4]{talagrand10} and later Jagannath \cite[Corollary 2.8]{jagannath17}. 
For spherical pure spin models, similar results were
proved by Subag \cite[Theorem 1]{subag17}. 
An advantage of our approach, beyond its
generality, is that our assumptions on $\xi$ are elementary to check and fairly
loose (they include all even spin models), and the temperature condition
$p'(\beta) < \beta$ is explicit and sharp.

While the literature on replica overlaps in spin glasses is vast, the reader
will find much information in \cite{mezard-parisi-virasoro87,talagrand11I,talagrand11II,panchenko13}; see also \cite{jagannath-tobasco17II} and references therein.

\subsubsection{Directed polymers} \label{directed_polymers}
Given a positive integer $d$, let $\Sigma_n$ be the set of all maps from $\{0,1,\ldots,n\}$ into $\Z^d$, and let $P_n$ be the law, projected onto $\Sigma_n$, of a homogeneous random walk on $\Z^d$ starting at the origin.
That is, there is some probability mass function
$K$ on $\Z^d$ such that
\begin{subequations} \label{walk_assumption}
\begin{align}
P_n(\sigma(0) = 0) &= 1,\\
P_n\givenp{\sigma(i) = y}{\sigma(i-1)=x} &= K(y-x), \quad 1\leq i\leq n.
\end{align}
\end{subequations}
Let $(g(i,x) : i\geq1,x\in\Z^d)$ be i.i.d.~standard normal random variables.
The Hamiltonian for the model of directed polymers in Gaussian environment is then given by
\eq{
H_n(\sigma) = \sum_{i=1}^n g(i,\sigma(i)) = \sum_{i=1}^n\sum_{x\in\Z^d} g(i,x)\one_{\{\sigma(i)=x\}}.
}
In this case, the overlap between two paths is the fraction of time they intersect:
\eeq{ \label{polymer_overlap}
\RR_{1,2} = \frac{1}{n}\sum_{i=1}^n \one_{\{\sigma^1(i)=\sigma^2(i)\}}.
}
The assumption \eqref{free_energy_assumption} holds for any $K$ \cite[Section 2]{bates18}, although typically
$P_n$ is taken to be standard simple random walk; all the references below
refer to this case. 
Alternatively, one can consider point-to-point polymer measures, meaning the endpoint of the polymer is fixed.
This case is studied in \cite{rassoul-seppalainen14,georgiou-rassoul-seppalainen16} and accommodates the same structure as above, up to changing the reference measure $P_n$.

Notice that the identity \eqref{expected_overlap_formula} immediately implies $\lim_{n\to\infty}\E\langle \RR_{1,2}\rangle > 0$ when $p'(\beta) < \beta$. 
Theorem~\ref{averages_squared} goes a step further, showing that the random variable $\langle \RR_{1,2}\rangle$ is itself stochastically bounded away from 0.
For a certain class of
bounded random environments,  a quantitative version of Theorem~\ref{averages_squared} was proved by Chatterjee~\cite{chatterjee19}, but Theorem~\ref{expected_overlap_thm} is the first of its kind. 
Unlike some other conjectured polymer properties, the statement \eqref{expected_overlap_thm_eq} has not been verified for the so-called exactly solvable models in $d = 1$ \cite{seppalainen12,corwin-seppalainen-shen15,oconnell-ortmann15,barraquand-corwin17,thiery-doussal15}. For heavy-tailed
environments, a stronger notion of localization is considered in \cite{auffinger-louidor11,torri16} and
also discussed in \cite{dey-zygouras16,berger-torri19}. 
Historically, studying pathwise localization has found
somewhat greater success in the context of \textit{continuous} space-time polymer
models \cite{comets-yoshida05,comets-yoshida13,comets-cranston13,comets-cosco??}.

For polymers in Gaussian environment, it is known (see \cite[Proposition 2.1(iii)]{comets17}) that $p'$
is bounded from above by a constant, and so $\E\langle \RR_{1,2}\rangle\to1$ as $\beta\to\infty$ by \eqref{expected_overlap_formula}. 
(While convexity guarantees $p(\cdot)$ is differentiable almost everywhere, it is an open problem to show that $p(\cdot)$ is everywhere differentiable, let
alone analytic away from the critical value separating the high and low temperature phases.) In this sense, the polymer measure becomes completely
localized near the maximizer of $H_n(\cdot)$ as $\beta\to\infty$. A main motivation for the
present study was to formulate a version of ``complete localization" for fixed
$\beta$ in the low-temperature regime.

In \cite{vargas07,bates-chatterjee20}, complete localization was phrased in terms of the endpoint
distribution: the law of $\sigma(n)$ under $\mu^\beta_{n}$.
Loosely speaking, what was shown
is that if $p(\beta) < \beta^2/2$, then with probability at least $1-\eps$, one can find sufficiently many (independent of $n$) random vertices $x_1,\dots,x_k$ in $\Z^d$ so that
\eeq{ \label{atomic_localization}
\mu^\beta_{n}\big(\big\{\sigma : \sigma(n) \in \{x_1,\dots,x_k\}\big\}\big) \geq 1 - \eps.
}
This behavior is called ``asymptotic pure atomicity", referring to the fact that
even as $n$ grows large, the endpoint distribution remains concentrated on an
$O(1)$ number of sites (rather than diffuse polynomially as in simple random
walk). 
This is analogous to the results of this paper, except that the endpoint statistic
has been used to reduce the state space to $\Z^d$. 
The pathwise localization in Theorem~\ref{easy_cor} describes a more global phenomenon occurring in the original state space $\Sigma_n$.
Rephrased below, it says that up to arbitrarily small probabilities, the Gibbs measure is concentrated on paths intersecting  one of a few distinguished paths a positive fraction of the time. 

\begin{thm} \label{polymer_thm}
Assume \eqref{walk_assumption} and that $\beta\geq0$ is a point of differentiability for $p(\cdot)$ such that $p'(\beta) < \beta$.
Then for every $\eps > 0$, there exist integers $k = k(\beta,\eps)$ and $n_0 = n_0(\beta,\eps)$ and a number $\delta = \delta(\beta,\eps)>0$ such that the following is true for all $n\geq n_0$.
With $\P$-probability at least $1-\eps$, there exist paths $\sigma^1,\dots,\sigma^k\in\Sigma_n$ such that
\eq{
\mu_{n}^\beta\bigg(\bigcup_{j=1}^k \Big\{\sigma^{k+1}\in\Sigma_n: \frac{1}{n}\sum_{i=1}^n\one_{\{\sigma^{k+1}(i)=\sigma^j(i)\}}\geq\delta\Big\}\bigg) \geq 1 - \eps.
}
\end{thm}

In Section~\ref{no_path_atom}, we demonstrate that path localization does not
occur in the atomic sense \eqref{atomic_localization}. 
That is, any bounded number of paths will have a total mass under
$\mu^\beta_{n}$ that decays to $0$ as $n\to\infty$.
For this reason, the definitions from \cite{vargas07,bates-chatterjee20} of complete localization for the endpoint are inadequate for path
localization, necessitating a statement in terms of overlap. This distinguishes
the lattice polymer model from its mean-field counterpart on
regular trees, which is simply the statistical mechanical version of branching
random walk \cite{derrida-spohn88,comets17}. 
For those models, the endpoint distribution on the
leaves of the tree is obviously equivalent to the Gibbs measure because each
leaf is the termination point of a unique path. Moreover, the results of \cite{bates-chatterjee20}
can be interpreted equally well (and improved upon) in that setting (see
\cite{barral-rhodes-vargas12,jagannath16}), and so we will not elaborate on the fact that polymers on trees also fit into the framework of this paper.

\subsection{Other Gaussian fields} \label{other_fields}
Here we mention several other models to
which our results apply but for which they are not new. Indeed, each model
below is known to exhibit Poisson--Dirichlet statistics for the masses assigned by
$\mu^\beta_{n}$ to the ``peaks" discussed in the motivating Section~\ref{motivation}. 
In particular, asymptotically no mass is given to states having vanishing expected overlap with an independent sample.
\begin{itemize}
\item Derrida's Random Energy Model (REM) \cite{derrida80,derrida81} is set on the hypercube
$\Sigma_n = \{\pm1\}^n$ with uniform measure, and has the simplest possible
covariance structure: $\RR_{j,k} = \delta_{j,k}$.
With $\beta_\mathrm{c} = \sqrt{2\log 2}$, the following formula holds \cite[Theorem 9.1.2]{bovier06}:
\eq{
p(\beta) = \begin{cases}
\beta^2/2 &\beta \leq \beta_\mathrm{c} \\
\beta_\mathrm{c}^2/2 + (\beta - \beta_\mathrm{c})\beta_\mathrm{c} &\beta>\beta_\mathrm{c}.
\end{cases}
}
See also \cite[Chapter 1]{talagrand03book}, in particular Theorem 1.2.1.
\item The generalized random energy models have non-trivial covariance structure \cite{derrida85}, and can be tuned to have an arbitrary number of phase transitions.
The condition $p'(\beta) < \beta$ is satisfied as soon as the first phase transition occurs.
See also \cite[Chapter 10]{bovier06}.

\item Finally, in \cite{arguin-zindy14} Arguin and Zindy studied a discretization of a log-correlated Gaussian field from \cite{barral-mandelbrot02,bacry-muzy03} 
which has the same free energy as the REM. 
Their particular model had the
technical complication of correlations not following a
tree structure, unlike for instance the discrete Gaussian free field.
\end{itemize}
\subsection{Open problems}
There are a number of open questions which, if solved, would enhance the theory presented in this paper. A partial list is the following.
\begin{enumerate}
\item Understand conditions under which the number of localizing regions is exactly one. As mentioned before, this requires more conditions than \eqref{free_energy_assumption}--\eqref{field_decomposition}, because it does not hold for some models (such as REM), whereas it is supposed to hold for many others. 
\item A close cousin of the above problem is to understand conditions under which $\RR_{1,2}$ is itself guaranteed to be away from zero with high probability. This would have important implications about the FRSB picture in mean-field spin glasses and path localization in directed polymers. 
\item Obtain a good quantitative bound on $\delta$ in terms of $\eps$ in Theorem~\ref{expected_overlap_thm}. Our proof gives a very poor bound, since it is based on an iterative argument similar to those used in extremal combinatorics (see the proof sketch in Section~\ref{sketchsec2}).
\item For directed polymers, prove a stronger theorem about path localization that says a typical path localizes within a narrow neighborhood of one or more fixed paths, rather than saying that a typical path has nonzero intersection with one or more fixed paths. 
\item Prove more general versions of Theorems~\ref{easy_cor}, ~\ref{expected_overlap_thm} and~\ref{averages_squared} that do not require the condition \eqref{positive_overlap} guaranteeing asymptotically nonnegative correlations. This would allow the theory to include other models of interest, such as the Edwards--Anderson model \cite{edwards-anderson75} of lattice spin glasses. 
It is important to note, however, that the hypotheses and conclusions of these more general theorems may require adjustment in order to be physically meaningful.
\item For any finite $\beta$, prove estimates that stochastically bound $\langle \RR_{1,2}\rangle$ away from $1$.
More ambitiously, determine conditions which guarantee that $\langle \RR_{1,2}\rangle$ concentrates around its expectation as $n\to\infty$.
\item Even when the spin glass correlation function $\xi$ takes negative values (recall that $\xi(R_{1,2})=\RR_{1,2}$), it is possible for the Gibbs measure to concentrate on a set such that $R_{1,2}\geq0$.
This is Talagrand's positivity principle and is known to hold when the extended Ghirlanda--Guerra identities are satisfied; see \cite[Section 12.3]{talagrand11II} or \cite[Section 3.3]{panchenko13}.
Perhaps the methods of this paper can be adapted to use this input rather than the condition $\xi\geq0$.
\end{enumerate}

\section{Proof sketches}\label{sketchsec}
The proofs of Theorems~\ref{expected_overlap_thm} and~\ref{averages_squared} are long, but they contain ideas that may be useful for other problems. 
Therefore, we have included this proof-sketch section which, while still rather lengthy, distills the arguments to their central ideas.
It introduces some of the notations that will be used later in the manuscript; however, these notations will be reintroduced in the later sections, so it is safe to skip directly to Section~\ref{prep_section} should the reader decide to do so.

\subsection{Proof sketch of Theorem~\ref{averages_squared}}\label{sketchsec1}
For simplicity, let us assume that the representation \eqref{field_decomposition} consists of only finitely many terms:
\eq{
H_n(\sigma) = \sum_{i=1}^N g_i \vphi_i(\sigma).
}
Following the argument described below, the general case is handled by some routine calculations (made in Section~\ref{gibbs_measure}) to check that sending $N\to\infty$ poses no issues.

Given \eqref{expected_overlap_formula}, it is clear that $p'(\beta) < \beta$ would imply \eqref{averages_squared_eq} if we knew that $\langle\RR_{1,2}\rangle$ concentrates around its mean as $n\to\infty$.
Unfortunately, this may not be true in general.
Therefore, as a way of artificially imposing concentration, we let the environment evolve as an Ornstein--Uhlenbeck (OU) flow, and then eventually take an average over a short time interval.
Formally, this means we consider
\eeq{ \label{OU_def_sketch}
\vc g_t \coloneqq \e^{-t}\vc g + \e^{-t} \vc W({\e^{2t}-1}), \quad t\geq0,
}
where $\vc W(\cdot) = (W_i(\cdot))_{i=1}^N$ are independent Brownian motions that are also independent of $\vc g = \vc g_0$.
Recall the OU generator $\LL \coloneqq \Delta - \vc x \cdot \nabla$, and the fact that $\E\LL f(\vc g) =0$ for any $f$ with suitable regularity.
By expanding $f$ in an  orthonormal basis of eigenfunctions of $\LL$, and expressing both $\LL f(\vc g_t)$ and $\E \|\nabla f(\vc g)\|^2$ using the coefficients from this expansion, one can show that
\eeq{ \label{general_variance_bound}
\Var\bigg(\frac{1}{t} \int_0^{t} \LL f(\vc g_s)\ \dd s\bigg) \le \frac{2}{t} \E\|\nabla f(\vc g)\|^2.  
}
This inequality, established in Lemma~\ref{OU_variance_lemma}, provides the proof's essential estimate when applied to $f(\vc g) = F_n(\beta)$.
For this $f$, it is easy to verify that $\E\|\nabla f(\vc g)\|^2=O(1/n)$, and
\eq{
\LL f(\vc g_t) = \beta^2 - \beta^2 \langle \RR_{1,2}\rangle_t - \beta \frac{\partial}{\partial \beta} F_{n,t}(\beta),
}
where $\langle \RR_{1,2}\rangle_t$ and $F_{n,t}(\beta)$ are the expected overlap and free energy, respectively, in the environment $\vc g_t$.
Moreover, from standard methods (worked out in Section~\ref{derivatives_section}), it follows that $\frac{\partial}{\partial\beta}F_{n,t}(\beta)\approx p'(\beta)$ with high probability. 
Combining these observations about $f$ with the general variance estimate \eqref{general_variance_bound}, we arrive at
\eeq{\label{sketcheq2}
\frac{1}{T/n}\int_0^{T/n} \langle \RR_{1,2}\rangle_t \ \dd t = 1- \frac{p'(\beta)}{\beta} + O(1/T). 
}
In other words, averaging $\langle \RR_{1,2}\rangle_t$ over a long enough interval, but whose size is still $O(1/n)$, results in a value close to the expectation suggested by \eqref{expected_overlap_formula}.
We choose $T=T(\eps)$ large enough depending on $\eps$, which determines the level of precision required in \eqref{sketcheq2}.

Next comes the most crucial step in the proof, where we show that if $\langle \RR_{1,2}\rangle=\langle \RR_{1,2}\rangle_0\le \delta$ for some small $\delta$, then for each $t\in[0,T(\eps)/n]$, the quantity $\langle \RR_{1,2}\rangle_t$ is also small with high probability. 
If $p'(\beta)<\beta$, this leads to a contradiction to \eqref{sketcheq2} if $\delta$ is small enough. 
To avoid this contradiction, the probability of $\langle \RR_{1,2}\rangle\le \delta$ happening in the first place must be small, which is what we want to show.

To demonstrate our crucial claim, we consider any $t = T/n$, where $T\leq T(\eps)$ and $n$ is large.
First, note that
\eeq{\label{rrab}
\langle \RR_{1,2}\rangle_t = \frac{\langle \RR_{1,2}\e^{\beta A_t+ \beta B_t}\rangle}{\langle \e^{\beta A_t+\beta B_t}\rangle},
}
where $B_t$ comes from the Brownian part of \eqref{OU_def_sketch}, and $A_t$ comes from the initial environment:
\eq{
A_t &\coloneqq (\e^{-t} - 1)(H_n(\sigma^1) + H_n(\sigma^2)), \\
B_t &\coloneqq  \e^{-t}\sum_{i} W_i(\e^{2t}-1)(\vphi_i(\sigma^1)+\vphi_i(\sigma^2)).
}
Since $t=T/n \ll 1$, we have
\[
A_t \approx -\frac{T}{n} (H_n(\sigma^1) + H_n(\sigma^2)).
\]
By standard arguments (again presented in Section~\ref{derivatives_section}), $H_n(\sigma^1)/n$ and $H_n(\sigma^2)/n$ are both close to $p'(\beta)$ with high probability under the Gibbs measure. 
Thus, for fixed $t$, the random variable $A_t$ behaves like a constant inside $\langle\cdot\rangle$. 
Consequently, we can reduce \eqref{rrab} to 
\eeq{ \label{gibbs_t_1}
\langle \RR_{1,2}\rangle_t \approx \frac{\langle \RR_{1,2}\e^{\beta B_t}\rangle}{\langle \e^{\beta B_t}\rangle}.
}
Now let $h_i := W_i(\e^{2t}-1)/\sqrt{\e^{2t}-1}$, so that $h_i\sim \NN(0,1)$.  
Again since $t=T/n \ll 1$, we have
\eq{
B_t &= \sqrt{1-\e^{-2t}}\sum_i h_i (\vphi_i(\sigma^1)+\vphi_i(\sigma^2)) \approx \sqrt{\frac{2T}{n}}\sum_i h_i (\vphi_i(\sigma^1)+\vphi_i(\sigma^2)).
}
Thus, if $\E_{\vc h}$ denotes expectation in $\vc h = (h_1,\ldots, h_N)$ only, then 
\eq{
\E_{\vc h} \langle \e^{\beta B_t}\rangle 
&\stackrel{\phantom{\mbox{\footnotesize\eqref{variance_assumption}}}}{\approx} 
 \Big\langle\exp\Big(\frac{\beta^2T}{n}\sum_i (\vphi_i(\sigma^1)+\vphi_i(\sigma^2))^2\Big)\Big\rangle \\
&\stackrel{\mbox{\footnotesize\eqref{variance_assumption}}}{=} \exp\big(2\beta^2 T(1+\RR_{1,2})\big). 
}
In the event that $\langle \RR_{1,2}\rangle$ is small, the 
assumption \eqref{positive_overlap} implies that $\RR_{1,2}\approx 0$ with high probability under the Gibbs measure. 
Therefore, conditional on this event (which depends only on $\vc g$, not $\vc h$), we have
\[
\E_{\vc h} \langle \e^{\beta B_t}\rangle\approx \e^{2\beta^2T}.
\]
By a similar argument, we also have
\eq{
&\E_{\vc h} \langle \e^{\beta B_t}\rangle^2 \approx \E_{\vc h} \Big\langle \exp\Big(\beta\sqrt{\frac{2T}{n}} \sum_i h_i (\vphi_i(\sigma^1)+\vphi_i(\sigma^2)+ \vphi_i(\sigma^3)+\vphi_i(\sigma^4))\Big)\Big\rangle\\
&\qquad = \Big\langle \exp\Big(\frac{\beta^2T}{n} \sum_i (\vphi_i(\sigma^1)+\vphi_i(\sigma^2)+ \vphi_i(\sigma^3)+\vphi_i(\sigma^4))^2\Big)\Big\rangle
\approx \e^{4\beta^2T}.
}
In summary, if $\langle \RR_{1,2}\rangle \approx 0$, then
\[
\Var_{\vc h} \langle \e^{\beta B_t}\rangle = \E_{\vc h} \langle \e^{\beta B_t}\rangle^2 - (\E_{\vc h} \langle \e^{\beta B_t}\rangle)^2 \approx 0,
\]
and thus, with high probability, 
\eeq{ \label{gibbs_t_2}
 \langle \e^{\beta B_t}\rangle\approx \E_{\vc h} \langle \e^{\beta B_t}\rangle \approx \e^{2\beta^2T}.
}
By following exactly the same steps with $\langle \RR_{1,2} \e^{\beta B_t}\rangle$ instead of $\langle \e^{\beta B_t}\rangle$, we show that 
\eeq{ \label{gibbs_t_3}
\langle \RR_{1,2} \e^{\beta B_t}\rangle \approx \langle \RR_{1,2} \rangle \e^{2\beta^2T}.
}
Combining \eqref{gibbs_t_1}--\eqref{gibbs_t_3}, we conclude that if $\langle \RR_{1,2}\rangle \approx 0$, then $\langle \RR_{1,2}\rangle_t \approx \langle \RR_{1,2}\rangle \approx 0$.

\subsection{Proof sketch of Theorem~\ref{expected_overlap_thm}}\label{sketchsec2}
We begin this proof sketch where the previous section left off, namely the observation that if the average overlap $\langle\RR_{1,2}\rangle$ in environment $\vc g$ is small, then Gibbs averages of the type in \eqref{gibbs_t_2} and \eqref{gibbs_t_3} are well concentrated.
By the same type of argument --- see Lemma~\ref{h_variance_lemma}(b) and \eqref{upper_X3} --- we can say something more general: no matter the size of $\langle\RR_{1,2}\rangle$, these averages remain concentrated so as long as they are restricted to the set $\AA_{n,\delta}$ defined in \eqref{A_def}, where \textit{conditional} average overlap $\givena{\RR_{1,2} }{\sigma^1}$ is small.
That is, if $\wt H_n$ is an independent Hamiltonian (i.e.~defined with $\vc h$, an independent copy of $\vc g$), then with high probability,
\eeq{ \label{small_fluctuations_restricted}
\langle\one_{\AA_{n,\delta}}\e^{\frac{\beta}{\sqrt{n}}\wt H_n(\sigma)}\rangle \approx \E_{\vc h}  \langle\one_{\AA_{n,\delta}}\e^{\frac{\beta}{\sqrt{n}}\wt H_n(\sigma)}\rangle\stackref{variance_assumption}{=} \e^{\frac{\beta^2}{2}}\langle \one_{\AA_{n,\delta}}\rangle.
}
In fact, the opposite is true off of the set $\AA_{n,\delta}$. 
If $\langle \RR_{1,2}\rangle$ is not too small relative to $\delta$, then the fluctuations of $\langle\one_{\AA_{n,\delta}^\cc}\e^{\frac{\beta}{\sqrt{n}}\wt H_n(\sigma)}\rangle$ due to $\vc h$ are $\Omega(1)$ as $n\to\infty$.
This is again an elementary calculation; see \eqref{lower_X4_prep}--\eqref{lower_X4_almost}.

On the other hand, a convenient consequence of Gaussianity is that $H_n + \frac{1}{\sqrt{n}}\wt H_n \stackrel{\dd}{=} \sqrt{1+\frac{1}{n}}H_n$.
That is, an environment perturbation is equivalent in distribution to a temperature perturbation.
(In fact, this simple observation underlies the Aizenman--Contucci identities \cite{aizenman-contucci98}, the predecessor of the Ghirlanda--Guerra identities.)
Therefore, if we keep track of the dependence on $\beta$ by writing $\langle \cdot\rangle_\beta$, and abbreviate $\AA_{n,\delta}$ to $\AA_{\delta}$, we have
\eeq{ \label{sketch_2_1}
\langle\one_{\AA_{\delta}}\rangle_{\beta\sqrt{1+\frac{1}{n}}} 
\stackrel{\dd}{=}\frac{\langle\one_{\AA_{\delta}}\e^{\frac{\beta}{\sqrt{n}}\wt H_n(\sigma)}\rangle_\beta}{\langle \e^{\frac{\beta}{\sqrt{n}}\wt H_n(\sigma)}\rangle_\beta}.
}
By rewriting the denominator in a trivial way and using our observation \eqref{small_fluctuations_restricted}, we see that with high probability, 
\eeq{ \label{sketch_2_2}
\frac{\langle\one_{\AA_{\delta}}\e^{\frac{\beta}{\sqrt{n}}\wt H_n(\sigma)}\rangle_\beta}{\langle \e^{\frac{\beta}{\sqrt{n}}\wt H_n(\sigma)}\rangle_\beta} 
&= \frac{\langle\one_{\AA_{\delta}}\e^{\frac{\beta}{\sqrt{n}}\wt H_n(\sigma)}\rangle_\beta}{\langle\one_{\AA_{\delta}}\e^{\frac{\beta}{\sqrt{n}}\wt H_n(\sigma)}\rangle_\beta+\langle\one_{\AA^\cc_{\delta}}\e^{\frac{\beta}{\sqrt{n}}\wt H_n(\sigma)}\rangle_\beta}  \\
&\approx \frac{\e^{\frac{\beta^2}{2}}\langle\one_{\AA_{\delta}}\rangle_\beta}{\e^{\frac{\beta^2}{2}}\langle\one_{\AA_{\delta}}\rangle_\beta+\langle\one_{\AA^\cc_{\delta}}\e^{\frac{\beta}{\sqrt{n}}\wt H_n(\sigma)}\rangle_\beta}.
}
In the last expression above, the only term depending on $\vc h$ is the second summand in the denominator.
Therefore, Jensen's inequality gives
\eeq{ \label{sketch_2_3}
\E_\vc h\bigg[&\frac{\e^{\frac{\beta^2}{2}}\langle\one_{\AA_{\delta}}\rangle_\beta}{\e^{\frac{\beta^2}{2}}\langle\one_{\AA_{\delta}}\rangle_\beta+\langle\one_{\AA^\cc_{\delta}}\e^{\frac{\beta}{\sqrt{n}}\wt H_n(\sigma)}\rangle_\beta}\bigg] \\
&> \frac{\e^{\frac{\beta^2}{2}}\langle\one_{\AA_{\delta}}\rangle_\beta}{\e^{\frac{\beta^2}{2}}\langle\one_{\AA_{\delta}}\rangle_\beta+\E_{\vc h}\langle\one_{\AA^\cc_{\delta}}\e^{\frac{\beta}{\sqrt{n}}\wt H_n(\sigma)}\rangle_\beta} \\
&=\frac{\e^{\frac{\beta^2}{2}}\langle\one_{\AA_{\delta}}\rangle_\beta}{\e^{\frac{\beta^2}{2}}\langle\one_{\AA_{\delta}}\rangle_\beta+\e^{\frac{\beta^2}{2}}\langle\one_{\AA^\cc_{\delta}}\rangle_\beta} 
= \langle \one_{\AA_{\delta}}\rangle_\beta.
}
A more careful analysis shows that the Jensen gap is large enough that we can replace the lower bound by $(1+\gamma)\langle\one_{\AA_{\delta}}\rangle_\beta - C\sqrt{\delta}$, where $\gamma$ and $C$ are positive constants.
One important caveat is that this stronger lower bound is valid only when $\langle\RR_{1,2}\rangle$ is not too small (so that the fluctuations of $\langle\one_{\AA^\cc_{\delta}}\e^{\frac{\beta}{\sqrt{n}}\wt H_n(\sigma)}\rangle_\beta$ are order $1$), which is why Theorem~\ref{averages_squared} is needed beforehand.
Reading \eqref{sketch_2_1}--\eqref{sketch_2_3} from start to end, we obtain
\eeq{ \label{sketch_2_4}
\E\langle\one_{\AA_\delta}\rangle_{\beta\sqrt{1+\frac{1}{n}}} \geq (1+\gamma)\E\langle\one_{\AA_\delta}\rangle_\beta - C\sqrt{\delta}.
}
While the above inequality is the most important step of the proof, a key shortcoming is that the set $\AA_\delta$ is defined using $\langle\cdot\rangle_\beta$ rather than $\langle\cdot\rangle_{\beta\sqrt{1+\frac{1}{n}}}$.
Since we will want to apply the inequality iteratively, we need to replace $\AA_\delta$ on the left-hand side by $\AA_{\delta,1}$, where
\eq{
\AA_{\delta,k} \coloneqq \Big\{\sigma\in\Sigma_n : \frac{1}{n}\sum_i\vphi_i(\sigma)\langle\vphi_i\rangle_{\beta\sqrt{1+\frac{k}{n}}}\leq\delta\Big\}, \quad k = 0,1,2,\dots
}
To make this replacement, we produce a complementary inequality, again using the equivalence of environment/temperature perturbations.
For simplicity, let us assume $\RR_{1,2}\geq0$, which is essentially realized by \eqref{positive_overlap} for large $n$.
Observe that
\eq{
\givena{\RR_{1,2}}{\sigma^1}_{\beta\sqrt{1+\frac{1}{n}}} 
&\stackrel{\dd}{=} \frac{\givena{\RR_{1,2}\e^{\frac{\beta}{\sqrt{n}}\wt H_n(\sigma^2)}}{\sigma^1}_\beta}{\langle\e^{\frac{\beta}{\sqrt{n}}\wt H_n(\sigma)}\rangle_\beta}  \\
&\leq \sqrt{\givena{\RR_{1,2}}{\sigma^1}_\beta}\underbrace{\sqrt{\langle \e^{\frac{2\beta}{\sqrt{n}}\wt H_n(\sigma)}\rangle_\beta}\langle\e^{\frac{-\beta}{\sqrt{n}}\wt H_n(\sigma)}\rangle_\beta}_{X},
}
where we have applied Cauchy--Schwarz (and then $\RR_{1,2}^2 \leq \RR_{1,2}\leq1$) and Jensen's inequality (using the convexity of $x\mapsto x^{-1}$).
When $\sigma^1 \in \AA_{\delta}=\AA_{\delta,0}$, the final expression is at most $X\sqrt{\delta}$, and so the inequality implies $\AA_{\delta,0} \subset \AA_{X\sqrt{\delta},1}$.
Now, the random variable $X$ has moments of all orders (admitting simple upper bounds), and so it can be essentially regarded as a large constant.
In particular, when $\delta$ is small, we will have $X \leq \delta^{-1/4}$ with high probability, in which case
$\AA_{\delta,0} \subset \AA_{\delta^{1/4},1}$.
Combining these ideas with \eqref{sketch_2_4}, we show
\eq{
\E\langle \one_{\AA_{\delta^{1/4},1}}\rangle_{\beta\sqrt{1+\frac{1}{n}}} \geq (1+\gamma)\E\langle\one_{\AA_\delta}\rangle_\beta - C\sqrt{\delta}.
}
More generally, for any integer $k\geq1$,
\eeq{ \label{ineq_to_be_iterated}
\E\langle \one_{\AA_{\delta^{1/4},k}}\rangle_{\beta\sqrt{1+\frac{k}{n}}} \geq (1+\gamma)\E\langle\one_{\AA_{\delta,k-1}}\rangle_{\beta\sqrt{1+\frac{k-1}{n}}} - C\sqrt{\delta}.
}
This inequality can now be iterated, with $\delta$ being replaced by $\delta^{1/4}$, then $\delta^{1/16}$, and so on, as the expectation on the left is inserted on the right in the next iteration.

Since the left-hand side of \eqref{ineq_to_be_iterated}  is always at most $1$, we clearly obtain a contradiction if $\E\langle\one_{\AA_{\delta,0}}\rangle_\beta$ is larger than $x$, where $x$ is the solution to $x = (1+\gamma)x - C\sqrt{\delta}$. 
This would complete the proof of Theorem~\ref{expected_overlap_thm} if not for the subtlety that $\gamma$ actually depends on $k$ in a non-trivial way.
Nevertheless, \eqref{ineq_to_be_iterated} can still be used to derive a contradiction of the same spirit unless 
$\E\langle\one_{\AA_{\delta^{1/4^k},k}}\rangle$ is small for some $k\leq K$, where $K$ is large and tends to infinity as $\eps\to0$, but crucially does not depend on $n$.
This approach is reminiscent of tower-type arguments in extremal combinatorics.

Replacing $\delta$ by $\delta^{4^k}$, we can then say $\E\langle\one_{\AA_{\delta,k}}\rangle$ is small.
Finally, to deduce the smallness of $\E\langle \one_{\AA_{\delta,0}}\rangle$ from the smallness of $\E\langle\one_{\AA_{\delta,k}}\rangle$, we make use of standard arguments showing that if an event is rare at inverse temperature $\beta$, then it remains rare at inverse temperature $\beta + O(1/n)$.

\subsection{Proof sketch of Theorem~\ref{easy_cor}}
To deduce Theorem~\ref{easy_cor} from Theorem~\ref{expected_overlap_thm}, simply let $\sigma^1,\ldots,\sigma^k, \sigma^{k+1}$ be i.i.d.~draws from the Gibbs measure. Then by the law of large numbers, when $k$ is large,
\[
\frac{1}{k}\sum_{j=1}^k \RR_{j, k+1} \approx \RR(\sigma^{k+1})
\]
with high probability. 
But by Theorem~\ref{expected_overlap_thm}, we know that with high probability, $\RR(\sigma^{k+1})$ is not close to zero. Therefore, with high probability, there must exist $1\le j\le k$ such that $\RR_{j,k+1}$ is not close to zero.

\section{General preliminaries} \label{prep_section}
In this preliminary section, we record several facts needed in the proofs of Theorems~\ref{expected_overlap_thm} and~\ref{averages_squared}.
These preparatory results are mostly elementary.


\subsection{The Gibbs measure and partition function} \label{gibbs_measure}
 In order for our results to apply to a broad collection of models, we have allowed the state
space $\Sigma_n$ to be completely general, and the Hamiltonian $H_n$ to consist of countably infinite summands. 
We begin by checking that these assumptions pose no issues to computation. 
So for the remainder of Section~\ref{gibbs_measure}, we fix the value of $n$. 
 
Let $\langle\cdot\rangle_N$ denote expectation with respect to the Gibbs measure when the Hamiltonian is replaced by the finite sum $H_{n,N}\coloneqq\sum_{i=1}^Ng_i\vphi_i$.
That is,
\eeq{ \label{finite_gibbs_expectation}
\langle f(\sigma)\rangle_N = \frac{E_n(f(\sigma)\e^{\beta H_{n,N}(\sigma)})}{E_n(\e^{\beta H_{n,N}(\sigma)})}.
}
So that we can pass from $\langle \cdot\rangle_N$ to $\langle\cdot\rangle$,  we begin with the following lemma. 

\begin{lemma} \label{exp_moments_lemma}
For all $\beta\in\R$ and any $f\in L^2(\Sigma_n)$, the following limits hold almost surely and in $L^\alpha$ for any $\alpha\in[1,\infty)$:
\begin{subequations}
\begin{align} 
\lim_{N\to\infty} \langle f(\sigma)\rangle_N &= \langle f(\sigma)\rangle < \infty, \label{exp_moments_lemma_a}\\
\lim_{N\to\infty} \langle H_{n,N}(\sigma)\rangle_N &= \langle H_n(\sigma)\rangle < \infty. \label{exp_moments_lemma_b}
\end{align}
\end{subequations}
\end{lemma}


\begin{proof}
We organize the proof into a sequence of claims.
\begin{claim} \label{claim_1}
With $\P$-probability equal to $1$,
\eq{
\lim_{N\to\infty} H_{n,N}(\sigma) = H_n(\sigma) \quad \text{for $P_n$-$\mathrm{a.e.}$~$\sigma\in\Sigma_n$}.
}
\end{claim}

\begin{proof} Observe that for fixed $\sigma\in\Sigma_n$, the
sequence $(H_{n,N} (\sigma))_{N\geq0}$ is a martingale with respect to $\P$. 
Since
\eq{
\sup_{N\geq0}\E[H_{n,N}(\sigma)^2] = \sup_{N\geq0}\sum_{i=1}^N \vphi_i(\sigma)^2\stackrel{\mbox{\footnotesize\eqref{variance_assumption},\eqref{field_decomposition}}}{=} n,
}
the martingale convergence theorem guarantees that $H_{n,N}(\sigma)$ converges $\P$-almost surely as $N\to\infty$ to a limit we call $H_n(\sigma)$.
Now Fubini's theorem proves the claim:
\eq{
\E E_n(\one_{\{H_{n,N}(\sigma)\to H_n(\sigma)\}}) = E_n(\E[\one_{\{H_{n,N}(\sigma)\to H_n(\sigma)\}}]) = E_n(1) = 1.
}
\end{proof}

\begin{claim} \label{claim_2}
There exist nonnegative random variables $(M^+(\sigma))_{\sigma\in\Sigma_n}$ and $(M^-(\sigma))_{\sigma\in\Sigma_n}$ such that
\eeq{ \label{every_N}
\pm H_{n,N}(\sigma) \leq M^\pm (\sigma) \quad \text{for all $N\geq0,\, \sigma\in\Sigma_n$},
}
and
\eeq{ \label{good_denominator}
\E E_n(\e^{\beta M^\pm(\sigma)}) < \infty \quad \text{for all $\beta\geq0$}.
}
\end{claim}

\begin{proof}
We simply take
\eq{
M^\pm(\sigma) \coloneqq \sup_{N\geq0} \pm H_{n,N}(\sigma)\geq \pm H_{n,0}(\sigma) = 0,
}
so that \eqref{every_N} is satisfied by definition.
Since $M^+ \stackrel{\text{d}}{=} M^-$, we need only check \eqref{good_denominator} for $M^+$. 
Observe that for any $\beta \geq 0$, $(\e^{\beta H_{n,N}(\sigma)})_{N\geq0}$ is a submartingale.
By Doob's inequality, for any $\lambda > 0$ and any integer $m\geq0$,
\eq{
\P\Big(\max_{0\leq N\leq m} \e^{\beta H_{n,N}(\sigma)}\geq\lambda\Big)
&= \P\Big(\max_{0\leq N\leq m} \e^{2\beta H_{n,N}(\sigma)}\geq\lambda^2\Big) \\
&\leq \lambda^{-2}\E(\e^{2\beta H_{n,m}(\sigma)}) \\
&= \lambda^{-2}\e^{2\beta^2\sum_{i=1}^m\vphi_i^2(\sigma)}\stackrel{\mbox{\footnotesize\eqref{variance_assumption}}}{\leq} \lambda^{-2}\e^{2\beta^2n}.
}
Therefore, for any $0 < \eps<\lambda $, 
\eq{
\P(\e^{\beta M^+(\sigma)}\geq\lambda) &\leq \P\Big(\e^{\beta M^+(\sigma)}\geq\lambda-\frac{\eps}{2}\Big) \\
&\leq \lim_{m\to\infty} \P\Big(\max_{0\leq N\leq m} \e^{\beta H_{n,N}(\sigma)}\geq\lambda-\eps\Big) \leq (\lambda-\eps)^{-2}\e^{2\beta^2n},
}
which implies
\eq{ 
\E(\e^{\beta M^+(\sigma)}) &= \int_0^\infty \P(\e^{\beta M^+(\sigma)}\geq\lambda)\ \dd\lambda \\
&\leq 1+\eps + \e^{2\beta^2n}\int_{1+\eps}^\infty (\lambda-\eps)^{-2}\ \dd\lambda < \infty.
}
Since Tonelli's theorem gives $\E E_n(\e^{\beta M^+(\sigma)}) = E_n(\E \e^{\beta M^+(\sigma)})$, \eqref{good_denominator} follows from the above display.
\end{proof}

\begin{claim} \label{claim_3}
For any $f\in L^2(\Sigma_n)$ and any continuous function $\phi : \R \to \R$ such that $|\phi(x)| \leq a\e^{b|x|}$ for all $x\in\R$, for some $a,b\geq0$, we have
\eeq{ \label{claim_3_eq}
\lim_{N\to\infty} E_n[f(\sigma)\phi(H_{n,N}(\sigma))] = E_n[f(\sigma)\phi(H_n(\sigma))] \quad \mathrm{a.s.}
}
\end{claim}

\begin{proof}
By Claim~\ref{claim_1} and the continuity of $\phi$, we almost surely have that $\phi(H_{n,N}(\sigma)) \to \phi(H_{n}(\sigma))$ for $P_n$-a.e.~$\sigma\in\Sigma_n$, as $N\to\infty$.
And by hypothesis,
\eeq{ \label{exponential_domination}
|\phi(H_{n,N}(\sigma))| \leq a(\e^{bM^+(\sigma)} + \e^{bM^-(\sigma)}).
}
Since
\eq{
E_n\big[|f(\sigma)|(\e^{b M^+(\sigma)} + \e^{bM^-(\sigma)})\big] 
&\leq \sqrt{ E_n[f(\sigma)^2] E_n[(\e^{b M^+(\sigma)} + \e^{bM^-(\sigma)})^2]} \\
&\leq \sqrt{ E_n[f(\sigma)^2] E_n[2(\e^{2b M^+(\sigma)} + \e^{2bM^-(\sigma)})]},
}
and Claim~\ref{claim_2} implies that almost surely $E_n(\e^{2b M^\pm(\sigma)}) < \infty$,
\eqref{claim_3_eq} now follows from dominated convergence (with respect to $P_n$).
\end{proof}

\begin{claim} \label{claim_4}
For any $f\in L^2(\Sigma_n)$ and any continuous function $\phi : \R \to \R$ such that $|\phi(x)| \leq a\e^{b|x|}$ for all $x\in\R$, for some $a,b\geq0$, we have
\eeq{ \label{claim_4_eq}
\lim_{N\to\infty} \langle f(\sigma)\phi(H_{n,N}(\sigma))\rangle_N = \langle f(\sigma)\phi(H_n(\sigma))\rangle \quad \mathrm{a.s.}\text{ and in } L^\alpha, \alpha\in[1,\infty).
}
\end{claim}

\begin{proof}
Recall that
\eq{
\langle f(\sigma)\phi(H_{n,N}(\sigma))\rangle_N &= \frac{E_n[f(\sigma)\phi(H_{n,N}(\sigma))\e^{\beta H_{n,N}(\sigma)}]}{E_n(\e^{\beta H_{n,N}(\sigma)})}, \\
\langle f(\sigma)\phi(H_n(\sigma))\rangle &= \frac{E_n[f(\sigma)\phi(H_{n}(\sigma))\e^{\beta H_{n}(\sigma)}]}{E_n(\e^{\beta H_{n}(\sigma)})}.
}
Since $|\phi(x)|\e^{\beta x} \leq  a\e^{(b+\beta)|x|}$, the almost sure part of \eqref{claim_4_eq} is immediate from Claim~\ref{claim_3}.
The convergence in $L^\alpha$ is then a consequence of dominated convergence (with respect to $\P$).
Indeed, by Cauchy--Schwarz and Jensen's inequality, we have the majorization 
\eq{
&|\langle f(\sigma)\phi(H_{n,N}(\sigma))\rangle_N|
= \frac{|E_n(f(\sigma)\phi(H_{n,N}(\sigma))\e^{\beta H_{n,N}(\sigma)})|}{E_n(\e^{\beta H_{n,N}(\sigma)})} \\
&\stackrel{\phantom{\mbox{\footnotesize\eqref{exponential_domination}}}}{\leq} \frac{\sqrt{E_n(f(\sigma)^2)E_n(\phi(H_{n,N}(\sigma))^2\e^{2\beta H_{n,N}(\sigma)})}}{E_n(\e^{-\beta M^-(\sigma)})} \\
&\stackrel{\mbox{\footnotesize\eqref{exponential_domination}}}{\leq} \sqrt{ E_n(f(\sigma)^2)  E_n[2a^2(\e^{2(b+\beta)M^+(\sigma)} + \e^{2(b+\beta)M^-(\sigma)})]}E_n(\e^{\beta M^-(\sigma)}),
}
where the final expression has moments of all orders by \eqref{good_denominator}.
\end{proof}

We now complete the proof of Lemma~\ref{exp_moments_lemma} by taking $\phi \equiv 1$ for \eqref{exp_moments_lemma_a}, and $f\equiv1$, $\phi(x) = x$ for \eqref{exp_moments_lemma_b}.

\end{proof}

\begin{remark} The essential feature of the above proof was checking in Claim~\ref{claim_2} that \eqref{variance_assumption} is enough
to guarantee the first equality below:
\eeq{ \label{freq_identity}
\E(\e^{\beta\sum_{i=1}^\infty g_i\vphi_i}) = \lim_{N\to\infty} \E(\e^{\beta\sum_{i=1}^Ng_i\vphi_i})
= \lim_{N\to\infty} \e^{\frac{\beta^2}{2}\sum_{i=1}^N\vphi_i^2}
\stackrel{\mbox{\footnotesize\eqref{variance_assumption}}}{=} \e^{\frac{\beta^2}{2}n}.
}
\end{remark}
We will frequently use the above identity, an easy consequence of which is
the following.

\begin{lemma} \label{moments_lemma}
For any $\beta\in\R$, we have
\eeq{ \label{first_moment}
\E Z_n(\beta) = \e^{\frac{\beta^2}{2}n},
}
as well as
\eeq{ \label{negative_first_moment}
\E[Z_n(\beta)^{-1}] \leq \e^{\frac{\beta^2}{2}n}.
}
\end{lemma}

\begin{proof}
By exchanging the order of expectation in the identity $\E Z_n(\beta) = \E[E_n(\e^{\beta H_n(\sigma)})]$
(which we are permitted to do by Tonelli's theorem) and applying \eqref{freq_identity}, we
obtain \eqref{first_moment}. 
For \eqref{negative_first_moment}, we apply Jensen's inequality to obtain
\eq{
Z_n(\beta)^{-1} = [E_n(\e^{\beta H_n(\sigma)})]^{-1} \leq E_n(\e^{-\beta H_n(\sigma)}),
}
then take expectation $\E(\cdot)$ of both sides, and again exchange the order of
expectation. 
\end{proof}

Let us also record two consequences of Lemma~\ref{exp_moments_lemma} that will be needed later in the paper.

\begin{cor}
For any $\beta\in\R$, the following limits hold almost surely and in $L^\alpha$ for any $\alpha\in[1,\infty)$:
\eeq{ \label{limit_for_later}
\lim_{N\to\infty} \sum_{i=1}^N \langle \vphi_i^2\rangle_N &= n \qquad \text{and} \qquad
\lim_{N\to\infty} \sum_{i=1}^N \langle\vphi_i\rangle^2_N = \sum_{i=1}^\infty \langle\vphi_i\rangle^2.
}
\end{cor}

\begin{proof}
First we argue the almost sure statements.
The $L^\alpha$ statements will then follow from bounded convergence, since \eqref{variance_assumption} gives the uniform bound
\eq{
0\leq\sum_{i=1}^N \langle \vphi_i\rangle_N^2 \leq \sum_{i=1}^N \langle \vphi_i^2\rangle_N \leq n \quad \text{for every $N$.}
}
So we fix the disorder $\vc g$.
By Lemma~\ref{exp_moments_lemma}, it is almost surely the case that for every $i\geq1$, $\langle \vphi_i\rangle_N \to \langle \vphi_i\rangle$ and $\langle \vphi_i^2\rangle_N \to \langle \vphi_i^2\rangle$ as $N\to\infty$.
We also know $\sum_{i=1}^\infty \vphi_i^2 = n$.
In particular, given $\eps>0$, we can choose $M$ so large that
\eq{
n-\eps \leq \sum_{i=1}^M \langle \vphi_i^2\rangle \leq n \quad \Rightarrow \quad
\sum_{i=M+1}^\infty \langle \vphi_i^2\rangle \leq \eps.
}
Given $M$, there is $N_0$ such that for all $N\geq N_0$,
\eq{
\bigg|\sum_{i=1}^M (\langle \vphi_i^2\rangle_N -  \langle\vphi_i^2\rangle)\bigg| \leq \eps \qquad \text{and} \qquad
\bigg|\sum_{i=1}^M (\langle \vphi_i\rangle_N^2 -  \langle\vphi_i\rangle^2)\bigg| \leq \eps.
}
In particular, for all $N\geq N_0 \vee M$,
\eq{
n-2\eps\leq\sum_{i=1}^M \langle \vphi_i^2\rangle_N \leq n 
\quad &\Rightarrow \quad
n-2\eps\leq\sum_{i=1}^N \langle \vphi_i^2\rangle_N \leq n,
}
and also
\eq{
\bigg|\sum_{i=1}^N \langle \vphi_i\rangle_N^2 - \sum_{i=1}^\infty \langle\vphi_i\rangle^2\bigg|
&\leq \bigg|\sum_{i=1}^M (\langle \vphi_i\rangle_N^2 - \langle\vphi_i\rangle^2)\bigg| + \sum_{i=M+1}^\infty (\langle\vphi_i\rangle_N^2 + \langle\vphi_i\rangle^2) \\
&\leq \bigg|\sum_{i=1}^M (\langle \vphi_i\rangle_N^2 - \langle\vphi_i\rangle^2)\bigg| + \sum_{i=M+1}^\infty (\langle\vphi_i^2\rangle_N + \langle\vphi_i^2\rangle) 
\leq 4\eps.
}
\end{proof}

\subsection{Derivative of free energy} \label{derivatives_section}
This section records some important facts regarding convergence of the free energy's derivative.
%
By Lemma~\ref{exp_moments_lemma},
 it is almost surely the case that the random variable
$H_n(\sigma)$ has exponential moments of all orders with respect to $P_n$.
Standard
calculations then show that the free energy $F_n(\beta) = \frac{1}{n}\log Z_n(\beta)$ satisfies
\eeq{ \label{nrg_2deriv}
F_n'(\beta) = \frac{\langle H_n(\sigma)\rangle}{n} \quad \text{and} \quad
F_n''(\beta) = \frac{\langle H_n(\sigma)^2\rangle - \langle H_n(\sigma)\rangle^2}{n} \quad \mathrm{a.s.}
}
Recall from \eqref{free_energy_assumption} that $F_n(\beta) \to p(\beta)$. 
Since $F_n(\cdot)$ is convex for every $n$, $p(\cdot)$ is necessarily convex.
This assumption implies the following lemma, which is a general fact about the convergence of convex functions.

\begin{lemma} \label{betas_converging}
If $p(\cdot)$ is differentiable at $\beta$, and $\beta_n = \beta + \delta(n)$ with $\delta(n) \to 0$ as $n\to\infty$, then
\eq{
\lim_{n\to\infty} F_n'(\beta_n) = p'(\beta) \quad \mathrm{a.s.}\text{ and in }L^1.
}
\end{lemma}

\begin{proof}
Let $\eps > 0$.
By differentiability, we can choose $h > 0$ sufficiently small that
\eeq{ \label{eps_1}
p'(\beta) - \eps \leq \frac{p(\beta)-p(\beta-h)}{h} \leq \frac{p(\beta+h)-p(\beta)}{h} \leq p'(\beta) + \eps,
}
where the middle inequality is due to convexity.
Given $h$, we next choose $\delta > 0$ such that
\begin{subequations}
\begin{align}
0 \leq \frac{p(\beta+\delta+h) - p(\beta+\delta)}{h} - \frac{p(\beta+h) - p(\beta)}{h} \leq \eps \label{eps_2a},
\intertext{as well as} 
0 \leq \frac{p(\beta) - p(\beta-h)}{h} - \frac{p(\beta-\delta)-p(\beta-\delta-h)}{h} \leq \eps, \label{eps_2b}
\end{align}
\end{subequations}
which is possible by the continuity of $p(\cdot)$.
Now, convexity of $F_n$ implies the following for all $n$ such that $\delta(n) \leq \delta$:
\begin{subequations}
\begin{align}
F_n'(\beta_n) &\leq \frac{F_n(\beta+\delta(n)+h)-F_n(\beta+\delta(n))}{h}\notag \\
&\leq \frac{F_n(\beta+\delta+h)-F_n(\beta+\delta)}{h}. \label{eps_3a}
\intertext{Similarly, for all $n$ such that $\delta(n) \geq -\delta$,}
F_n'(\beta_n) &\geq\frac{F_n(\beta-\delta)-F_n(\beta-\delta-h)}{h}.\hspace{1.3in} \label{eps_3b} \raisetag{\baselineskip}
\end{align} 
\end{subequations}
Upon defining
\eeq{ \label{Delta_def}
\Delta_n^-(\beta,h) &\coloneqq \frac{F_n(\beta)-F_n(\beta-h)}{h} - \frac{p(\beta)-p(\beta-h)}{h}, \\
\Delta_n^+(\beta,h) &\coloneqq \frac{F_n(\beta+h)-F_n(\beta)}{h} - \frac{p(\beta+h)-p(\beta)}{h},
}
it follows that for all sufficiently large $n$, 
\eq{
F_n'(\beta_n) - p'(\beta)
&\stackrel{\mbox{\hspace{3ex}\footnotesize\eqref{eps_3a}\hspace{3ex}}}{\leq} \frac{F_n(\beta+\delta+h)-F_n(\beta+\delta)}{h} - p'(\beta) \\
&\stackrel{\mbox{\footnotesize\eqref{eps_1},\eqref{eps_2a}}}{\leq} \Delta_n^+(\beta+\delta,h) + 2\eps.
}
Analogously, \eqref{eps_1}, \eqref{eps_2b}, and \eqref{eps_3b} together yield the lower bound
\eq{
F_n'(\beta_n) - p'(\beta) \geq \Delta_n^-(\beta-\delta,h) - 2\eps.
}
By \eqref{free_energy_assumption},  both $\Delta_n^-(\beta-\delta,h)$ and $\Delta_n^+(\beta+\delta,h)$ tend to $0$ almost surely and in $L^1$ as $n\to\infty$.
As $\eps$ is arbitrary, the desired result follows.
\end{proof}

\begin{cor} \label{overlap_identity_cor}
For every $\beta\geq0$ at which $p(\cdot)$ is differentiable, 
\eeq{ \label{overlap_ito_derivative}
p'(\beta) = \beta\big(1 - \lim_{n\to\infty} \E \langle \RR_{1,2}\rangle\big).
}
In particular, $0 \leq p'(\beta) \leq \beta$, and there is thus some $\beta_\cc \in [0,\infty]$ such that
\eq{
0 \leq \beta \leq \beta_\cc \quad &\Rightarrow \quad p(\beta) = \frac{\beta^2}{2}, \\
\beta>\beta_\cc \quad &\Rightarrow \quad p(\beta) < \frac{\beta^2}{2}.
}
\end{cor}

\begin{proof}
Using the notation of Lemma~\ref{exp_moments_lemma}, we have
\eq{
\E F_n'(\beta) 
\stackrel{\mbox{\footnotesize\eqref{nrg_2deriv}}}{=} \frac{\E\langle H_n(\sigma)\rangle}{n} 
&\stackrel{\mbox{\footnotesize\eqref{exp_moments_lemma_b}}}{=}\lim_{N\to\infty} \frac{\E\langle H_{n,N}(\sigma)\rangle_N}{n} \\
&\stackrel{\phantom{\mbox{\footnotesize\eqref{exp_moments_lemma_b}}}}{=} \lim_{N\to\infty} \E\Big\langle\frac{1}{n} \sum_{i=1}^N g_i\vphi_i\Big\rangle_N \\
&\stackrel{\phantom{\mbox{\footnotesize\eqref{exp_moments_lemma_b}}}}{=} \lim_{N\to\infty}  \frac{1}{n}\sum_{i=1}^N \E[g_i\langle\vphi_i\rangle_N].
}
By Gaussian integration by parts,
\eq{
\E[g_i\langle\vphi_i\rangle_N] = \E\Big[\frac{\partial}{\partial g_i}\langle\vphi_i\rangle_N\Big] = \beta\E[\langle\vphi_i^2\rangle_N - \langle\vphi_i\rangle_N^2],
}
and then Lemma~\ref{betas_converging} allows us to write
\eq{
p'(\beta) = \lim_{n\to\infty} \E F_n'(\beta) 
&\stackrel{\phantom{\mbox{\footnotesize\eqref{limit_for_later}}}}{=} \lim_{n\to\infty} \lim_{N\to\infty} \beta\E\bigg[\frac{1}{n}\sum_{i=1}^N (\langle\vphi_i^2\rangle_N - \langle\vphi_i\rangle_N^2)\bigg] \\
&\stackrel{\mbox{\footnotesize\eqref{limit_for_later}}}{=} \lim_{n\to\infty} \beta\E\bigg[1-\frac{1}{n}\sum_{i=1}^\infty \langle\vphi_i\rangle^2\bigg] \\
&\stackrel{\phantom{\mbox{\footnotesize\eqref{limit_for_later}}}}{=}\lim_{n\to\infty} \beta(1 - \E\langle \RR_{1,2}\rangle),
}
which completes the proof of \eqref{overlap_ito_derivative}.
The inequalities $0\leq p'(\beta)\leq\beta$ now follow from
\eq{
1 \stackref{variance_assumption}{\geq} \lim_{n\to\infty}\E\langle\RR_{1,2}\rangle \stackref{positive_overlap}{\geq} -\lim_{n\to\infty}\EEE_n = 0.
}

For the second part of the claim, we recall that $p(\cdot)$ is convex and thus absolutely continuous.
Since $p(0) = 0$, we then have
\eq{
\frac{\beta^2}{2} - p(\beta) = \int_0^{\beta} [t - p'(t)]\ \dd t.
}
Since the integrand is nonnegative, it follows that $\beta\mapsto \beta^2/2 - p(\beta)$ is non-decreasing for $\beta \geq 0$.
\end{proof}
%

So that we can be explicit in the inverse
temperature parameter $\beta$, for the remainder of the section we will write $\langle\cdot\rangle_\beta$
for expectation with respect to $\mu_{n}^\beta$.
In light of \eqref{nrg_2deriv}, Lemma~\ref{betas_converging} implies
\eq{
\lim_{n\to\infty} \Big|\frac{\langle H_n(\sigma)\rangle_\beta}{n} - p'(\beta)\Big| = 0 \quad \mathrm{a.s.}\quad \text{whenever $p'(\beta)$ exists}.
}
We will require the following stronger form of this result, which also appears in 
\cite[Theorem 3]{auffinger-chen18}.
Our proof is adapted from the elegant approach of \cite{panchenko10}, and included for completeness.

\begin{lemma} \label{lemma:step_1_2}
If $\beta$ is a point of differentiability for $p(\cdot)$, then
\eq{
\lim_{n\to\infty} \Big\langle\Big|\frac{H_n(\sigma)}{n} - p'(\beta)\Big|\Big\rangle_\beta = 0 \quad \mathrm{a.s.} \text{ and in }L^1.
}
\end{lemma}

\begin{proof}
By Lemma~\ref{betas_converging}, it suffices to show that if $\beta_0$ is a point of differentiability for $p(\cdot)$, then
\eq{
\lim_{n\to\infty} \Big\langle\Big|\frac{H_n(\sigma)}{n} - F_n'(\beta_0)\Big|\Big\rangle_{\beta_0} = 0 \quad \mathrm{a.s.} \text{ and in }L^1.
}
Fix $\eps > 0$ and choose $h>0$ small enough that
\eeq{ \label{h_choice}
p'(\beta_0)-\eps \leq \frac{p(\beta_0)-p(\beta_0-h)}{h} \leq \frac{p(\beta_0+h)-p(\beta_0)}{h} \leq p'(\beta_0)+\eps.
}
Given $h$, differentiability allows us to take $\beta_1>\beta_0$ sufficiently close to $\beta_0$ to satisfy
\eeq{ \label{beta_prime_choice}
\frac{p(\beta_1+h)-p(\beta_1)}{h} \leq \frac{p(\beta_0+h)-p(\beta_0)}{h} + \eps \leq p'(\beta_0)+2\eps.
}
By adding and subtracting $\langle |H_n(\sigma^1)-H_n(\sigma^2)|\rangle_{\beta_0}$, we have
\eeq{ \label{trivial_integrals}
&\int_{\beta_0}^{\beta_1} \langle |H_n(\sigma^1)-H_n(\sigma^2)|\rangle_\beta\ \dd \beta \\
&= (\beta_1-{\beta_0})\langle |H_n(\sigma^1)-H_n(\sigma^2)|\rangle_{\beta_0} \\
&\phantom{=}+ \int_{\beta_0}^{\beta_1} \big[\langle|H_n(\sigma^1)-H_n(\sigma^2)|\rangle_\beta - \langle|H_n(\sigma^1)-H_n(\sigma^2)|\rangle_{{\beta_0}}\big]\ \dd \beta \\
&= (\beta_1-{\beta_0})\langle |H_n(\sigma^1)-H_n(\sigma^2)|\rangle_{\beta_0} \\
&\phantom{=}+ \int_{\beta_0}^{\beta_1} \int_{\beta_0}^\beta \frac{\partial}{\partial x} \langle |H_n(\sigma^1)-H_n(\sigma^2)|\rangle_x\ \dd x\, \dd \beta.
}
A simple calculation, followed by Cauchy--Schwarz, shows
\eq{ 
&\Big|\frac{\partial}{\partial x} \langle |H_n(\sigma^1)-H_n(\sigma^2)|\rangle_x\Big| \\
&= \big|\big\langle |H_n(\sigma^1)-H_n(\sigma^2)| \cdot \big(H_n(\sigma^1)+H_n(\sigma^2)-2H_n(\sigma^3)\big)\big\rangle_x\big| \\
&\leq \sqrt{\big\langle \big(H_n(\sigma^1)-H_n(\sigma^2)\big)^2\rangle_x\big\langle \big(H_n(\sigma^1)+H_n(\sigma^2)-2H_n(\sigma^3)\big)^2\big\rangle_x}.
}
By another application of Cauchy--Schwarz, we have
\eq{
&\big\langle \big(H_n(\sigma^1)+H_n(\sigma^2)-2H_n(\sigma^3)\big)^2\big\rangle_x \\
&= \big\langle \big(H_n(\sigma^1)-H_n(\sigma^3)+H_n(\sigma^2)-H_n(\sigma^3)\big)^2\big\rangle_x \\
&\leq 2\big\langle\big(H_n(\sigma^1)-H_n(\sigma^3)\big)^2\big\rangle_x + 2\big\langle\big(H_n(\sigma^1)-H_n(\sigma^2)\big)^2\big\rangle_x \\
&= 4\big\langle\big(H_n(\sigma^1)-H_n(\sigma^2)\big)^2\big\rangle_x.
}
From the previous two displays, we find
\eq{
\Big|\frac{\partial}{\partial x} \langle |H_n(\sigma^1)-H_n(\sigma^2)|\rangle_x\Big| 
&\leq 2\big\langle\big(H_n(\sigma^1)-H_n(\sigma^2)\big)^2\big\rangle_x \\
&= 4\langle H_n(\sigma)^2\rangle_x-4\langle H_n(\sigma)\rangle_x^2.
}
In light of this inequality, \eqref{trivial_integrals} now shows
\eq{
\langle |H_n(\sigma^1)-H_n(\sigma^2)|\rangle_{\beta_0} 
&\leq \frac{1}{\beta_1-{\beta_0}}\int_{\beta_0}^{\beta_1}\langle|H_n(\sigma^1)-H_n(\sigma^2)|\rangle_\beta\ \dd \beta \\
&\phantom{\leq}+\frac{4}{\beta_1-{\beta_0}}\int_{\beta_0}^{\beta_1}\int_{\beta_0}^\beta (\langle H_n(\sigma)^2\rangle_x-\langle H_n(\sigma)\rangle_x^2)\ \dd x\, \dd \beta \\
&\leq \frac{2}{\beta_1-{\beta_0}}\int_{\beta_0}^{\beta_1}\big\langle|H_n(\sigma)-\langle H_n(\sigma)\rangle_\beta|\big\rangle_\beta\ \dd \beta \\
&\phantom{\leq}+4\int_{\beta_0}^{\beta_1} (\langle H_n(\sigma)^2\rangle_x-\langle H_n(\sigma)\rangle_x^2)\ \dd x,
}
where
\eq{
&\frac{2}{\beta_1-{\beta_0}}\int_{\beta_0}^{\beta_1}\big\langle|H_n(\sigma)-\langle H_n(\sigma)\rangle_\beta|\big\rangle_\beta\ \dd \beta \\
&\leq 2\bigg(\frac{1}{\beta_1-{\beta_0}}\int_{\beta_0}^{\beta_1}\big\langle|H_n(\sigma)-\langle H_n(\sigma)\rangle_\beta|\big\rangle_\beta^2\ \dd \beta\bigg)^{1/2} \\
&\leq2\bigg(\frac{1}{\beta_1-{\beta_0}}\int_{\beta_0}^{\beta_1}(\langle H_n(\sigma)^2\rangle_\beta-\langle H_n(\sigma)\rangle_\beta^2)\ \dd \beta\bigg)^{1/2}. \\
}
In summary,
\eeq{ \label{beta_integral_summary}
\Big\langle\Big|\frac{H_n(\sigma)}{n}-F_n'(\sigma)\Big|\Big\rangle_{\beta_0}
&= \Big\langle\Big|\frac{H_n(\sigma)}{n}-\frac{\langle H_n(\sigma)\rangle_{\beta_0}}{n}\Big|\Big\rangle_{\beta_0} \\
&\leq\frac{\langle|H_n(\sigma^1)-H_n(\sigma^2)|\rangle_{\beta_0}}{n}  \\
&\leq 2\sqrt{\frac{I_n(\beta_1)}{n(\beta_1-{\beta_0})}}+4I_n(\beta_1),
}
where
\eq{
I_n(\beta_1) &\coloneqq \frac{1}{n}\int_{\beta_0}^{\beta_1}(\langle H_n(\sigma)^2\rangle_\beta-\langle H_n(\sigma)\rangle_\beta^2)\ \dd \beta 
\stackrel{\mbox{\footnotesize\eqref{nrg_2deriv}}}{=} F_n'(\beta_1) - F_n'(\beta_0).
}
Therefore, convexity of $F_n(\cdot)$ implies
\eq{
&I_n(\beta_1) 
\leq
 \frac{F_n(\beta_1+h)-F_n(\beta_1)}{h} - \frac{F_n({\beta_0})-F_n({\beta_0}-h)}{h} \\
&\stackrel{{\mbox{\footnotesize\hspace{2ex}\eqref{Delta_def}\hspace{2ex}}}}{=} \frac{p(\beta_1+h)-p(\beta_1)}{h} - \frac{p({\beta_0})-p({\beta_0}-h)}{h} + \Delta_n^+(\beta_1,h) + \Delta_n^-({\beta_0},h) \\
&\hspace{-1.3ex}\stackrel{\mbox{\footnotesize\eqref{h_choice},\eqref{beta_prime_choice}}}{\leq} 3\eps + \Delta_n^+(\beta_1,h) + \Delta_n^-({\beta_0},h).
}
As $n\to\infty$, \eqref{free_energy_assumption} shows that $\Delta_n^+(\beta_1,h)$ and $\Delta_n^-({\beta_0},h)$ each converge to $0$ almost surely and in $L^1$.
Thus \eqref{beta_integral_summary} and the above display together yield the desired result, as $\eps$ is arbitrary.
\end{proof}

\subsection{Temperature perturbations} \label{temp_perturbations}
Here we derive upper bounds for the effects of temperature perturbations on certain expectations with respect to~$\mu_{n}^\beta$.

\begin{lemma} \label{connecting_betas}
The following statements hold for any $\beta_1\geq\beta_0\geq0$.
\begin{itemize}
\item[(a)] For any measurable $f : \Sigma_n\to[-1,1]$,
\eq{
|\langle f(\sigma)\rangle_{\beta_1} - \langle f(\sigma)\rangle_{\beta_0}| \leq \sqrt{n(\beta_1-\beta_0)(F_n'(\beta_1)-F_n'(\beta_0))}.
}
\item[(b)] For any $\sigma \in \Sigma_n$,
\eeq{\label{connecting_betas_2}
\frac{1}{n}\Big|\sum_i\vphi_i\langle\vphi_i\rangle_{\beta_1} - \sum_i \vphi_i\langle\vphi_i\rangle_{\beta_0}\Big| \leq \sqrt{n(\beta_1-\beta_0)(F_n'(\beta_1)-F_n'(\beta_0))}.
}
\item[(c)] Finally,
\eeq{ \label{connecting_betas_3}
\frac{1}{n}\Big|\sum_i\langle\vphi_i\rangle^2_{\beta_1} -\sum_i\langle\vphi_i\rangle^2_{\beta_0}\Big|
\leq 2\sqrt{n(\beta_1-\beta_0)(F_n'(\beta_1)-F_n'(\beta_0))}.
}
\end{itemize}
\end{lemma}

\begin{proof}
All three claims follow from two crucial observations.
First, for any $f \in L^2(\Sigma_n)$,
\eeq{ \label{crucial_1}
&\Big|\frac{\partial}{\partial \beta}\langle f(\sigma)\rangle_\beta\Big| 
=  |\langle f(\sigma)H_n(\sigma)\rangle_\beta - \langle f(\sigma) \rangle_\beta\langle H_n(\sigma)\rangle_\beta| \\
&\stackrel{\phantom{\mbox{\footnotesize\eqref{nrg_2deriv}}}}{\leq} \sqrt{\langle H_n(\sigma)^2\rangle_\beta - \langle H_n(\sigma)\rangle_\beta^2}\sqrt{\langle f(\sigma)^2\rangle_\beta - \langle f(\sigma)\rangle_\beta^2} \\
&\stackrel{\mbox{\footnotesize\eqref{nrg_2deriv}}}{=} \sqrt{nF_n''(\beta)}\sqrt{\langle f(\sigma)^2\rangle_\beta - \langle f(\sigma)\rangle_\beta^2}
\leq \sqrt{nF_n''(\beta)}\sqrt{\langle f(\sigma)^2\rangle_\beta}. 
}
And second,
\eeq{ \label{crucial_2}
\int_{\beta_0}^{\beta_1} \sqrt{nF_n''(\beta)}\ \dd \beta
&\leq \sqrt{n(\beta_1-\beta_0)\int_{\beta_0}^{\beta_1} F_n''(\beta)\ \dd \beta} \\
&= \sqrt{n(\beta_1-\beta_0)(F_n'(\beta_1)-F_n'(\beta_0))}.
}
Then part (a) immediately follows, since
\eq{
|f| \leq 1 \quad &\stackrel{\mbox{\footnotesize\eqref{crucial_1}}}{\Rightarrow} \quad 
\Big|\frac{\partial}{\partial \beta}\langle f(\sigma)\rangle_\beta\Big| \leq \sqrt{nF_n''(\beta)} \\
&\stackrel{\mbox{\footnotesize\eqref{crucial_2}}}{\Rightarrow} \quad
|\langle f(\sigma)\rangle_{\beta_1} - \langle f(\sigma)\rangle_{\beta_0}| \leq \sqrt{n(\beta_1-\beta_0)(F_n'(\beta_1)-F_n'(\beta_0))}.
}
For part (b), we first observe that if $0\leq\beta\leq\beta_1$, then
\eq{
\Big|\frac{\partial}{\partial\beta}\langle\vphi_i\rangle_\beta\Big| &\stackrel{\mbox{\footnotesize\eqref{crucial_1}}}{\leq}\sqrt{nF_n''(\beta)}\sqrt{\langle\vphi_i^2\rangle_\beta} \\
&\stackrel{\phantom{\mbox{\footnotesize\eqref{crucial_1}}}}{=} \sqrt{nF_n''(\beta)}\sqrt{\frac{E_n(\vphi_i^2\e^{\beta H_n(\sigma)})}{Z_n(\beta)}} \\
&\stackrel{\phantom{\mbox{\footnotesize\eqref{crucial_1}}}}{\leq} \sqrt{nF_n''(\beta)}\sqrt{\frac{E_n(\vphi_i^2)}{Z_n(\beta)} + \frac{E_n(\vphi_i^2\e^{\beta_1 H_n(\sigma)})}{Z_n(\beta)}}  \\
&\stackrel{\phantom{\mbox{\footnotesize\eqref{crucial_1}}}}{\leq} \sqrt{n\max_{\beta_0\in[0,\beta_1]}F_n''(\beta_0)}\sqrt{\frac{\max(Z_n(0),Z_n(\beta_1))}{\min_{\beta_0\in[0,\beta_1]} Z_n(\beta_0)}}\sqrt{\langle \vphi_i^2\rangle_0 + \langle\vphi_i^2\rangle_{\beta_1}},
}
where now the right-hand side is independent of $\beta$ and (almost surely) finite.
Moreover, we have the following finiteness condition when summing over $i$:
\eq{
\sum_{i} |\vphi_i|\sqrt{\langle \vphi_i^2\rangle_0 + \langle\vphi_i^2\rangle_{\beta_1}}
&\leq\sqrt{ \sum_i \vphi_i^2\sum_i (\langle\vphi_i^2\rangle_0 + \langle\vphi_i^2\rangle_{\beta_1})} 
\stackrel{\mbox{\footnotesize\eqref{variance_assumption}}}{=} \sqrt{2}n < \infty.
}
It thus follows that
\eq{
\frac{\partial}{\partial\beta} \sum_i \vphi_i\langle\vphi_i\rangle_\beta = \sum_i \vphi_i\frac{\partial}{\partial\beta}\langle\vphi_i\rangle_\beta.
}
In particular,
\eq{
\Big|\frac{\partial}{\partial\beta} \frac{1}{n}\sum_i \vphi_i\langle\vphi_i\rangle_\beta\Big|
&\stackrel{\phantom{\mbox{\footnotesize\eqref{crucial_1}}}}{\leq} \frac{1}{n}\sum_i \Big|\vphi_i\frac{\partial}{\partial\beta}\langle\vphi_i\rangle_\beta\Big| \\
&\stackrel{\mbox{\footnotesize\eqref{crucial_1}}}{\leq} \sqrt{\frac{F_n''(\beta)}{n}} \sum_i |\vphi_i| \sqrt{\langle\vphi_i^2\rangle_\beta} \\
&\stackrel{\phantom{\mbox{\footnotesize\eqref{crucial_1}}}}{\leq} \sqrt{\frac{F_n''(\beta)}{n}}\sqrt{\sum_i \vphi_i^2\sum_i \langle\vphi_i^2\rangle_\beta}
\stackrel{\mbox{\footnotesize\eqref{variance_assumption}}}{=} \sqrt{nF_n''(\beta)}.
}
As in part (a), \eqref{crucial_2} now proves \eqref{connecting_betas_2}.
For part (c), we can argue similarly in order to obtain
\eq{
\Big|\frac{\partial}{\partial\beta} \frac{1}{n}\sum_i \langle\vphi_i\rangle_\beta^2\Big| 
&\stackrel{\phantom{\mbox{\footnotesize\eqref{crucial_1}}}}{=} \Big|\frac{2}{n}\sum_i \langle\vphi_i\rangle_\beta\frac{\partial}{\partial\beta}\langle\vphi_i\rangle_\beta\Big| \\
&\stackrel{{\mbox{\footnotesize\eqref{crucial_1}}}}{\leq} 2\sqrt{\frac{F_n''(\beta)}{n}}\sum_i|\langle\vphi_i\rangle_\beta|\sqrt{\langle \vphi_i^2\rangle_\beta} \\
&\stackrel{\phantom{\mbox{\footnotesize\eqref{crucial_1}}}}{\leq} 2\sqrt{\frac{F_n''(\beta)}{n}}\sum_i \langle \vphi_i^2\rangle_\beta 
\stackrel{\mbox{\footnotesize\eqref{variance_assumption}}}{=} 2\sqrt{nF_n''(\beta)},
}
from which \eqref{crucial_2} proves \eqref{connecting_betas_3}.
\end{proof}

\section{Proof of Theorem~\ref{averages_squared}} \label{proof_1}
Recall the event under consideration:
\eq{
B_\delta = \Big\{\frac{1}{n}\sum_i\langle\vphi_i\rangle^2\leq\delta\Big\}.
}
The proof of Theorem~\ref{averages_squared} is a perturbative argument using an Ornstein--Uhlenbeck (OU) flow on the environment,
\eeq{ \label{OU_def}
\vc g_t \coloneqq \e^{-t}\vc g + \e^{-t} \vc W({\e^{2t}-1}), \quad t\geq0,
}
where $\vc W(\cdot) = (W_i(\cdot))_{i=1}^\infty$ is a collection of independent Brownian motions that are also independent of $\vc g = \vc g_0$, and the above definition is understood coordinate-wise.
Within Section~\ref{proof_1}, we denote expectation with respect to $\mu_{n,\vc g_t}^\beta$ by $\langle \cdot \rangle_t$, not to be confused with $\langle\cdot\rangle_\beta$ used in Section~\ref{prep_section}.
We now prove Theorem~\ref{averages_squared} by juxtaposing the following two propositions.
Notice that if $\P(B_\delta) = 0$, then there is nothing to be done; therefore, we may henceforth assume $\P(B_\delta) > 0$ so that conditioning on $B_\delta$ is well-defined.

\begin{prop} \label{prep_prop_good}
If $\beta$ is a point of differentiability for $p(\cdot)$, and $p'(\beta) < \beta$, then there exists $\kappa = \kappa(\beta) > 0$ such that the following holds:
For any $\eps > 0$, there is $T = T(\beta,\eps)$ sufficiently large that
\eeq{ \label{prep_eq_good}
\liminf_{n\to\infty} \P\bigg(\Big|\kappa - \frac{1}{T/n}\int_0^{T/n} \frac{1}{n}\sum_i \langle \vphi_i\rangle_t^2\ \dd t\Big| \leq \eps \bigg) \geq 1 - \eps.
}
More specifically,
\eq{
\kappa(\beta) = \frac{\beta-p'(\beta)}{\beta}.
}
\end{prop}

For the statement of the second result, 
let $\FFF_t$ denote the $\sigma$-algebra generated by $\vc g_0$ and $(\vc W(s))_{0\leq s\leq \e^{2t}-1}$.

\begin{prop} \label{prep_prop_bad}
Assume $\beta$ is a point of differentiability for $p(\cdot)$.
Then there is a process $(I_t)_{t>0}$ adapted to the filtration $(\FFF_t)_{t>0}$, such that
the following statements hold:
\begin{itemize}
\item[(a)] For any $T,\eps>0$,
\eeq{ \label{prep_prop_bad_eq1}
\lim_{n\to\infty} \P\bigg(\Big|I_{T/n} - \frac{1}{T/n}\int_0^{T/n} \frac{1}{n}\sum_{i}\langle\vphi_i\rangle_t^2\ \dd t\Big| > \eps\bigg) = 0.
} 
\item[(b)] For any $T,\eps_1,\eps_2>0$, there exist $\delta_1 = \delta_1(\beta,T,\eps_1,\eps_2)>0$ sufficiently small and $n_0=n_0(\beta,T,\eps_1,\eps_2)$ sufficiently large, that
\eeq{ \label{prep_prop_bad_eq2}
\P\givenp[\bigg]{\Big|I_{T/n} - \frac{1}{n}\sum_i \langle\vphi_i\rangle^2\Big| \geq \eps_1}{B_{\delta}} \leq \eps_2 \quad \text{for all $0<\delta\leq\delta_1$, $n\geq n_0$}.
}
\end{itemize}
\end{prop}


\begin{proof}[Proof of Theorem~\ref{averages_squared}]
Let $\eps > 0$ be given, and assume the hypotheses of Proposition~\ref{prep_prop_good}.
By that result, there is $\kappa>0$ and $T$ large enough that
\eeq{ \label{first_liminf}
\liminf_{n\to\infty} \P\bigg(\frac{1}{T/n}\int_0^{T/n} \frac{1}{n}\sum_i \langle \vphi_i\rangle_t^2\ \dd t \geq \frac{4\kappa}{5}\bigg) \geq 1-\frac{\eps}{2}.
}
Let $(I_t)_{t\geq0}$ be the process guaranteed by Proposition~\ref{prep_prop_bad}, and define the events
\eq{
G &\coloneqq \bigg\{\frac{1}{T/n}\int_0^{T/n} \frac{1}{n}\sum_i \langle \vphi_i\rangle_t^2\ \dd t \geq \frac{4\kappa}{5}\bigg\}, \\
H &\coloneqq \bigg\{\frac{1}{T/n}\int_0^{T/n} \frac{1}{n}\sum_i \langle \vphi_i\rangle_t^2\ \dd t \leq \frac{3\kappa}{5} \bigg\}, \\
H_1 &\coloneqq \bigg\{\Big|I_{T/n} - \frac{1}{T/n}\int_0^{T/n} \frac{1}{n}\sum_i\langle\vphi_i\rangle_t^2\ \dd t\Big| \leq \frac{\kappa}{5}\Big\}, \\
H_2 &\coloneqq \bigg\{\Big|I_{T/n} - \frac{1}{n}\sum_i \langle\vphi_i\rangle^2\Big| \leq \frac{\kappa}{5}\bigg\}.
}
By Proposition~\ref{prep_prop_bad}(a),
\eeq{ \label{second_liminf}
\lim_{n\to\infty} \P(H_1) = 1.
}
And by Proposition~\ref{prep_prop_bad}(b), we can choose $0 < \delta \leq \kappa/5$ sufficiently small and $n_0$ sufficiently large that
\eeq{ \label{conditional_assumption}
\P\givenp{H_2}{B_{\delta}} \geq \frac{1}{2} \quad \text{for all $n\geq n_0$}.
}
Observe that $B_{\delta} \cap H_1\cap H_{2} \subset H$, and clearly the events $G$ and $H$ are disjoint.
We thus have
\eq{
\P(B_{\delta} \cap H_1 \cap H_2) \leq \P(H) \leq 1 - \P(G).
}
On the other hand,
\eq{
\P(B_{\delta} \cap H_1 \cap H_2)
&\stackrel{\phantom{\mbox{\footnotesize{\eqref{conditional_assumption}}}}}{\geq} \P(H_1) + \P(H_2 \cap B_{\delta}) - 1 \\
&\stackrel{\phantom{\mbox{\footnotesize{\eqref{conditional_assumption}}}}}{=} \P(H_1) - 1 + \P\givenp{H_2}{B_{\delta}}\P(B_{\delta}) \\
&\stackrel{\mbox{\footnotesize{\eqref{conditional_assumption}}}}{\geq} \P(H_1) - 1 + \frac{\P(B_{\delta})}{2}.
}
Putting the two previous displays together, we find
\eq{
\P(B_{\delta}) \leq 2\big(2 - \P(G) - \P(H_1)\big),
}
and so
\eq{
\limsup_{n\to\infty} \P(B_{\delta}) \leq 2\big(2 - \liminf_{n\to\infty} \P(G) - \lim_{n\to\infty} \P(H_1)\big) 
\stackrel{\mbox{\footnotesize\eqref{first_liminf},\eqref{second_liminf}}}{\leq}\eps.
}

\end{proof}

\subsection{Proof of Proposition~\ref{prep_prop_good}}
We will need to recall some facts about Ornstein--Uhlenbeck processes.
To avoid technical complications, we restrict ourselves to finite-dimensional OU processes, and then take an appropriate limit at a later stage.

\subsubsection{General OU theory} \label{OU_theory}
Fix a positive integer $N$, and consider a vector $\vc g = (g_1,\dots,g_N)$ of i.i.d.~standard normal random variables.
Let $\vc W = (\vc W(t))_{t\geq0}$ be an independent $N$-dimensional Brownian motion.
The OU flow starting at $\vc g$ is given by 
\eq{
\vc g_t \coloneqq \e^{-t}\vc g + \e^{-t}\vc W(\e^{2t}-1), \quad t\geq0.
}
This is a continuous-time, stationary Markov chain.
Let $(\PP_t)_{t\geq0}$ denote the OU semigroup; that is, for $f: \R^N \to \R$,
\eq{
\PP_tf(\vc x) \coloneqq \E f\big(\e^{-t}\vc x + \e^{-t}\vc W({\e^{2t}-1})\big), \quad \vc x\in\R^N.
}
Denote the OU generator by $\LL \coloneqq \Delta - \vc x \cdot \nabla$.
It is especially useful to consider the spectral decomposition of $\LL$, whose eigenfunctions are the multivariate Hermite polynomials.
For our purposes, it suffices to recall the following well-known facts (see, for instance, \cite[Chapter 6]{chatterjee14}):
\begin{itemize}
\item Let $\gamma_N$ denote the $N$-dimensional standard Gaussian measure.
There is an orthonormal basis $\{\phi_j\}_{j=0}^\infty$ of $L^2(\gamma_N)$ consisting of eigenfunctions of $\LL$, where $\phi_0 \equiv 1$, $\LL \phi_0 = \lambda_0 \phi_0= 0$, and $\LL \phi_j = -\lambda_j\phi_j$ with $\lambda_j > 0$ for $j\geq1$.
Therefore, if $f = \sum_{j=0}^\infty a_j\phi_j \in L^2(\gamma_N)$, then
\begin{align} 
\E f(\vc g) &= a_0, \label{mean_coef} \\
\LL f &= -\sum_{j=1}^\infty \lambda_ja_j\phi_j, \label{apply_L} \\
\Rightarrow \E \LL f(\vc g) &= 0 \label{L_zero_mean}.
\end{align}
Furthermore, if $f_1 = \sum_{j=0}^\infty a_j\phi_j, f_2 = \sum_{j=0}^\infty b_j\phi_j \in L^2(\gamma_N)$, then
\eeq{ \label{cov_formula}
\Cov\big(f_1(\vc g),f_2(\vc g)\big) = \sum_{j=1}^\infty a_jb_j.
}
\item The OU semigroup acts on $L^2(\gamma_N)$ by
\eq{
\PP_t\phi_j = \e^{-\lambda_j t}\phi_j, \quad j\geq0.
}
Therefore, if $f = \sum_{j=0}^\infty a_j\phi_j \in L^2(\gamma_N)$, then
\eeq{ \label{apply_PL}
\PP_t\LL f = -\sum_{j=1}^\infty \lambda_ja_j\e^{-\lambda_j t}\phi_j.
}
\item The associated Dirichlet form is given by
\eq{
- \E[f_1(\vc g)\LL f_2(\vc g)] = \E[\nabla f_1(\vc g)\cdot\nabla f_2(\vc g)],
}
whenever $f_1$ and $f_2$ are twice-differentiable functions in $L^2(\gamma_N)$ such that both expectations above are finite.
In particular, if $f_1 = f_2 = \sum_{j=0}^\infty a_j\phi_j\in L^2(\gamma_N)$ is twice-differentiable, then
\eeq{ \label{exp_grad_formula}
\E(\|\nabla f(\vc g)\|^2) = \sum_{j=1}^\infty \lambda_{j}a_{j}^2.
}
\end{itemize}


\begin{lemma} \label{OU_variance_lemma}
For any twice differentiable $f\in L^2(\gamma_N)$ with $\LL f \in L^2(\gamma_N)$, we have
\eq{
\Var\bigg(\frac{1}{t}\int_0^t \LL f(\vc g_s)\ \dd s\bigg) \leq \frac{2}{t}\E(\|\nabla f(\vc g)\|^2).
}
\end{lemma}

\begin{proof}
Take any $0\leq s\leq t$.
By the law of total variance, we have
\eq{
\Cov\big(f(\vc g_s),f(\vc g_t)\big) &= 
\E\big[\Cov\givenp{f(\vc g_s),f(\vc g_t)}{\vc g_s}\big] + \Cov\big(f(\vc g_s),\E\givenk{f(\vc g_t)}{\vc g_s}\big) \\
&= 0 + \Cov\big(f(\vc g_s),\E\givenk{f(\vc g_t)}{\vc g_s}\big) \\
&= \Cov\big(f(\vc g_s),\PP_{t-s}f(\vc g_s)\big) \\
&= \Cov\big(f(\vc g_0),\PP_{t-s}f(\vc g_0)\big).
}
In particular, if we write $f$ in the form $f = \sum_{j=0}^\infty a_{j}\phi_{j}$, then
\eq{
\Cov\big(\LL f(\vc g_s),\LL f(\vc g_t)\big) 
&\stackrel{\phantom{\mbox{\footnotesize\eqref{apply_L},\eqref{apply_PL},\eqref{cov_formula}}}}{=}  \Cov\big(\LL f(\vc g_0),\PP_{t-s}\LL f(\vc g_0)\big)  \\
&\stackrel{\mbox{\footnotesize\eqref{apply_L},\eqref{apply_PL},\eqref{cov_formula}}}{=} \sum_{j=1}^\infty \lambda_{j}^2a_j^2\e^{-\lambda_j(t-s)}.
}
Therefore,
\eq{
\int_0^t \Cov\big(\LL f(\vc g_s),\LL f(\vc g_t)\big)\ \dd s
&=\int_0^t \sum_{j=1}^\infty  \lambda_j^2a_j^2 \e^{-\lambda_j(t-s)}\ \dd s \\
&= \sum_{j=1}^\infty \lambda_ja_j^2(1-\e^{-\lambda_jt}) \\
&\leq \sum_{j=1}^\infty	 \lambda_ja_j^2
\stackrel{\mbox{\footnotesize\eqref{exp_grad_formula}}}{=} \E(\|\nabla f(\vc g)\|^2).
}
Hence
\eq{
\Var\bigg(\int_0^t \LL f(\vc g_s)\ \dd s\bigg)
&= \int_0^t\int_0^t \Cov\big(\LL f(\vc g_s),\LL f(\vc g_u)\big)\ \dd s\, \dd u \\
&= 2\int_0^t\int_0^u \Cov\big(\LL f(\vc g_s),\LL f(\vc g_u)\big)\ \dd s\, \dd u \\
&\leq 2t\E(\|\nabla f(\vc g)\|^2).
}
\end{proof}

\begin{proof}[Proof of Proposition~\ref{prep_prop_good}]
Let $(\vc g_t)_{t\geq0}$ be the OU flow from \eqref{OU_def}, and write
\eq{
g_i(t) \coloneqq \e^{-t}g_i + \e^{-t}W_i(\e^{2t}-1), \quad i\geq1.
}
Recall that $\langle \cdot \rangle_t$ denotes expectation with respect to $\mu_{n,\vc g_t}^\beta$. 
Let $Z_{n,t}(\beta)$ and $F_{n,t}(\beta)$ be the associated partition function and free energy, respectively.
That is, with $H_{n,t} \coloneqq \sum_i g_i(t)\vphi_i$, we have
\eq{
Z_{n,t}(\beta) \coloneqq E_n(\e^{\beta H_{n,t}}), \qquad
F_{n,t}(\beta) \coloneqq \frac{1}{n}\log Z_{n,t}(\beta).
}
So that we can use the finite-dimensional facts discussed before, define $H_{n,t,N} \coloneqq \sum_{i=1}^N g_i(t)\vphi_i$,
as well as
\eq{
Z_{n,t,N}(\beta) \coloneqq E_n(\e^{\beta H_{n,t,N}}), \qquad
F_{n,t,N}(\beta) \coloneqq \frac{1}{n}\log Z_{n,t,N}(\beta), \quad N\geq0.
}
Define $f : \R^{N} \to \R$ by
\eq{
f(\vc x) \coloneqq \frac{1}{n}\log E_n(\e^{\beta \sum_{i=1}^N x_i\vphi_i}),
}
so that $f(\vc g_t) = F_{n,t,N}(\beta)$, where $\vc g_t$ is understood to mean $(g_1(t),\dots,g_N(t))$.
Note that $f\in L^2(\gamma_{N})$, since
$\log^2 x \leq x + x^{-1}$ for $x>0$,
and so using the
same arguments as in Lemma~\ref{moments_lemma} yields
\eq{
\E \log^2 Z_{n,t,N}(\beta) &\leq \E Z_{n,t,N}(\beta)+\E[Z_{n,t,N}(\beta)^{-1}] \\
&\leq E_n(\E \e^{\beta H_{n,t,N}(\sigma)}) + E_n(\E \e^{-\beta H_{n,t,N}(\sigma)})\\
&= 2E_n(\e^{\frac{\beta^2}{2}\sum_{i=1}^N\vphi_i^2}) \stackrel{\mbox{\footnotesize\eqref{variance_assumption}}}{\leq} 2\e^{\frac{\beta^2}{2}n}.
}
Similar to \eqref{finite_gibbs_expectation}, for general $\ff \in L^2(\Sigma_n)$, we define
\eeq{ \label{finite_gibbs_expectation_time}
\langle \ff(\sigma)\rangle_{t,N} = \frac{E_n(\ff(\sigma)\e^{\beta H_{n,t,N}(\sigma)})}{E_n(\e^{\beta H_{n,t,N}(\sigma)})}.
}
Observe that
\eq{
\frac{\partial f}{\partial x_i}(\vc g_t) = \frac{\beta\langle\vphi_i\rangle_{t,N}}{n}, \quad 1\leq i\leq N,
}
which implies
\eeq{ \label{grad_ineq}
\|\nabla f(\vc g_t)\|^2 = \frac{\beta^2}{n^2}\sum_{i=1}^N\langle\vphi_i\rangle_{t,N}^2 \leq 
\frac{\beta^2}{n^2}\sum_{i=1}^N\langle\vphi_i^2\rangle_{t,N} \stackrel{\mbox{\footnotesize\eqref{variance_assumption}}}{\leq} \frac{\beta^2}{n},
}
as well as
\eq{
\vc g_t \cdot \nabla f(\vc g_t) = \frac{\beta}{n}\sum_{i=1}^Ng_i(t)\langle \vphi_i\rangle_{t,N} 
= \frac{\beta}{n}\langle H_{n,t,N}(\sigma)\rangle_{t,N}
\stackrel{\mbox{\footnotesize\eqref{nrg_2deriv}}}{=} \beta F_{n,t,N}'(\beta),
}
where the derivative is with respect to $\beta$.
Note that
\eeq{ \label{Fprime_L2}
\E[F_{n,t,N}'(\beta)^2]
&\stackrel{\phantom{\mbox{\footnotesize\eqref{variance_assumption}}}}{=}  \frac{1}{n^2}\E\bigg[\Big(\sum_{i=1}^N g_i(t)\langle\vphi_i\rangle_{t,N}\Big)^2\bigg] \\
&\stackrel{\phantom{\mbox{\footnotesize\eqref{variance_assumption}}}}{\leq} \frac{1}{n^2}\E\bigg[\Big(\sum_{i=1}^N g_i(t)^2\Big)\Big(\sum_{i=1}^N\langle\vphi_i\rangle_{t,N}^2\Big)\bigg] \\
&\stackrel{\phantom{\mbox{\footnotesize\eqref{variance_assumption}}}}{\leq} \frac{1}{n^2}\E\bigg[\Big(\sum_{i=1}^N g_i(t)^2\Big)\Big(\sum_{i=1}^N\langle\vphi_i^2\rangle_{t,N}\Big)\bigg] \\
&\stackrel{\mbox{\footnotesize\eqref{variance_assumption}}}{\leq} \frac{1}{n}\E\Big(\sum_{i=1}^N g_i(t)^2\Big) 
= \frac{N}{n} < \infty.
}
Furthermore,
\eq{
\frac{\partial^2 f}{\partial x_i^2}(\vc g_t) = \frac{\beta^2}{n}(\langle \vphi_i^2\rangle_{t,N}-\langle\vphi_i\rangle_{t,N}^2), \quad 1\leq i\leq N.
}
We thus have
\eq{
\LL f(\vc g_t) &= \frac{\beta^2}{n} \sum_{i=1}^N (\langle\vphi_i^2\rangle_{t,N}-\langle\vphi_i\rangle_{t,N}^2) - \beta F_{n,t,N}'(\beta).
}
From \eqref{Fprime_L2}, it is clear that $\LL f \in L^2(\gamma_{N})$.
Therefore, by Lemma~\ref{OU_variance_lemma} and \eqref{grad_ineq},
\eq{
\Var\bigg(\frac{1}{t} \int_0^{t} \Big[\frac{\beta^2}{n} \sum_{i=1}^N (\langle\vphi_i^2\rangle_{s,N}-\langle\vphi_i\rangle_{s,N}^2) - \beta F_{n,s,N}'(\beta)\Big]\ \dd s\bigg) \leq \frac{2\beta^2}{tn}.
}
Moreover, from \eqref{L_zero_mean} we know
\eq{
\E\bigg(\frac{1}{t} \int_0^{t} \Big[\frac{\beta^2}{n} \sum_{i=1}^N (\langle\vphi_i^2\rangle_{s,N}-\langle\vphi_i\rangle_{s,N}^2) - \beta F_{n,s,N}'(\beta)\Big]\ \dd s\bigg) = 0.
}
We can now apply \eqref{exp_moments_lemma_a} (together with \eqref{nrg_2deriv}) and \eqref{limit_for_later} to take the limit $N\to\infty$ in the two previous displays and obtain
\eq{
\Var\bigg(\frac{1}{t} \int_0^{t} \Big[\beta^2 - \frac{\beta^2}{n} \sum_{i} \langle\vphi_i\rangle_{s}^2 - \beta F_{n,s}'(\beta)\Big]\ \dd s\bigg) &\leq \frac{2\beta^2}{tn}, \\
\E\bigg(\frac{1}{t} \int_0^{t} \Big[\beta^2 - \frac{\beta^2}{n} \sum_{i} \langle\vphi_i\rangle_{s}^2 - \beta F_{n,s}'(\beta)\Big]\ \dd s\bigg) &= 0.
}
Consequently, for any $\eps > 0$, Chebyshev's inequality shows
\eeq{ \label{limsup_prob_1}
\P\bigg(\Big|\frac{1}{t} \int_0^{t} \Big[\beta - \frac{\beta}{n}\sum_i\langle\vphi_i\rangle_s^2 - F_{n,s}'(\beta)\Big]\ \dd s\Big| \geq \frac{\eps}{2}\bigg) \leq \frac{8}{tn\eps^2}.
}
Now consider that
\eq{
\E\Big|p'(\beta)-\frac{1}{t}\int_0^t F_{n,s}'(\beta)\ \dd s\Big|
&\leq \frac{1}{t}\int_0^t \E|p'(\beta) - F_{n,s}'(\beta)|\ \dd s \\
&= \E|p'(\beta) - F_n'(\beta)|.
}
Therefore, if $\beta$ is a point of differentiability for $p(\cdot)$, then for any sequence $(t(n))_{n\geq1}$, Lemma~\ref{betas_converging} guarantees
\eeq{ \label{limsup_prob_2}
\limsup_{n\to\infty} \P\bigg(\Big|p'(\beta)-\frac{1}{t(n)}\int_0^{t(n)} F_{n,s}'(\beta)\ \dd s\Big| \geq \frac{\eps}{2}\bigg) = 0.
}
When $t = t(n) = T/n$ for fixed $T$, \eqref{limsup_prob_1} and \eqref{limsup_prob_2} together show
\eq{
\limsup_{n\to\infty} \P\bigg(\Big|\frac{1}{T/n} \int_0^{T/n} \Big[\beta - \frac{\beta}{n}\sum_i\langle\vphi_i\rangle_s^2 - p'(\beta)\Big]\ \dd s\Big| \geq \eps\bigg) \leq \frac{8}{T\eps^2}.
}
Assuming $p'(\beta)<\beta$, we let $\kappa = \kappa(\beta) \coloneqq \frac{\beta - p'(\beta)}{\beta} > 0$.
Then the previous display implies
\eq{
\limsup_{n\to\infty} \P\bigg(\Big|\kappa - \frac{1}{T/n} \int_0^{T/n} \frac{1}{n}\sum_i\langle\vphi_i\rangle_s^2\ \dd s\Big| \geq \eps\bigg) \leq \frac{8}{T\beta^2\eps^2}.
}
The proof is completed by taking $T = T(\beta,\eps)$ sufficiently large that
\eq{
\frac{8}{T\beta^2\eps^2} \leq \eps.
}
\end{proof}

\subsection{Proof of Proposition~\ref{prep_prop_bad}}
Let us rewrite \eqref{OU_def} as
\eq{
\vc g_t = \vc g + \e^{-t} \vc W(\e^{2t}-1) + (\e^{-t}-1)\vc g, \quad t\geq 0.
}
Recall that $\langle \cdot\rangle_0 = \langle\cdot\rangle$.
For any $f \in L^2(\Sigma_n)$, we have
\eq{
\langle f(\sigma)\rangle_t = \frac{\langle f(\sigma)\e^{\beta \e^{-t}\sum_{i} W_i(\e^{2t}-1)\vphi_i} \e^{\beta(\e^{-t} - 1)H_n(\sigma)}\rangle}{\langle\e^{\beta\e^{-t}\sum_{i} W_i(\e^{2t}-1)\vphi_i}\e^{\beta(\e^{-t} - 1)H_n(\sigma)}\rangle}.
}
In light of Lemma~\ref{lemma:step_1_2}, we anticipate that for $t = O(n^{-1})$,
\eeq{ \label{Qt_def}
\langle f(\sigma)\rangle_t  
&\approx \frac{\langle f(\sigma)\e^{\beta \e^{-t}\sum_{i} W_i(\e^{2t}-1)\vphi_i}\e^{-\beta tnp'(\beta)}\rangle }{\langle\e^{\beta\e^{-t}\sum_{i} W_i(\e^{2t}-1)\vphi_i}\e^{-\beta tnp'(\beta)} \rangle} \\
&= \frac{\langle f(\sigma)\e^{\beta \e^{-t}\sum_{i} W_i(\e^{2t}-1)\vphi_i}\rangle }{\langle\e^{\beta\e^{-t}\sum_{i} W_i(\e^{2t}-1)\vphi_i}\rangle}
 \eqqcolon Q_t(f).
}
Indeed,  the process that will satisfy the conclusions of Proposition~\ref{prep_prop_bad} is
\eeq{ \label{I_def}
I_t \coloneqq \frac{1}{t}\int_0^t \frac{1}{n}\sum_i Q_s(\vphi_i)^2\ \dd s, \quad t > 0.
}
To prove so, the following lemma will suffice.
Recall that
\eq{
B_\delta = \Big\{\frac{1}{n}\sum_i\langle\vphi_i\rangle^2\leq\delta\Big\}.
}

\begin{lemma} \label{prep_prop_bad_lemma}
For any $T,\eps>0$, the following statements hold:
\begin{itemize}
\item[(a)] 
If $\beta$ is a point of differentiability for $p(\cdot)$, then there is a sequence of nonnegative random variables $(M_n)$ depending only on $\beta$, $T$, and $\eps$, such that
\eeq{ \label{M_condition}
\limsup_{n\to\infty} \E(M_n) \leq \eps,
}
and for every $f \in L^2(\Sigma_n)$, $t\in[0,\frac{T}{n}]$,
\eeq{ \label{prep_prop_bad_lemma_eq1}
\E|Q_t(f)^2-\langle f(\sigma)\rangle_t^2| \leq \E(\langle f(\sigma)^2\rangle M_n).
}
\item[(b)] 
There exist $\delta_1 = \delta_1(\beta,T,\eps)>0$ sufficiently small and $n_0 = n_0(\beta,T,\eps)$ sufficiently large, that for every $n\geq n_0$, $f\in L^2(\Sigma_n)$, $ t\in[0, \frac{T}{n}]$, and $\delta\in(0,\delta_1]$, we have
\eeq{ \label{prep_prop_bad_lemma_eq2}
\E\givenp[\big]{|Q_t(f)^2- \langle f(\sigma)\rangle^2|}{B_{\delta}} \leq \eps \E\langle f(\sigma)^2\rangle.
}
\end{itemize}
\end{lemma}

Before checking these facts, let us use them to prove Proposition~\ref{prep_prop_bad}.
The idea is to use the above sequence $M_n$ to control the differences $Q_t(\vphi_i)^2-\langle\vphi_i\rangle^2$ simultaneously across all $i$ and $t\in[0,\frac{T}{n}]$; this will allow us to prove \eqref{prep_prop_bad_eq1}.
On the other hand, \eqref{prep_prop_bad_lemma_eq2} shows that when $\langle\RR_{1,2}\rangle$ is small, $Q_t(\vphi_i)^2$ remains close to $Q_0(\vphi_i)^2=\langle\vphi_i\rangle^2$.
That this approximation holds uniformly over $t\in[0,\frac{T}{n}]$ will lead to \eqref{prep_prop_bad_eq2}.

\begin{proof}[Proof of Proposition~\ref{prep_prop_bad}]
First we prove part (a). 
Let $T,\eps>0$ be fixed.
From Lemma~\ref{prep_prop_bad_lemma}(a), we identify a sequence of random variables $(M_n)$ such that \eqref{prep_prop_bad_lemma_eq1} holds, and
\eeq{ \label{M_condition_applied}
\limsup_{n\to\infty} \E(M_n) \leq \eps^2.
}
Under our definition \eqref{I_def}, we have
\eq{
&\E\Big|I_{T/n} - \frac{1}{T/n}\int_0^{T/n} \frac{1}{n}\sum_i \langle \vphi_i\rangle_t^2\ \dd t\Big| \\
&\stackrel{\phantom{\mbox{\footnotesize{\eqref{prep_prop_bad_lemma_eq1}}}}}{=} \E\bigg|\frac{1}{T/n}\int_0^{T/n} \frac{1}{n}\sum_i [Q_t(\vphi_i)^2 - \langle \vphi_i\rangle_t^2]\ \dd t\bigg| \\
&\stackrel{\phantom{\mbox{\footnotesize{\eqref{prep_prop_bad_lemma_eq1}}}}}{\leq} \frac{1}{T/n}\int_0^{T/n} \frac{1}{n}\sum_i \E|Q_t(\vphi_i)^2 - \langle \vphi_i\rangle_t^2|\ \dd t \\
&\stackrel{\mbox{\footnotesize{\eqref{prep_prop_bad_lemma_eq1}}}}{\leq} \frac{1}{T/n}\int_0^{T/n} \frac{1}{n}\sum_i \E(\langle \vphi_i^2\rangle M_n)\ \dd t
\stackrel{\hspace{0.5ex}{\mbox{\footnotesize{\eqref{variance_assumption}}}}\hspace{0.5ex}}{=}  \E(M_n).
}
Now Markov's inequality and \eqref{M_condition_applied} together imply 
\eq{
\limsup_{n\to\infty} \P\bigg(\Big|I_{T/n} - \frac{1}{T/n}\int_0^{T/n} \frac{1}{n}\sum_i \langle \vphi_i\rangle_t^2\ \dd t\Big|\geq\eps\bigg) \leq \frac{\eps^2}{\eps} = \eps,
}
which completes the proof of (a).

Next we prove part (b).
Let $\eps_1,\eps_2 > 0$ be given.
Similar to above, for any $\delta>0$ we have
\eq{
&\E\givenp[\Big]{\big|I_{T/n} - \frac{1}{n}\sum_i \langle \vphi_i\rangle^2\big|}{B_\delta}
= \E\givenp[\Big]{\big|I_{T/n} - \frac{1}{T/n}\int_0^{T/n} \frac{1}{n}\sum_i \langle \vphi_i\rangle^2\ \dd t\big|}{B_\delta} \\
&\qquad\qquad\qquad\qquad\qquad\quad\leq \frac{1}{T/n}\int_0^{T/n} \frac{1}{n}\sum_i \E\givenp[\big]{|Q_t(\vphi_i)^2 - \langle \vphi_i\rangle^2|}{B_\delta}\, \dd t.
} 
From Lemma~\ref{prep_prop_bad_lemma}(b), we choose $\delta_1$ 
sufficiently small that \eqref{prep_prop_bad_lemma_eq2} holds for all $\delta\in(0,\delta_1]$, with $\eps = \eps_1\eps_2$.
We then have, for all $n$ sufficiently large,
\eq{
\E\givenp[\Big]{\big|I_{T/n} - \frac{1}{n}\sum_i \langle \vphi_i\rangle^2\big|}{B_\delta}
\leq \frac{1}{T/n}\int_0^{T/n} \frac{1}{n}\sum_i \eps_1\eps_2\E\langle\vphi_i^2\rangle\ \dd s
\stackrel{\mbox{\footnotesize\eqref{variance_assumption}}}{=} \eps_1\eps_2.
}
Then applying Markov's inequality yields \eqref{prep_prop_bad_eq2}.
\end{proof}

It now remains to prove Lemma~\ref{prep_prop_bad_lemma}.
To do so, we will make use of the following preparatory result, which in fact is the common thread between the proofs of Theorems~\ref{expected_overlap_thm} and~\ref{averages_squared}.
Let $\vc h = (h_i)_{i=1}^\infty$ be an independent copy of the disorder $\vc g$.
We will use $\E_{\vc h}$ and $\Var_{\vc h}$ to denote expectation and variance with respect to $\vc h$, conditional on $\vc g$.
All statements involving these conditional quantities will be almost sure with respect to $\P$, although we will not repeatedly write this.

%

\begin{lemma} \label{h_variance_lemma}
Recall the constant $\EEE_n$ from \eqref{positive_overlap}.
For any $t\geq0$, the following statements hold:
\begin{itemize}
\item[(a)] For any $f\in L^2(\Sigma_n)$,
\eq{
\Var_{\vc h}\langle f(\sigma)\e^{\frac{t}{\sqrt{n}}\sum_i h_i\vphi_i}\rangle \leq \e^{2t^2}\langle f(\sigma)^2\rangle\sqrt{\frac{1}{n}\sum_i \langle\vphi_i\rangle^2+2\EEE_n}.
}
\item[(b)] For any measurable $f : \Sigma_n \to [0,1]$,
\eq{
\Var_{\vc h}\langle f(\sigma)\e^{\frac{t}{\sqrt{n}}\sum_i h_i\vphi_i}\rangle \leq \e^{2t^2}\Big(\Big\langle f(\sigma)\frac{1}{n}\sum_i \vphi_i\langle \vphi_i\rangle\Big\rangle+2\EEE_n\Big).
}
\end{itemize}
\end{lemma}

\begin{proof}
For any $f\in L^2(\Sigma_n)$,
\eeq{ \label{general_f}
&\Var_{\vc h}\langle f(\sigma)\e^{\frac{t}{\sqrt{n}}\sum_i h_i\vphi_i}\rangle \\ 
&\stackrel{\phantom{\mbox{\footnotesize\eqref{freq_identity}}}}{=} \E_{\vc h}\langle f(\sigma^1)f(\sigma^2)\e^{\frac{t}{\sqrt{n}}\sum_{i} h_i(\vphi_i(\sigma^1)+\vphi_i(\sigma^2))}\rangle - \big(\E_{\vc h}\langle f(\sigma)\e^{\frac{t}{\sqrt{n}}\sum_i h_i\vphi_i}\rangle\big)^2\\
&\stackrel{\mbox{\footnotesize\eqref{freq_identity}}}{=} \e^{t^2}\big(\langle f(\sigma^1)f(\sigma^2)\e^{\frac{t^2}{n}\sum_i\vphi_i(\sigma^1)\vphi_i(\sigma^2)}\rangle - \langle f(\sigma)\rangle^2\big) \\
&\stackrel{\phantom{\mbox{\footnotesize\eqref{freq_identity}}}}{=} \e^{t^2}\langle f(\sigma^1)f(\sigma^2)(\e^{\frac{t^2}{n}\sum_i\vphi_i(\sigma^1)\vphi_i(\sigma^2)}-1)\rangle \\
&\stackrel{\phantom{\mbox{\footnotesize\eqref{freq_identity}}}}{\leq} \e^{t^2}\langle f(\sigma)^2\rangle\sqrt{\langle(\e^{\frac{t^2}{n}\sum_i\vphi_i(\sigma^1)\vphi_i(\sigma^2)}-1)^2\rangle}. \raisetag{3\baselineskip}
}
Now,  for all $x\in[-1,1]$, we have  $|\e^{t^2x}-1|\leq \e^{t^2}|x|$.
In particular, since
\eeq{ \label{unit_interval}
\Big|\frac{1}{n}\sum_i \vphi_i(\sigma^1)\vphi_i(\sigma^2)\Big|
\leq \frac{1}{n}\sqrt{\sum_i\vphi_i(\sigma^1)^2\sum_i\vphi_i(\sigma^2)^2} \stackrel{\mbox{\footnotesize\eqref{variance_assumption}}}{=} 1,
}
we see from \eqref{general_f} that
\eq{
\Var_{\vc h}\langle f(\sigma)\e^{\frac{t}{\sqrt{n}}\sum_i h_i\vphi_i}\rangle
&\stackrefp{positive_overlap}{\leq} \e^{2t^2}\langle f(\sigma)^2\rangle\sqrt{\Big\langle\Big(\frac{1}{n}\sum_i\vphi_i(\sigma^1)\vphi_i(\sigma^2)\Big)^2\Big\rangle} \\
&\stackref{positive_overlap}{\leq} \e^{2t^2}\langle f(\sigma)^2\rangle\sqrt{\Big\langle\frac{1}{n}\sum_i\vphi_i(\sigma^1)\vphi_i(\sigma^2)\Big\rangle+2\EEE_n} \\
&\stackrefp{positive_overlap}{=} \e^{2t^2}\langle f(\sigma)^2\rangle\sqrt{\frac{1}{n}\sum_i\langle\vphi_i\rangle^2+2\EEE_n}.
}
Alternatively, if $f:\Sigma_n\to[0,1]$, then we can use the equalities in \eqref{general_f} to write
\eq{
\Var_{\vc h}\langle f(\sigma)\e^{\frac{t}{\sqrt{n}}\sum_i h_i\vphi_i}\rangle 
&\stackrefp{positive_overlap}{=} \e^{t^2}\langle f(\sigma^1)f(\sigma^2)(\e^{\frac{t^2}{n}\sum_i\vphi_i(\sigma^1)\vphi_i(\sigma^2)}-1)\rangle \\
&\stackrefp{positive_overlap}{\leq} \e^{2t^2}\Big\langle f(\sigma^1)\Big|\frac{1}{n}\sum_i \vphi_i(\sigma^1)\vphi_i(\sigma^2)\Big|\Big\rangle \\
&\stackref{positive_overlap}{\leq} \e^{2t^2}\Big\langle f(\sigma^1)\Big(\frac{1}{n}\sum_i \vphi_i(\sigma^1)\vphi_i(\sigma^2)+2\EEE_n\Big)\Big\rangle \\
&\stackrefp{positive_overlap}{\leq} \e^{2t^2}\Big(\Big\langle f(\sigma^1)\frac{1}{n}\sum_i \vphi_i(\sigma^1)\langle\vphi_i(\sigma^2)\rangle\Big\rangle+2\EEE_n\Big).
}
\end{proof}

We are now ready to prove Lemma~\ref{prep_prop_bad_lemma}.

\begin{proof}[Proof of Lemma~\ref{prep_prop_bad_lemma}]
Let $f\in L^2(\Sigma_n)$ be arbitrary.
Recall the random variable $Q_t(f)$ defined in \eqref{Qt_def}.
Observe that for fixed $t\geq0$, $\e^{-t}\vc W(\e^{2t}-1)$ is equal in law to $\sqrt{1-\e^{-2t}}\vc h$, where $\vc h$ is an independent copy of $\vc g$.
Therefore, if we define
\eq{
X &\coloneqq  \langle f(\sigma)\e^{\beta\sqrt{1-\e^{-2t}}\sum_{i} h_i\vphi_i}\e^{\beta(\e^{-t}-1) H_n(\sigma)}\rangle, \\
Y &\coloneqq  \langle \e^{\beta\sqrt{1-\e^{-2t}}\sum_{i} h_i\vphi_i}\e^{\beta(\e^{-t}-1) H_n(\sigma)}\rangle, \\
X' &\coloneqq  \langle f(\sigma)\e^{\beta\sqrt{1-\e^{-2t}}\sum_{i} h_i\vphi_i}\rangle\e^{\beta (\e^{-t}-1)np'(\beta)}, \\
Y' &\coloneqq  \langle \e^{\beta\sqrt{1-\e^{-2t}}\sum_{i} h_i\vphi_i}\rangle\e^{\beta (\e^{-t}-1)np'(\beta)},
}
then
\eq{
(\langle f(\sigma)\rangle_t,Q_t(f)) \stackrel{\text{d}}{=} \Big(\frac{X}{Y},\frac{X'}{Y'}\Big).
}
Since the conclusions of Lemma~\ref{prep_prop_bad_lemma} depend only on marginal distributions at fixed $t\leq T/n$, it suffices to prove bounds of the form
\eeq{ \label{prep_prop_bad_lemma_eq1_new}
\E\Big|\Big(\frac{X}{Y}\Big)^2-\Big(\frac{X'}{Y'}\Big)^2\Big| \leq \E(\langle f(\sigma)^2\rangle M_n),
}
where $M_n$ satisfies \eqref{M_condition}, and
\eeq{ \label{prep_prop_bad_lemma_eq2_new}
\E\givenp[\bigg]{\Big|\Big(\frac{X'}{Y'}\Big)^2 - \langle f(\sigma)\rangle^2\Big|}{B_{\delta}} \leq \eps\E\langle f(\sigma)^2\rangle \quad \text{for all large enough $n$.}
}
So henceforth we fix $T,\eps>0$, and $t \in [0,\frac{T}{n}]$.
We will need the following four claims.
In checking these claims, we will frequently use the following inequality, which holds for any $c\geq0$:
\eeq{
n(1 - \e^{-ct}) \leq nct \leq cT. \label{frequent_ineq}
}

\begin{claim} \label{claim_denom_prime}
For any $q\in(-\infty,0]\cup[1,\infty)$,
\eeq{ \label{bad_prep_denom_bound_prime}
\E_{\vc h}[(Y')^{q}] \leq C(\beta,T,q).
}
\end{claim}

\begin{claim} \label{claim_num_prime}
For any $q\geq2$,
\eeq{ \label{bad_prep_num_bound_prime}
\E_{\vc h}[(X')^q] \leq C(\beta,T,q) \langle f(\sigma)^2\rangle^{q/2}.
}
\end{claim}

\begin{claim} \label{claim_denom}
Given any $q>0$, set $k = \floor{\log_2\frac{n}{qT}}$.
For all $n$ large enough that $k\geq1$,
\eeq{ \label{bad_prep_denom_bound}
\E_{\vc h}(Y^{-q}) \leq C(\beta,T,q)Z_n(\beta)^{-\frac{1}{2^k}}(Z_n(2\beta)^{\frac{1}{2^k}} +1).
}
\end{claim}

\begin{claim} \label{claim_var_bound}
For any even $q\geq 2$ and $\eps>0$, the following inequalities hold for all $n\geq(2q+1)T$: 
\eeq{ \label{bad_prep_var_bound}
&\E_{\vc h}[(X-X')^q]\\
&\leq C(\beta,T,q)\langle f(\sigma)^2\rangle^{q/2}\Big[C(\eps)\Big\langle\Big|p'(\beta)-\frac{H_n(\sigma)}{n} \Big|\Big\rangle + \eps Z_n(\beta)^{-\frac{2(q+1)T}{n}}\Big],
}
and thus
\eeq{ \label{bad_prep_var_bound_Y}
\E_{\vc h}[(Y-Y')^q]
&\leq C(\beta,T,q)\Big[C(\eps)\Big\langle\Big|p'(\beta)-\frac{H_n(\sigma)}{n} \Big|\Big\rangle + \eps Z_n(\beta)^{-\frac{(2q+1)T}{n}}\Big].
}
\end{claim}

Before proving the claims, we use them to obtain the desired statements.

\subsubsection{Proof of Lemma~\ref{prep_prop_bad_lemma}(a)}
First note that for any random variables $W$ and $Z$,
\eeq{ \label{square_inside_outside}
\E|W^2-Z^2| &= \E|(W-Z)^2 + 2Z(W-Z)| \\
&\leq \E[(W-Z)^2] + 2\sqrt{\E(Z^2)\E[(W-Z)^2]}.
}
Therefore,
\eeq{ \label{prep_for_full_bound}
&\E_{\vc h}\Big|\Big(\frac{X}{Y}\Big)^2-\Big(\frac{X'}{Y'}\Big)^2\Big| \\
&\leq \E_{\vc h}\Big[\Big(\frac{X}{Y} - \frac{X'}{Y'}\Big)^2\Big] + 2\sqrt{\E_{\vc h}\Big[\Big(\frac{X'}{Y'}\Big)^2\Big]\E_{\vc h}\Big[\Big(\frac{X}{Y} - \frac{X'}{Y'}\Big)^2\Big]} \\
&\leq\E_{\vc h}\Big[\Big(\frac{X}{Y} - \frac{X'}{Y'}\Big)^2\Big] + 2\big(\E_{\vc h}[(Y')^{-4}]\E_{\vc h}[(X')^4]\big)^\frac{1}{4}\sqrt{\E_{\vc h}\Big[\Big(\frac{X}{Y} - \frac{X'}{Y'}\Big)^2\Big]} \\
&\stackrel{\mbox{\footnotesize\eqref{bad_prep_denom_bound_prime},\eqref{bad_prep_num_bound_prime}}}{\leq}\E_{\vc h}\Big[\Big(\frac{X}{Y} - \frac{X'}{Y'}\Big)^2\Big]+C(\beta,T)\sqrt{\langle f(\sigma)^2\rangle}\sqrt{\E_{\vc h}\Big[\Big(\frac{X}{Y} - \frac{X'}{Y'}\Big)^2\Big]}.\hspace{0.4in} \raisetag{3.75\baselineskip}
}
Let $\delta$ be a positive number to be chosen later.
Anticipating the application of Claims~\ref{claim_denom} and~\ref{claim_var_bound}, we condense notation by defining
\eq{
V_n^{(q)} &= \big(Z_n(\beta)^{-\frac{1}{2^k}}(Z_n(2\beta)^{\frac{1}{2^k}} + 1) \big)^{2/q}, \quad \text{where} \quad k = \Big\lfloor\log_2 \frac{n}{qT}\Big\rfloor, \\
W_n^{(q)} &= \Big(C(\delta)\Big\langle\Big|p'(\beta)-\frac{H_n(\sigma)}{n}\Big|\Big\rangle + \delta Z_n(\beta)^{-\frac{2(q+1)T}{n}}\Big)^{2/q}.
}
Because of \eqref{prep_for_full_bound}, we seek a bound of the form
\eq{ 
&\E_{\vc h}\Big[\Big(\frac{X}{Y} - \frac{X'}{Y'}\Big)^2\Big]
=\E_{\vc h}\Big[\Big(\frac{X-X'}{Y} - \frac{X'}{Y'}\frac{Y-Y'}{Y}\Big)^2\Big] \\
&\stackrel{\phantom{\mbox{\footnotesize\eqref{bad_prep_denom_bound_prime}--\eqref{bad_prep_var_bound_Y}}}}{\leq} 2\E_{\vc h}\Big[\frac{(X-X')^2}{Y^2}+\frac{(X')^2}{(Y')^2}\frac{(Y-Y')^2}{Y^2}\Big] \\
&\stackrel{\phantom{\mbox{\footnotesize\eqref{bad_prep_denom_bound_prime}--\eqref{bad_prep_var_bound_Y}}}}{\leq} 2\big(\E_{\vc h}[Y^{-4}]\E_{\vc h}[(X-X')^{4}]\big)^{1/2} \\
&\phantom{\stackrel{\mbox{\footnotesize\eqref{bad_prep_denom_bound_prime}--\eqref{bad_prep_var_bound_Y}}}{\leq}}+ 2\big(\E_{\vc h}[(Y')^{-8}]\E_{\vc h}[(X')^{8}]\E_{\vc h}(Y^{-8})\E_{\vc h}[(Y-Y')^{8}]\big)^{1/4} \\
&\stackrel{\mbox{\footnotesize\eqref{bad_prep_denom_bound_prime}--\eqref{bad_prep_var_bound_Y}}}{\leq} C(\beta,T)\langle f(\sigma)^2\rangle(V_n^{(4)}W_n^{(4)}+V_n^{(8)}W_n^{(8)}).
}
Therefore, once we set
\eq{ 
M_n \coloneqq C(\beta,T)[(V_n^{(4)}W_n^{(4)}+V_n^{(8)}W_n^{(8)}) + (V_n^{(4)}W_n^{(4)}+V_n^{(8)}W_n^{(8)})^{1/2}]
}
and take expectation, \eqref{prep_for_full_bound} becomes
\eq{
\E\Big|\Big(\frac{X}{Y}\Big)^2-\Big(\frac{X'}{Y'}\Big)^2\Big| \leq \E(\langle f(\sigma)^2\rangle M_n),
}
which is exactly \eqref{prep_prop_bad_lemma_eq1_new}.
To complete the proof of Lemma~\ref{prep_prop_bad_lemma}(a), we need to show that given any $\eps >0$, we can choose $\delta$ sufficiently small that \eqref{M_condition} holds ($M_n$ depends on $\delta$ through $W_n^{(4)}$ and $W_n^{(8)}$).

Indeed, by Cauchy--Schwarz we have
\eeq{ \label{M_ineq}
\E(M_n) &\leq C(\beta,T)\bigg(\sqrt{\E[(V_n^{(4)})^2]\E[(W_n^{(4)})^2]} + \sqrt{\E[(V_n^{(8)})^2]\E[(W_n^{(8)})^2]} \\
&\phantom{\leq} + \sqrt{\sqrt{\E[(V_n^{(4)})^2]\E[(W_n^{(4)})^2]} + \sqrt{\E[(V_n^{(8)})^2]\E[(W_n^{(8)})^2]}}\,\bigg).
}
Next we observe that for $q \geq 4$ and $n$ sufficiently large such that $k = \floor{\log_2 \frac{n}{qT}}\geq1$,
\eeq{ \label{V_ineq}
\E[(V_n^{(q)})^2] 
&\stackrel{\phantom{\mbox{\footnotesize\eqref{first_moment},\eqref{negative_first_moment}}}}{\leq}  \Big(\E\big[Z_n(\beta)^{-\frac{1}{2^k}}(Z_n(2\beta)^{\frac{1}{2^k}} + 1) \big]\Big)^{4/q} \\
&\stackrel{\phantom{\mbox{\footnotesize\eqref{first_moment},\eqref{negative_first_moment}}}}{\leq}  \Big(\sqrt{\E[Z_n(\beta)^{-\frac{2}{2^{k}}}]\E[Z_n(2\beta)^{\frac{2}{2^k}}]}+\E[Z_n(\beta)^{-\frac{1}{2^{k}}}]\Big)^{4/q} \\
&\stackrel{\phantom{\mbox{\footnotesize\eqref{first_moment},\eqref{negative_first_moment}}}}{\leq}  \Big(\sqrt{\E[Z_n(\beta)^{-1}]^\frac{2}{2^{k}}\E[Z_n(2\beta)]^{\frac{2}{2^k}}}+\E[Z_n(\beta)^{-1}]^\frac{1}{2^{k}}\Big)^{4/q} \\
&\stackrel{\mbox{\footnotesize\eqref{first_moment},\eqref{negative_first_moment}}}{\leq}\Big(\sqrt{\e^\frac{\beta^2n}{2^{k}}\e^{\frac{4\beta^2 n}{2^k}}}+\e^\frac{\beta^2 n}{2^{k+1}}\Big)^{4/q} \\
&\stackrel{\phantom{\mbox{\footnotesize\eqref{first_moment},\eqref{negative_first_moment}}}}{\leq} \Big(\sqrt{\e^{\beta^2qT}\e^{4\beta^2 qT}}+\e^\frac{\beta^2 qT}{2}\Big)^{4/q} = C(\beta,T,q).
}
Meanwhile, if $q\geq4$ and $n\geq2(q+1)T$, then
\eq{
\E[(W_n^{(q)})^2] 
&\stackrel{\phantom{\mbox{\footnotesize\eqref{negative_first_moment}}}}{\leq} \Big(C(\delta)\E\Big\langle\Big|p'(\beta)-\frac{H_n(\sigma)}{n}\Big|\Big\rangle + \delta \E[Z_n(\beta)^{-\frac{2(q+1)T}{n}}]\Big)^{4/q} \\
&\stackrel{\phantom{\mbox{\footnotesize\eqref{negative_first_moment}}}}{\leq} \Big(C(\delta)\E\Big\langle\Big|p'(\beta)-\frac{H_n(\sigma)}{n}\Big|\Big\rangle + \delta \E[Z_n(\beta)^{-1}]^{\frac{2(q+1)T}{n}}\Big)^{4/q} \\
&\stackrel{\mbox{\footnotesize\eqref{negative_first_moment}}}{\leq}  \Big(C(\delta)\E\Big\langle\Big|p'(\beta)-\frac{H_n(\sigma)}{n}\Big|\Big\rangle + \delta \e^{\beta^2(q+1)T}\Big)^{4/q}.
}
By Lemma~\ref{lemma:step_1_2}, the previous display shows
\eq{
\limsup_{n\to\infty} \E[(W_n^{(q)})^2] &\leq \delta^{4/q}\e^\frac{4\beta^2(q+1)T}{q} = C(\beta,T,q)\delta^{4/q}.
}
In light of \eqref{M_ineq} and \eqref{V_ineq}, it is clear from this inequality that $\delta$ can be chosen sufficiently small that \eqref{M_condition} holds.

\subsubsection{Proof of Lemma~\ref{prep_prop_bad_lemma}(b)}
To establish \eqref{prep_prop_bad_lemma_eq2_new}, it will be easier to replace $X'/Y'$ by $X''/Y''$, where
\eq{
X'' &\coloneqq \frac{X'}{\e^{\frac{\beta^2}{2}(1-\e^{-2t})n}\e^{\beta(\e^{-t}-1)np'(\beta)}} = \frac{\langle f(\sigma)\e^{\beta\sqrt{1-\e^{-2t}}\sum_ih_i\vphi_i}\rangle}{\e^{\frac{\beta^2}{2}(1-\e^{-2t})n}}, \\
\qquad Y'' &\coloneqq \frac{Y'}{\e^{\frac{\beta^2}{2}(1-\e^{-2t})n}\e^{\beta(\e^{-t}-1)np'(\beta)}} = \frac{\langle \e^{\beta\sqrt{1-\e^{-2t}}\sum_ih_i\vphi_i}\rangle}{\e^{\frac{\beta^2}{2}(1-\e^{-2t})n}}.
}
By Lemma~\ref{h_variance_lemma}(a),
\eq{
\Var_{\vc h}\langle f(\sigma)\e^{\beta\sqrt{1-\e^{-2t}}\sum_i h_i\vphi_i}\rangle
&\leq \e^{2\beta^2(1-\e^{-2t})n}\langle f(\sigma)^2\rangle\sqrt{\frac{1}{n}\sum_i \langle\vphi_i\rangle^2+2\EEE_n},
}
and so
\eeq{ \label{observation_1}
\Var_{\vc h}(X'') &\stackrel{\phantom{\mbox{\footnotesize\eqref{frequent_ineq}}}}{\leq} \e^{\beta^2(1-\e^{-2t})n}\langle f(\sigma)^2\rangle\sqrt{\frac{1}{n}\sum_i\langle\vphi_i\rangle^2+2\EEE_n} \\
&\stackrel{{\mbox{\footnotesize\eqref{frequent_ineq}}}}{\leq}
 C(\beta,T)\langle f(\sigma)^2\rangle\sqrt{\frac{1}{n}\sum_i\langle\vphi_i\rangle^2+2\EEE_n},
 }
as well as
\eq{
\Var_{\vc h}(Y'') &\leq C(\beta,T)\sqrt{\frac{1}{n}\sum_i\langle\vphi_i\rangle^2+2\EEE_n}.
}
Because
\eq{
\E_{\vc h} \langle f(\sigma) \e^{\beta\sqrt{1-\e^{-2t}}\sum_i h_i\vphi_i}\rangle
\stackrel{{\mbox{\footnotesize\eqref{freq_identity}}}}{=} \e^{\frac{\beta^2 }{2}(1-\e^{-2t})n}\langle f(\sigma)\rangle,
}
we have $\E_{\vc h}(Y'') = 1$ and can thus apply Chebyshev's inequality to obtain
\eeq{ \label{observation_2}
\P_{\vc h}(|Y'' - 1|\geq\theta) \leq \frac{C(\beta,T)}{\theta^2} \sqrt{\frac{1}{n}\sum_i\langle\vphi_i\rangle^2+2\EEE_n} \quad \text{for any $\theta>0$.}
}
We will use these inequalities in the following bound:
\eeq{ \label{observation_0}
&\E_{\vc h}\Big[\Big(\frac{X'}{Y'} - \langle f(\sigma)\rangle\Big)^2\Big] \\
&= \E_{\vc h}\Big[\Big(\frac{X''}{Y''} - \langle f(\sigma)\rangle\Big)^2\Big] \\
&= \E_{\vc h}\Big[\Big(\frac{X''}{Y''}(1 - Y'') + X'' - \langle f(\sigma)\rangle\Big)^2\Big] \\
&\leq 2\E_{\vc h}\Big[\Big(\frac{X''}{Y''}\Big)^2(Y''-1)^2 + \big(X'' - \langle f(\sigma)\rangle\big)^2\Big] \\
&\leq  2\E_{\vc h}\Big[\Big(\frac{X''}{Y''}\Big)^2(\theta^2+\one_{\{|Y''-1|\geq\theta\}}(Y''-1)^2) + \big(X'' - \langle f(\sigma)\rangle\big)^2\Big] \\
&\leq 2\big(\E_{\vc h}[(Y'')^{-8}]\E_{\vc h}[(X'')^8])^{1/4}\sqrt{\E_{\vc h}\big[\big(\theta^2+\one_{\{|Y''-1|\geq\theta\}}(Y''-1)^2\big)^2\big]} \\
&\phantom{\leq}+ 2\Var_{\vc h}(X'') \\
&\leq 2\sqrt{2}\big(\E_{\vc h}[(Y'')^{-8}]\E_{\vc h}[(X'')^8])^{1/4}\sqrt{\theta^4 + \sqrt{\P_{\vc h}(|Y''-1|\geq\theta)\E_{\vc h}[(Y''-1)^8]}} \\
&\phantom{\leq}+ 2\Var_{\vc h}(X'').
\raisetag{7\baselineskip}
}
Now,
\eeq{ \label{observation_3}
\E_{\vc h}[(Y'')^{-8}] = \frac{\E_{\vc h}[(Y')^{-8}]}{\e^{-4\beta^2(1-\e^{-2t})n}\e^{-8\beta(\e^{-t}-1)np'(\beta)}} 
\stackrel{\mbox{\footnotesize\eqref{frequent_ineq},\eqref{bad_prep_denom_bound_prime}}}{\leq} C(\beta,T),
}
and
\eeq{ \label{observation_4}
\E_{\vc h}[(X'')^8] = \frac{\E_{\vc h}[(X')^8]}{\e^{4\beta^2(1-\e^{-2t})n}\e^{8\beta(\e^{-t}-1)np'(\beta)}}
\stackrel{\mbox{\footnotesize\eqref{frequent_ineq},\eqref{bad_prep_num_bound_prime}}}{\leq}  C(\beta,T)\langle f(\sigma)^2\rangle^4.
}
In addition,
\eeq{ \label{observation_5}
\E_{\vc h}[(Y''-1)^8] 
&\stackrel{\phantom{\mbox{\footnotesize\eqref{frequent_ineq},\eqref{bad_prep_denom_bound_prime}}}}{\leq} 2^4(\E_{\vc h}[(Y'')^8]+1)  \\
&\stackrel{\phantom{\mbox{\footnotesize\eqref{frequent_ineq},\eqref{bad_prep_denom_bound_prime}}}}{=} 2^4\Big( \frac{\E_{\vc h}[(Y')^{8}]}{\e^{4\beta^2(1-\e^{-2t})n}\e^{8\beta(\e^{-t}-1)np'(\beta)}}+1\Big)  \\
&\stackrel{\mbox{\footnotesize\eqref{frequent_ineq},\eqref{bad_prep_denom_bound_prime}}}{\leq} C(\beta,T).
}
Using \eqref{observation_1}, \eqref{observation_2}, and \eqref{observation_3}--\eqref{observation_5} in \eqref{observation_0}, we find
\eq{
&\E_{\vc h}\Big[\Big(\frac{X'}{Y'} - \langle f(\sigma)\rangle\Big)^2\Big] \\
&\leq C(\beta,T)\langle f(\sigma)^2\rangle\sqrt{\theta^4+\frac{C(\beta,T)}{\theta}\Big(\frac{1}{n}\sum_i\langle\vphi_i^2\rangle+2\EEE_n\Big)^{1/4}} \\
&\phantom{\leq}+C(\beta,T)\langle f(\sigma)^2\rangle \sqrt{\frac{1}{n}\sum_i \langle\vphi_i\rangle^2+2\EEE_n}.
}
In particular, for any $\delta>0$ and $n$ large enough that $\EEE_n\leq\delta/2$,
\eq{
\one_{B_\delta}\E_{\vc h}\Big[\Big(\frac{X'}{Y'} - \langle f(\sigma)\rangle\Big)^2\Big] \leq \one_{B_\delta}C(\beta,T)\langle f(\sigma)^2\rangle\Big(\sqrt{\theta^4+\theta^{-1}(2\delta)^{1/4}}+\sqrt{2\delta}\Big),
}
and so \eqref{square_inside_outside} implies
\eq{
&\one_{B_\delta}\E_{\vc h}\Big|\Big(\frac{X'}{Y'}\Big)^2 - \langle f(\sigma)\rangle^2\Big| \\
&\leq \one_{B_\delta}\E_{\vc h}\Big[\Big(\frac{X'}{Y'} - \langle f(\sigma)\rangle\Big)^2\Big]
+ 2\one_{B_\delta}\sqrt{\langle f(\sigma)\rangle^2\E_{\vc h}\Big[\Big(\frac{X'}{Y'} - \langle f(\sigma)\rangle\Big)^2\Big]} \\
&\leq \one_{B_\delta}C(\beta,T)\langle f(\sigma)^2\rangle\Big(\sqrt{\theta^4+\theta^{-1}\delta^{1/4}}+\sqrt{\delta} 
+ \sqrt{\sqrt{\theta^4+\theta^{-1}\delta^{1/4}}+\sqrt{\delta}}\, \Big).
}
Given $\eps>0$, we choose $\theta$ and $\delta$ small enough (in that order, and depending only on $\beta$, $T$, and $\eps$) so that the rightmost expression above is at most $\one_{B_\delta}\eps\langle f(\sigma)^2\rangle$. 
Moreover, it is clear that once $\theta$ and $\delta$ are chosen, $\one_{B_\delta}$ could be replaced by $\one_{B_{\delta'}}$ for any $\delta'\in(0,\delta)$, and the rightmost expression will be bounded from above by $\one_{B_{\delta'}}\eps\langle f(\sigma)^2\rangle$.
Taking expectations on both sides yields \eqref{prep_prop_bad_lemma_eq2_new}.

\subsubsection{Proof of Claim~\ref{claim_denom_prime}}
Assume $q\leq0$ or $q\geq1$.
Using Jensen's inequality, we have
\eq{ 
\E_{\vc h}[(Y')^{q}] &\stackrel{\phantom{\mbox{\footnotesize\eqref{frequent_ineq}}}}{=} \e^{q\beta (\e^{-t}-1)np'(\beta)}\E_{\vc h}\big[\langle\e^{\beta\sqrt{1-\e^{-2t}}\sum_{i} h_i\vphi_i}\rangle^{q}\big] \\
&\stackrel{\phantom{\mbox{\footnotesize\eqref{frequent_ineq}}}}{\leq} \e^{q\beta (\e^{-t}-1)np'(\beta)}\E_{\vc h}\langle\e^{q\beta\sqrt{1-\e^{-2t}}\sum_{i} h_i\vphi_i}\rangle \\
&\stackrel{\hspace{0.5ex}{\mbox{\footnotesize\eqref{freq_identity}}}\hspace{0.5ex}}{=} \e^{q\beta (\e^{-t}-1)np'(\beta)}\e^{\frac{q^2\beta^2}{2}(1-\e^{-2t})n} \\
&\stackrel{{\mbox{\footnotesize\eqref{frequent_ineq}}}}{\leq} C(\beta,T,q).
}

\subsubsection{Proof of Claim~\ref{claim_num_prime}}
Assume $q\geq2$.
By Cauchy--Schwarz and Jensen's inequality, we have
\eq{
\E_{\vc h}[(X')^q]
&\stackrel{\phantom{\mbox{\footnotesize\eqref{frequent_ineq}}}}{=} \e^{q\beta(\e^{-t}-1)np'(\beta)}\E_{\vc h}(\langle f(\sigma) \e^{\beta\sqrt{1-\e^{-2t}}\sum_ih_i\vphi_i}\rangle^q) \\
&\stackrel{\phantom{\mbox{\footnotesize\eqref{frequent_ineq}}}}{\leq} \e^{q\beta(\e^{-t}-1)np'(\beta)}\E_{\vc h}(\langle f(\sigma)^2\rangle^{q/2} \langle\e^{2\beta\sqrt{1-\e^{-2t}}\sum_ih_i\vphi_i}\rangle^{q/2}) \\
&\stackrel{\phantom{\mbox{\footnotesize\eqref{frequent_ineq}}}}{\leq}\e^{q\beta(\e^{-t}-1)np'(\beta)} \langle f(\sigma)^2\rangle^{q/2}\E_{\vc h} \langle\e^{q\beta\sqrt{1-\e^{-2t}}\sum_ih_i\vphi_i}\rangle \\
&\stackrel{\hspace{0.5ex}{\mbox{\footnotesize\eqref{freq_identity}}}\hspace{0.5ex}}{=}\e^{q\beta (\e^{-t}-1)np'(\beta)} \langle f(\sigma)^2\rangle^{q/2}\e^{\frac{q^2\beta^2}{2} (1-\e^{-2t})n} \\
&\stackrel{\mbox{\footnotesize\eqref{frequent_ineq}}}{\leq} C(\beta,T,q)\langle f(\sigma)^2\rangle^{q/2}.
}

\subsubsection{Proof of Claim~\ref{claim_denom}}
Assume $q>0$.
By Jensen's inequality,
\eeq{ \label{pre_squish_expectation}
\E_{\vc h}(Y^{-q}) &\stackrel{\phantom{\mbox{\footnotesize\eqref{frequent_ineq}}}}{=}  
\E_{\vc h}[\langle \e^{\beta\sqrt{1-\e^{-2t}}\sum_ih_i\vphi_i}\e^{\beta(\e^{-t}-1)H_n(\sigma)}\rangle^{-q}]\\
&\stackrel{\phantom{\mbox{\footnotesize\eqref{frequent_ineq}}}}{\leq} \E_{\vc h}\langle \e^{-q\beta \sqrt{1-\e^{-2t}}\sum_ih_i\vphi_i}\e^{q\beta (1-\e^{-t})H_n(\sigma)}\rangle \\
&\stackrel{\hspace{0.5ex}{\mbox{\footnotesize\eqref{freq_identity}}}\hspace{0.5ex}}{=} \e^{\frac{q^2\beta^2}{2}(1-\e^{-2t})n}\langle \e^{\beta q(1-\e^{-t})H_n(\sigma)}\rangle \\
&\stackrel{\mbox{\footnotesize\eqref{frequent_ineq}}}{\leq} C(\beta,T,q)\langle \e^{\beta q(1-\e^{-t})H_n(\sigma)}\rangle.
}
Recall that $k = \floor{\log_2\frac{n}{qT}}$, and we assume $k\geq1$.
By \eqref{frequent_ineq},
\eq{
q(1-\e^{-t}) \leq \frac{qT}{n} = \frac{1}{2^{\log_2\frac{n}{qT}}}  \leq \frac{1}{2^k},
}
which implies
\eeq{ \label{squish_expectation}
\langle \e^{\beta q(1-\e^{-t})H_n(\sigma)}\rangle \leq \langle \e^{-\beta H_n(\sigma)/2^{k}}\rangle + \langle \e^{\beta H_n(\sigma)/2^{k}}\rangle.
}
Repeated applications of Cauchy--Schwarz yield
\eeq{ \label{repeated_CS}
\langle \e^{\beta H_n(\sigma)/2^k}\rangle 
&= \frac{E_n(\e^{\beta(1+\frac{1}{2^k})H_n(\sigma)})}{E_n(\e^{\beta H_n(\sigma)})} \\
&= \frac{E_n(\e^{\frac{\beta}{2}H_n(\sigma)}\e^{\beta(\frac{1}{2}+\frac{1}{2^k})H_n(\sigma)})}{E_n(\e^{\beta H_n(\sigma)})}  \\
&\leq \frac{\sqrt{E_n(\e^{\beta H_n(\sigma)})E_n(\e^{\beta(1+\frac{1}{2^{k-1}})H_n(\sigma)})}}{E_n(\e^{\beta H_n(\sigma)})} \\
&\leq \frac{\sqrt{E_n(\e^{\beta H_n(\sigma)})\sqrt{E_n(\e^{\beta H_n(\sigma)})E_n(\e^{\beta(1+\frac{1}{2^{k-2}})H_n(\sigma)})}}}{E_n(\e^{\beta H_n(\sigma)})} \\
&\hspace{1.3ex}\vdots \\
&\leq E_n(\e^{\beta H_n(\sigma)})^{-1+\sum_{i=1}^{k} \frac{1}{2^i}}E_n(\e^{2\beta H_n(\sigma)})^{\frac{1}{2^k}} \\
&= Z_n(\beta)^{-\frac{1}{2^k}}Z_n(2\beta)^{\frac{1}{2^k}}.
}
By similar manipulations,
\eeq{ \label{repeated_CS_2}
\langle \e^{-\beta H_n(\sigma)/2^k}\rangle \leq Z_n(\beta)^{-\frac{1}{2^k}}Z_n(0)^{\frac{1}{2^k}} = Z_n(\beta)^{-\frac{1}{2^k}}.
}
Together, \eqref{pre_squish_expectation}--\eqref{repeated_CS_2} yield \eqref{bad_prep_denom_bound}.

\subsubsection{Proof of Claim~\ref{claim_var_bound}}
Assume $q\geq2$ is even. 
By Cauchy--Schwarz and Jensen's inequality, we have
\eeq{ \label{next_1}
&\E_{\vc h}[(X-X')^q] \\
&=\E_{\vc h}[\langle f(\sigma)\e^{\beta\sqrt{1-\e^{-2t}}\sum_{i} h_i\vphi_i}(\e^{\beta(\e^{-t}-1) H_n(\sigma)}-\e^{\beta (\e^{-t}-1)np'(\beta)})\rangle^q] \\
&\leq \E_{\vc h}\big[\langle f(\sigma)^2\rangle^{q/2}\langle\e^{2\beta\sqrt{1-\e^{-2t}}\sum_{i} h_i\vphi_i}(\e^{\beta(\e^{-t}-1) H_n(\sigma)}-\e^{\beta (\e^{-t}-1)np'(\beta)})^2\rangle^{q/2}\big] \\
&\leq\langle f(\sigma)^2\rangle^{q/2}\e^{q\beta (\e^{-t}-1)np'(\beta)}\\
&\phantom{\leq}\quad\times\E_{\vc h}\langle\e^{q\beta\sqrt{1-\e^{-2t}}\sum_ih_i\vphi_i}(\e^{\beta(1-\e^{-t})(np'(\beta)- H_n(\sigma))}-1)^q\rangle \\
&\stackrel{\mbox{\footnotesize\eqref{freq_identity}}}{=} \langle f(\sigma)^2\rangle^{q/2}\e^{q\beta (\e^{-t}-1)np'(\beta)}\e^{\frac{q^2\beta^2}{2}(1-\e^{-2t})n}\langle(\e^{\beta(1-\e^{-t})(np'(\beta)- H_n(\sigma))}-1)^q\rangle \\
&\stackrel{\mbox{\footnotesize\eqref{frequent_ineq}}}{\leq} C(\beta,T,q)\langle f(\sigma)^2\rangle^{q/2} \langle(\e^{\beta(1-\e^{-t})(np'(\beta)- H_n(\sigma))}-1)^q\rangle. \raisetag{4.5\baselineskip}
}
For any $L>0$, we have the inequality $(\e^x - 1)^q \leq C(L,q)|x|$ 
for all $x\leq L$.
Hence
\eeq{ \label{next_2}
&\langle(\e^{\beta(1-\e^{-t})(np'(\beta)- H_n(\sigma))}-1)^q\rangle \\
&\stackrel{\phantom{\mbox{\footnotesize\eqref{frequent_ineq}}}}{\leq} C(L,q)\beta (1-\e^{-t})n\Big\langle\Big|p'(\beta)-\frac{H_n(\sigma)}{n}\Big|\Big\rangle \\
&\phantom{\stackrel{\mbox{\footnotesize\eqref{frequent_ineq}}}{\leq}} + \langle(\e^{\beta(1-\e^{-t})(np'(\beta)- H_n(\sigma))}-1)^q\one_{\{\beta(1-\e^{-t})(np'(\beta)- H_n(\sigma)) > L\}}\rangle \\
&\stackrel{{\mbox{\footnotesize\eqref{frequent_ineq}}}}{\leq} C(\beta,T,L,q) \Big\langle\Big|p'(\beta)-\frac{H_n(\sigma)}{n}\Big|\Big\rangle \\
&\phantom{\stackrel{\mbox{\footnotesize\eqref{frequent_ineq}}}{\leq}} + \langle(\e^{\beta(1-\e^{-t})(np'(\beta)- H_n(\sigma))}-1)^q\one_{\{\beta(1-\e^{-t})(np'(\beta)- H_n(\sigma)) > L\}}\rangle.
}
Assume $L\geq 2\beta T p'(\beta)$ so that whenever
\eq{
\beta(1-\e^{-t})\big(np'(\beta)- H_n(\sigma)\big) &> L \geq 2\beta Tp'(\beta) \stackrel{{\mbox{\footnotesize\eqref{frequent_ineq}}}}{\geq}2\beta (1-\e^{-t})n p'(\beta),
}
it follows that
\eq{
-\beta(1-\e^{-t}) H_n(\sigma) &> \beta (1-\e^{-t})np'(\beta) \\
\Rightarrow \quad -2\beta(1-\e^{-t}) H_n(\sigma)&> \beta(1-\e^{-t})\big(np'(\beta)- H_n(\sigma)\big) > L\geq 0 \\
\stackrel{\mbox{\footnotesize\eqref{frequent_ineq}}}{\Rightarrow} \quad -\frac{2\beta T}{n} H_n(\sigma) &> \beta(1-\e^{-t})\big(np'(\beta)- H_n(\sigma)
\big) > L \geq 0.
}
We thus have
\eeq{ \label{next_3}
&\langle(\e^{\beta(1-\e^{-t})(np'(\beta)-H_n(\sigma))}-1)^q\one_{\{\beta(1-\e^{-t})(np'(\beta)- H_n(\sigma)) > L\}}\rangle \\
&\leq \langle \e^{-\frac{2q\beta T}{n} H_n(\sigma)}\one_{\{-\frac{2\beta T}{n} H_n(\sigma) > L\}}\rangle \\
&\leq \e^{-L}\langle \e^{-\frac{2(q+1)\beta T}{n} H_n(\sigma)}\rangle \\
&= \e^{-L}\frac{E_n[\e^{\beta(1-\frac{2(q+1)T}{n}) H_n(\sigma)}]}{E_n[\e^{\beta H_n(\sigma)}]} \\
&\leq \e^{-L} \frac{(E_n[\e^{\beta H_n(\sigma)}])^{1-\frac{2(q+1)T}{n}}}{E_n[\e^{\beta H_n(\sigma)}]} \\
&= \e^{-L} Z_n(\beta)^{-\frac{2(q+1)T}{n}}.
}
Combining \eqref{next_1}--\eqref{next_3}, we have now shown that 
\eq{ 
\E_{\vc h}[(X-X')^q]&\leq \langle f(\sigma)^2\rangle^{q/2}\Big[C(\beta,T,L,q)\Big\langle\Big|p'(\beta)-\frac{H_n(\sigma)}{n} \Big|\Big\rangle \\
&\phantom{\leq f(\sigma)^2\rangle^{q/2}\Big[C}+ C(\beta,T,q)\e^{-L}Z_n(\beta)^{-\frac{2(q+1)T}{n}}\Big].
}
Finally, given $\eps>0$, we choose $L$ large enough that $\e^{-L} \leq \eps$, thereby producing \eqref{bad_prep_var_bound}.
Then \eqref{bad_prep_var_bound_Y} is the special case when $f\equiv1$.
\end{proof}



\section{Proof of Theorem~\ref{expected_overlap_thm}} \label{proof_2}

In this section, we consider perturbations to the environment of the form
\eq{
\vc g^{(k)} \coloneqq \vc g + \frac{1}{\sqrt{n}}\sum_{j=1}^k \vc h^{(j)}, \quad k\geq0,
}
where the $\vc h^{(j)}$'s are independent copies of $\vc g$.
An important observation is that
\eeq{ \label{temperature_equivalence}
 \vc g^{(k)}\stackrel{\text{d}}{=} \sqrt{1+\frac{k}{n}}\,\vc g \quad \Rightarrow \quad
 \mu_{n, \vc g^{(k)}}^\beta \stackrel{\text{d}}{=} \mu_{n,\vc g}^{\beta\sqrt{1+\frac{k}{n}}}.
}
We will continue to use $\E$ to denote expectation with respect to $\vc g$ and the $\vc h^{(k)}$'s jointly, whereas $\E_{\vc h^{(k)}}$ will denote expectation with respect to $\vc h^{(k)}$ conditional on $\vc g$ and $\vc h^{(j)}$, $1 \leq j \leq k-1$.
As before, all statements involving $\E_{\vc h^{(k)}}$ and $\Var_{\vc h^{(k)}}$ are to be interpreted as almost sure statements.

As in Section~\ref{prep_section}, $\langle\cdot\rangle_\beta$ will denote expectation with respect to $\mu_{n,\vc g}^\beta$.
On the other hand, we will write $\llangle\cdot\rrangle_k$ to denote expectation under the measure $\mu_{n,\vc g^{(k)}}^\beta$, where the dependence on $\beta$ is understood.
That is,
\eeq{ \label{k_induction}
{\llangle f(\sigma)\rrangle}_{k} 
&\coloneqq \frac{E_n(f(\sigma)\e^{\beta(H_n(\sigma)+\frac{1}{\sqrt{n}}\sum_{j=1}^k\sum_ih_i^{(j)}\vphi_i)})}{E_n(\e^{\beta(H_n(\sigma)+\frac{1}{\sqrt{n}}\sum_{j=1}^k\sum_ih_i^{(j)}\vphi_i)})}  \\
&= \frac{{\llangle f(\sigma)\e^{\frac{\beta}{\sqrt{n}}\sum_ih_i^{(k)}\vphi_i}\rrangle}_{k-1}}{{\llangle \e^{\frac{\beta}{\sqrt{n}}\sum_ih_i^{(k)}\vphi_i}\rrangle}_{k-1}}.
}
For $\delta >0$, 
define the set
\eq{
\AA_{\delta,k} \coloneqq \Big\{\sigma^1\in\Sigma_n : \frac{1}{n}\sum_{i}\vphi_i(\sigma^1)\llangle\vphi_i(\sigma^2)\rrangle_k\leq\delta\Big\},
}
where $\AA_{\delta,0} = \AA_\delta$ is the set under consideration in Theorem~\ref{expected_overlap_thm}, whose proof will rely on Propositions~\ref{pre_iteration_1} and~\ref{pre_iteration_2} below.

\begin{prop} \label{pre_iteration_1}
For any $\delta_0>0$, there exists $n_0 = n_0(\delta_0)$ such that for all $n\geq n_0$, $k\geq 1$, and $\delta\geq\delta_0$,
\eq{
\E\llangle \one_{\AA_{\delta,k-1}}\rrangle_k
\leq \E\llangle\one_{\AA_{\delta^{1/4},k}}\rrangle_k + C(\beta)\delta.
}
\end{prop}

\begin{proof}
For any measurable $f:\Sigma_n\to[0,1]$, an application of \eqref{k_induction}, followed by Cauchy--Schwarz and Jensen's inequality, gives
\eq{
\llangle f(\sigma) \rrangle_k 
&\leq \frac{\sqrt{\llangle f(\sigma)^2\rrangle_{k-1}}\sqrt{\llangle\e^{\frac{2\beta}{\sqrt{n}}\sum_ih_i^{(k)}\vphi_i}\rrangle_{k-1}}}{\llangle\e^{\frac{\beta}{\sqrt{n}}\sum_ih_i^{(k)}\vphi_i}\rrangle_{k-1}} \\
&\leq \sqrt{\llangle f(\sigma)\rrangle_{k-1}}\sqrt{\llangle\e^{\frac{2\beta}{\sqrt{n}}\sum_ih_i^{(k)}\vphi_i}\rrangle_{k-1}}\llangle\e^{\frac{-\beta}{\sqrt{n}}\sum_ih_i^{(k)}\vphi_i}\rrangle_{k-1}.
}
So we define the random variable
\eq{
X \coloneqq \sqrt{2\llangle\e^{\frac{2\beta}{\sqrt{n}}\sum_ih_i^{(k)}\vphi_i}\rrangle_{k-1}}\llangle\e^{\frac{-\beta}{\sqrt{n}}\sum_ih_i^{(k)}\vphi_i}\rrangle_{k-1},
}
and consider, for fixed $\sigma^1$, the function $f_{\sigma^1}(\sigma^2) = 0\vee\frac{1}{n}\sum_i \vphi_i(\sigma^1)\vphi_i(\sigma^2)$.
By \eqref{unit_interval}, $f_{\sigma^1}$ is $[0,1]$-valued, and \eqref{positive_overlap} implies
\eq{
f_{\sigma^1}(\sigma^2) \leq \EEE_n + \frac{1}{n}\sum_i \vphi_i(\sigma^1)\vphi_i(\sigma^2).
}
So the above estimate shows
\eq{
\frac{1}{n}\sum_i \vphi_i(\sigma^1)\llangle\vphi_i(\sigma^2)\rrangle_k
&\leq \llangle f_{\sigma^1}(\sigma^2)\rrangle_k \\
&\leq \frac{X}{\sqrt{2}}\sqrt{\EEE_n+\frac{1}{n}\sum_i \vphi_i(\sigma^1)\llangle\vphi_i(\sigma^2)\rrangle_{k-1}}.
}
In particular, when $n$ is sufficiently large that $\EEE_n\leq\delta$, 
\eq{
\one_{\AA_{\delta,k-1}}(\sigma^1)\frac{1}{n}\sum_i \vphi_i(\sigma^1)\llangle\vphi_i(\sigma^2)\rrangle_k\leq X\sqrt{\delta}.
}
We have thus shown
$\AA_{\delta,k-1} \subset \AA_{X\sqrt{\delta},k}$, which implies
\eq{
\E\llangle \one_{\AA_{\delta,k-1}}\rrangle_k \leq \E\llangle \one_{\AA_{X\sqrt{\delta},k}}\rrangle_k
\leq \E\llangle\one_{\AA_{t\sqrt{\delta},k}}\rrangle_k + \P(X > t) \quad \text{for any $t>0$},
}
where in the second inequality we have used the fact that if $\delta_1\leq\delta_2$, then $\AA_{\delta_1,k}\subset \AA_{\delta_2,k}$.
To handle the last term in the above display, we note that for any $p\geq1$,
\eq{
\P(X > t) 
&= \P(X^p > t^p) \\ 
&\leq t^{-p}\E(X^p) \\
&= t^{-p}2^{p/2}\E\big[\llangle\e^{\frac{2\beta}{\sqrt{n}}\sum_ih_i^{(k)}\vphi_i}\rrangle_{k-1}^{p/2}\llangle\e^{\frac{-\beta}{\sqrt{n}}\sum_ih_i^{(k)}\vphi_i}\rrangle_{k-1}^p] \\
&\leq t^{-p}2^{p/2}\sqrt{\E[\llangle\e^{\frac{2\beta}{\sqrt{n}}\sum_ih_i^{(k)}\vphi_i}\rrangle_{k-1}^p]\cdot\E[\llangle\e^{\frac{-\beta}{\sqrt{n}}\sum_ih_i^{(k)}\vphi_i}\rrangle_{k-1}^{2p}]} \\
&\leq t^{-p}2^{p/2}\sqrt{\E\llangle\e^{\frac{2\beta p}{\sqrt{n}}\sum_ih_i^{(k)}\vphi_i}\rrangle_{k-1}\cdot\E\llangle\e^{\frac{-2\beta p}{\sqrt{n}}\sum_ih_i^{(k)}\vphi_i}\rrangle_{k-1}}.
}
Now, for any $\theta\in\R$ and any $k\geq1$,
\eq{
\E\llangle \e^{\frac{\theta}{\sqrt{n}}\sum_ih_i^{(k)}\vphi_i}\rrangle _{k-1}
= \E\big[ \E_{\vc h^{(k)}}\llangle\e^{\frac{\theta}{\sqrt{n}}\sum_ih_i^{(k)}\vphi_i}\rrangle_{k-1}\big]
\stackrel{\mbox{\footnotesize\eqref{freq_identity}}}{=} \e^{\frac{\theta^2}{2}}.
}
Hence
\eq{
\P(X>t) \leq t^{-p}2^{p/2}\e^{2\beta^2p^2}.
}
Choosing $t=\delta^{-1/4}$ and $p = 4$, we arrive at
\eq{
\E\llangle \one_{\AA_{\delta,k-1}}\rrangle_k
\leq \E\llangle\one_{\AA_{\delta^{1/4},k}}\rrangle_k + C(\beta)\delta,
}
which holds for all $n$ such that $\EEE_n\leq\delta$.
\end{proof}

Next we consider the event
\eq{
B_{\delta,k} \coloneqq \Big\{\frac{1}{n}\sum_{i} \llangle\vphi_i\rrangle_{k}^2 \leq \delta\Big\},
}
where $B_{\delta,0} = B_\delta$ is the event under consideration in Theorem~\ref{averages_squared}.

\begin{lemma} \label{alpha_lemma}
Assume $\beta$ is a point of differentiability for $p(\cdot)$, and $p'(\beta) < \beta$.
For any $\eps>0$, there is $\delta = \delta(\beta,\eps) >0$ sufficiently small that for any positive constant $K$, the following is true.
If $k(n) \in \{0,1,\dots,K\}$ for all $n$, then
\eeq{ \label{alpha_choice}
\limsup_{n\to\infty} \P(B_{\delta,k(n)}) \leq \eps.
}
\end{lemma}

\begin{proof}
By Theorem~\ref{averages_squared}, there is $\delta>0$ sufficiently small that
\eeq{ \label{alpha_choice_2}
\limsup_{n\to\infty} \P(B_{2\delta,0}) \leq \eps.
}
Let us write $\beta_n \coloneqq \beta\sqrt{1+\frac{k(n)}{n}}$, and
then observe that 
\eeq{ \label{alpha_choice_3}
\P(B_{\delta,k(n)}) 
&\stackrel{\phantom{\mbox{\footnotesize\eqref{temperature_equivalence}}}}{=} \P\Big(\frac{1}{n}\sum_i \llangle\vphi_i\rrangle_{k(n)}^2 \leq \delta\Big) \\
&\stackrel{\mbox{\footnotesize\eqref{temperature_equivalence}}}{=} \P\Big(\frac{1}{n}\sum_i \langle\vphi_i\rangle_{\beta_n}^2 \leq \delta\Big) \\
&\stackrel{\phantom{\mbox{\footnotesize\eqref{temperature_equivalence}}}}{\leq} \P(B_{2\delta,0})
+\P\bigg(\Big|\frac{1}{n}\sum_i \langle\vphi_i\rangle_{\beta_n}^2 -\frac{1}{n}\sum_i\langle\vphi_i\rangle_{\beta}^2\Big|\geq \delta\bigg). \raisetag{3.5\baselineskip}
}
Since $\sqrt{1+\frac{k(n)}{n}} \leq 1 + \frac{k(n)}{n} \leq 1 + \frac{K}{n}$, we have $0\leq \beta_n - \beta \leq \frac{\beta K}{n}$, and thus Lemma~\ref{connecting_betas}(c) gives
\eq{
\Big|\frac{1}{n}\sum_i \langle\vphi_i\rangle_{\beta_n}^2 -\frac{1}{n}\sum_i\langle\vphi_i\rangle_{\beta}^2\Big| \leq 2\sqrt{\beta K}\sqrt{F_n'(\beta_n) - F_n'(\beta)}.
}
By Lemma~\ref{betas_converging}, the right-hand side above converges to $0$ almost surely as $n\to\infty$.
In particular,
\eq{
\lim_{n\to\infty} \P\bigg(\Big|\frac{1}{n}\sum_i \langle\vphi_i\rangle_{\beta_n}^2 -\frac{1}{n}\sum_i\langle\vphi_i\rangle_{\beta}^2\Big|\geq \delta\bigg) = 0,
}
and so \eqref{alpha_choice} follows from \eqref{alpha_choice_2} and \eqref{alpha_choice_3}.
\end{proof}

\begin{prop} \label{pre_iteration_2}
Given any $\alpha>0$, there are positive constants $C_1(\alpha,\beta)$ and $C_2(\beta)$ 
such that the following holds for any $\delta_0\in(0,1)$.
There exists $n_0 = n_0(\delta_0)$ so that for every $n\geq n_0$, $k\geq1$, and $\delta\in[\delta_0,1)$,
\eq{
\E_{\vc h^{(k)}}\llangle \one_{\AA_{\delta,k-1}}\rrangle_k \geq \llangle \one_{\AA_{\delta,k-1}}\rrangle_{k-1} + C_1(\alpha,\beta)\llangle\one_{\AA_{\delta,k-1}}\rrangle_{k-1}\one_{B_{\alpha,k-1}^\mathrm{c}} - C_2(\beta)\sqrt{\delta}.
}
\end{prop}
\begin{proof}
Let $\delta_0\in(0,1)$ be given, and take $n_0$ such that $\EEE_n\leq\delta_0/2$ for all $n\geq n_0$.
Consider any $\delta\in[\delta_0,1)$,
and define the random variables
\eq{
X &\coloneqq \llangle \e^{\frac{\beta}{\sqrt{n}}\sum_ih_i^{(k)}\vphi_i}\rrangle_{k-1}, \\
X_1 &\coloneqq \llangle \one_{\AA_{\delta,k-1}}\e^{\frac{\beta}{\sqrt{n}}\sum_i h_i^{(k)}\vphi_i}\rrangle_{k-1}, \\
X_2 &\coloneqq \llangle \one_{\AA_{\delta,k-1}^\mathrm{c}}\e^{\frac{\beta}{\sqrt{n}}\sum_i h_i^{(k)}\vphi_i}\rrangle_{k-1}, \\
Y_1 &\coloneqq \E_{\vc h^{(k)}} X_1 \stackrel{\mbox{\footnotesize\eqref{freq_identity}}}{=} \e^{\frac{\beta^2}{2}}\llangle \one_{\AA_{\delta,k-1}}\rrangle_{k-1}, \\
Y_2 &\coloneqq \E_{\vc h^{(k)}} X_2 \stackrel{\mbox{\footnotesize\eqref{freq_identity}}}{=} \e^{\frac{\beta^2}{2}}\llangle \one_{\AA_{\delta,k-1}^\mathrm{c}}\rrangle_{k-1}.
}

\noindent {\bf Step 1.} \textit{Show that $X_1$ is concentrated at $Y_1$, but $X_2$ is not concentrated at $Y_2$ when $B_{\alpha,k-1}^\cc$ occurs.}

First observe that for any $\theta \in (-\infty,0] \cup [1,\infty)$, Jensen's inequality implies
\eeq{ \label{gaussian_bound}
\E_{\vc h^{(k)}} X^\theta 
\leq \E_{\vc h^{(k)}}\llangle \e^{\frac{\theta\beta}{\sqrt{n}}\sum_i h_i^{(k)}\vphi_i}\rrangle_{k-1} 
\stackrel{\mbox{\footnotesize\eqref{freq_identity}}}{=} \e^{\frac{(\theta\beta)^2}{2}}.
}
In particular, for any $t > \e^{\frac{\beta^2}{2}}\geq Y_2$,
\eeq{ \label{upper_X4_truncated}
\E_{\vc h^{(k)}}[(X_2 - Y_2)^2\one_{\{X_2>t\}}]
&\leq \frac{\E_{\vc h^{(k)}}[(X_2 - Y_2)^4\one_{\{X_2>t\}}]}{(t-\e^{\frac{\beta^2}{2}})^2} \\
&\leq \frac{\E_{\vc h^{(k)}}(X_2^4)}{(t-\e^{\frac{\beta^2}{2}})^2}
\leq \frac{\E_{\vc h^{(k)}}(X^4)}{(t-\e^{\frac{\beta^2}{2}})^2}
\stackrel{\mbox{\footnotesize\eqref{gaussian_bound}}}{\leq} \frac{\e^{8\beta^2}}{(t-\e^{\frac{\beta^2}{2}})^2}.
}
On the other hand,
\eeq{ \label{lower_X4_prep}
\Var_{\vc h^{(k)}}(X_2)
&= \Var_{\vc h^{(k)}}(X-X_1) \\
&= \Var_{\vc h^{(k)}}(X) - 2\Cov_{\vc h^{(k)}}(X,X_1) + \Var_{\vc h^{(k)}}(X_1) \\
&\geq \Var_{\vc h^{(k)}}(X) - 2\sqrt{\Var_{\vc h^{(k)}}(X)\Var_{\vc h^{(k)}}(X_1)}.
}
We have the upper bound
\eeq{
\Var_{\vc h^{(k)}}(X) \leq \E_{\vc h^{(k)}}(X^2) \stackrel{\mbox{\footnotesize\eqref{gaussian_bound}}}{\leq} \e^{2\beta^2}, \label{upper_X2}
}
as well as the lower bound
\eeq{ \label{lower_X2}
&\Var_{\vc h^{(k)}}(X) \\
&\stackrel{\phantom{\mbox{\footnotesize\eqref{freq_identity}}}}{=}
\E_{\vc h^{(k)}}\llangle \e^{\frac{\beta}{\sqrt{n}}\sum_i h_i^{(k)}(\vphi_i(\sigma^1) + \vphi_i(\sigma^2))}\rrangle_{k-1} - \big(\E_{\vc h^{(k)}} \llangle \e^{\frac{\beta}{\sqrt{n}}\sum_i h_i^{(k)}\vphi_i}\rrangle_{k-1}\big)^2 \\
&\stackrel{\mbox{\footnotesize\eqref{freq_identity}}}{=}\e^{\beta^2}\big(\llangle\e^{\frac{\beta^2}{n}\sum_i \vphi_i(\sigma^1)\vphi_i(\sigma^2)}\rrangle_{k-1} - 1\big) \\
&\stackrel{\phantom{\mbox{\footnotesize\eqref{freq_identity}}}}{\geq} \e^{\beta^2}\big(\e^{\frac{\beta^2}{n}\sum_i \llangle \vphi_i\rrangle_{k-1}^2}-1\big) \\
&\stackrel{\phantom{\mbox{\footnotesize\eqref{freq_identity}}}}{\geq} \e^{\beta^2}\frac{\beta^2}{n}\sum_i\llangle \vphi_i\rrangle_{k-1}^2. 
\raisetag{4\baselineskip}
}
Meanwhile, we have $\EEE_n\leq\delta_0/2\leq\delta/2$ for all $n\geq n_0$.
Hence Lemma~\ref{h_variance_lemma}(b) implies
\eeq{ \label{upper_X3}
\Var_{\vc h^{(k)}}(X_1) 
&\leq \e^{2\beta^2}\Big(\Big\llangle \one_{\AA_{\delta,k-1}}(\sigma)\frac{1}{n}\sum_i\vphi_i\llangle\vphi_i\rrangle_{k-1}\Big\rrangle_{k-1}+2\EEE_n\Big) \\ 
&\leq 2\e^{2\beta^2}\delta \quad \text{for all $n\geq n_0$.}
}
Using \eqref{upper_X2}--\eqref{upper_X3} in \eqref{lower_X4_prep} yields
\eeq{ \label{lower_X4_almost}
\Var_{\vc h^{(k)}}(X_2) \geq \beta^2\e^{\beta^2}\frac{1}{n}\sum_i\llangle \vphi_i\rrangle_{k-1}^2- 2\e^{2\beta^2}\sqrt{2\delta} \quad \text{for all $n\geq n_0$.}
}
So on the event $B_{\alpha,k-1}^\mathrm{c} = \{\frac{1}{n}\sum_{i} \llangle\vphi_i\rrangle_{k-1}^2 > \alpha\}$,
 \eqref{lower_X4_almost} shows
\eeq{ \label{lower_X4}
\Var_{\vc h^{(k)}}(X_2)\one_{B_{\alpha,k-1}^\mathrm{c}} \geq (\beta^2\e^{\beta^2}\alpha - 2\e^{2\beta^2}\sqrt{2\delta})\one_{B_{\alpha,k-1}^\mathrm{c}}
}
for all $n\geq n_0$.
Given $\alpha$ and $\beta$, we fix $t = t(\alpha,\beta)$ large enough such that 
\begin{subequations} \label{t_choices}
\begin{align}
\label{t_lower}
t > \e^{\frac{\beta^2}{2}} &\geq \max(Y_1,Y_2) \hspace{0.4in}
\intertext{and}
\label{t_upper}
\frac{\e^{8\beta^2}}{(t-\e^{\frac{\beta^2}{2}})^2} &\leq \frac{1}{2}\beta^2\e^{\beta^2}\alpha.
\end{align}
\end{subequations}
Because of \eqref{t_upper}, the inequalities \eqref{upper_X4_truncated} and \eqref{lower_X4} together yield
\eeq{ \label{lower_X4_truncated}
&\E_{\vc h^{(k)}}[(X_2-Y_2)^2\one_{\{X_2\leq t\}}]\one_{B_{\alpha,k-1}^\mathrm{c}} \\
&= \big(\Var_{\vc h^{(k)}}(X_2) - \E[(X_2-Y_2)^2\one_{\{X_2>t\}}]\big)\one_{B_{\alpha,k-1}^\mathrm{c}} \\
&\geq \Big(\frac{1}{2}\beta^2\e^{2\beta^2}\alpha - 2\e^{2\beta^2}\sqrt{2\delta}\Big)\one_{B_{\alpha,k-1}^\mathrm{c}} 
=(C_1(\alpha,\beta) - C_2(\beta)\sqrt{\delta})\one_{B_{\alpha,k-1}^\mathrm{c}} \raisetag{2.25\baselineskip}
}
for all $n\geq n_0$. \\

\noindent {\bf Step 2.} \textit{Since $X_1 \approx Y_1$, obtain an upper bound on the error in the following approximation:}
\eq{
\E_{\vc h^{(k)}}\Big(\frac{X_1}{X_1+X_2}\Big)
\approx \E_{\vc h^{(k)}}\Big(\frac{Y_1}{Y_1+X_2}\Big).
}

Simple algebra gives
\eq{
\frac{X_1}{X_1+X_2} - \frac{Y_1}{Y_1+X_2} 
&= \frac{X_2(X_1-Y_1)}{(X_1+X_2)(Y_1+X_2)} 
= \frac{X_2(X_1-Y_1)}{X(Y_1+X_2)},
}
and
\eeq{ \label{approximation_error_bound}
\Big|\E_{\vc h^{(k)}}\Big(\frac{X_2(X_1-Y_1)}{X(Y_1+X_2)}\Big)\Big|
&\stackrel{\phantom{\mbox{\footnotesize\eqref{upper_X3},\eqref{gaussian_bound}}}}{\leq} \E_{\vc h^{(k)}} \Big(\frac{|X_1-Y_1|}{X}\Big) \\
&\stackrel{\phantom{\mbox{\footnotesize\eqref{upper_X3},\eqref{gaussian_bound}}}}{\leq} \E_{\vc h^{(k)}}(X^{-2})\sqrt{\Var_{\vc h^{(k)}}(X_1)} \\
&\stackrel{\mbox{\footnotesize\eqref{gaussian_bound},\eqref{upper_X3}}}{\leq} C(\beta)\sqrt{\delta} \quad \text{for all $n\geq n_0$.}\\
}

\noindent {\bf Step 3.} \textit{Since $X_2$ is not concentrated at $Y_2$ when $B_{\alpha,k-1}^\cc$ occurs,}
\textit{obtain a lower bound on the gap in the following application of Jensen's inequality:}
\eq{
\E_{\vc h^{(k)}}\Big(\frac{Y_1}{Y_1+X_2}\Big) = \frac{Y_1}{Y_1+Y_2} + (\text{Jensen gap}).
}

We consider the function $f : (-Y_1,\infty) \to [0,1]$ given by
\eq{
f(x) \coloneqq \frac{Y_1}{Y_1+x}, \quad \text{for which} \quad f''(x) = \frac{2Y_1}{(Y_1+x)^3} \geq 0.
}
In particular, we consider its Taylor series approximation about $Y_2$,
\eq{
f(x) = f(Y_2) + (x-Y_2)f'(Y_2) + \frac{(x-Y_2)^2}{2}f''(\xi_x),
}
where $\xi_x$ belongs to the interval between $x$ and $Y_2$.
We note that such an expansion exists because the identity $Y_1 + Y_2 = \e^{\frac{\beta^2}{2}}$ shows $Y_2 > -Y_1$.
Jensen's inequality implies
\eq{
\E_{\vc h^{(k)}}f(X_2) \geq f(\E_{\vc h^{(k)}}X_2) = f(Y_2)
= \frac{Y_1}{Y_1+Y_2} = \llangle \one_{\AA_{\delta,k-1}}\rrangle_{k-1}.
}
We will now produce a lower bound on the Jensen gap.

First observe that $f''$ is decreasing on $(-Y_1,\infty)$.
Consequently, if $x \in [Y_2,t]$, then $f''(\xi_{x}) \geq f''(x) \geq f''(t)$.
Similarly, if $x\leq Y_2$, then $f''(\xi_{x}) \geq f''(Y_2) \geq f''(t)$.
Therefore, for all $n\geq n_0$, we have
\eeq{ \label{jensen_gap}
&\E_{\vc h^{(k)}}f(X_2) - \llangle \one_{\AA_{\delta,k-1}}\rrangle_{k-1} \\
&\stackrel{\phantom{\mbox{\footnotesize\eqref{t_lower},\eqref{t_choices}}}}{=} \E_{\vc h^{(k)}}f(X_2) - f(Y_2) \\
&\stackrel{\phantom{\mbox{\footnotesize\eqref{t_lower},\eqref{t_choices}}}}{=} \frac{\E_{\vc h^{(k)}}[(X_2-Y_2)^2f''(\xi_{X_2})]}{2} \\
&\stackrel{\phantom{\mbox{\footnotesize\eqref{t_lower},\eqref{t_choices}}}}{\geq} \frac{f''(t)}{2}\E_{\vc h^{(k)}}[(X_2-Y_2)^2\one_{\{X_2\leq t\}}] \\
&\stackrel{\phantom{\mbox{\footnotesize\eqref{t_lower},\eqref{t_choices}}}}{\geq} \frac{Y_1}{(Y_1+t)^3} \E_{\vc h^{(k)}}[(X_2-Y_2)^2\one_{\{X_2\leq t\}}]\one_{B_{\alpha,k-1}^\mathrm{c}} \\
&\stackrel{\mbox{\footnotesize\eqref{t_lower},\eqref{lower_X4_truncated}}}{\geq} \frac{Y_1}{8t^3}(C_1(\alpha,\beta) - C_2(\beta)\sqrt{\delta})\one_{B_{\alpha,k-1}^\mathrm{c}}\\
&\stackrel{\phantom{\mbox{\footnotesize\eqref{t_lower},\eqref{t_choices}}}}{\geq}C_1(\alpha,\beta)\llangle\one_{\AA_{\delta,k-1}}\rrangle_{k-1}\one_{B_{\alpha,k-1}^\mathrm{c}} - C_2(\beta)\sqrt{\delta}, \raisetag{5.5\baselineskip}
}
where the second term in the final expression need not depend on $\alpha$ since $Y_1/(8t^3) \leq 1$. \\

\noindent {\bf Step 4.} \textit{Reckon the final bound.}

In summary, for all $n\geq n_0$,
\eq{
&\E_{\vc h^{(k)}}\llangle \one_{\AA_{\delta,k-1}}\rrangle_k 
\stackrel{\mbox{\footnotesize\eqref{k_induction}}}{=} \E_{\vc h^{(k)}}\Big(\frac{X_1}{X_1+X_2}\Big) \\
&\stackrel{\mbox{\footnotesize\eqref{approximation_error_bound}}}{\geq} \E_{\vc h^{(k)}}\Big(\frac{Y_1}{Y_1+X_2}\Big) - C(\beta)\sqrt{\delta} \\
&\stackrel{\phantom{\mbox{\footnotesize\eqref{approximation_error_bound}}}}{=} \E_{\vc h^{(k)}} f(X_2) - C(\beta)\sqrt{\delta} \\
&\stackrel{\mbox{\footnotesize\eqref{jensen_gap}}}{\geq} \llangle \one_{\AA_{\delta,k-1}}\rrangle_{k-1} + C_1(\alpha,\beta)\llangle\one_{\AA_{\delta,k-1}}\rrangle_{k-1}\one_{B_{\alpha,k-1}^\mathrm{c}} - C_2(\beta)\sqrt{\delta}.
}
\end{proof}

\begin{proof}[Proof of Theorem~\ref{expected_overlap_thm}]
Let $\eps> 0$ be given.
From Lemma~\ref{alpha_lemma}, we fix $\alpha = \alpha(\beta,\eps)>0$ so that for any bounded sequence $(k(n))_{n\geq1}$ of nonnegative integers, we have 
\eeq{ \label{alpha_choice_real}
\limsup_{n\to\infty} \P(B_{\alpha,k(n)}) \leq \frac{\eps}{2}.
}
We wish to find $\delta_* > 0$, depending only on $\beta$ and $\eps$, such that $\E\llangle \one_{\AA_{\delta_*}}\rrangle \leq \eps$.

Let $\delta_0 \in (0,1)$, its exact value to be decided later.
From Proposition~\ref{pre_iteration_2}, we know that for all $n \geq n_0 = n_0(\delta_0)$ and $\delta\in[\delta_0,1)$,
\eq{
&\E\llangle \one_{\AA_{\delta,k-1}}\rrangle_k \\
&\geq \E\llangle \one_{\AA_{\delta,k-1}}\rrangle_{k-1} + C_1(\beta,\eps)\E(\llangle\one_{\AA_{\delta,k-1}}\rrangle_{k-1}\one_{B_{\alpha,k-1}^\mathrm{c}}) - C_2(\beta)\sqrt{\delta}.
}
And from Proposition~\ref{pre_iteration_1}, we can assume
\eq{
\E\llangle \one_{\AA_{\delta,k-1}}\rrangle_k
\leq \E\llangle\one_{\AA_{\delta^{1/4},k}}\rrangle_k + C(\beta)\delta \quad \text{for all $n\geq n_0$, $\delta\in[\delta_0,1)$.}
}
Linking the two inequalities, we find that
\eq{
&\E\llangle\one_{\AA_{\delta^{1/4},k}}\rrangle_k \\
&\geq \E\llangle \one_{\AA_{\delta,k-1}}\rrangle_{k-1} + \mathbf C_1(\beta,\eps)\E(\llangle\one_{\AA_{\delta,k-1}}\rrangle_{k-1}\one_{B_{\alpha,k-1}^\mathrm{c}}) - \mathbf C_2(\beta)\sqrt{\delta},
} 
where now we fix the constants $\mathbf C_1(\beta,\eps)$ and $\mathbf C_2(\beta)$.
Note that $\delta_0\leq\delta\leq\delta^{1/4}<1$, and so this reasoning can be iterated.
Iterating $K$ times produces the estimate
\eq{ 
1&\geq\E\llangle \one_{\AA_{\delta^{1/4^K},K}}\rrangle_{K} \\
&\geq \sum_{k=0}^{K-1}\Big[\mathbf C_1(\beta,\eps)\E(\llangle \one_{\AA_{\delta^{1/4^k},k}}\rrangle_k\one_{B_{\alpha,k}^\mathrm{c}}) - \mathbf C_2(\beta)\sqrt{\delta^{1/4^k}}\ \Big]
+ \E\llangle \one_{\AA_{\delta,0}}\rrangle_0,
}
which implies the existence of some $k=k(n)\in\{0,1,\dots,K-1\}$ such that 
\eeq{ \label{k_existence}
\mathbf C_1(\beta,\eps)\E(\llangle \one_{\AA_{\delta^{1/4^k},k}}\rrangle_k\one_{B_{\alpha,k}^\mathrm{c}}) - \mathbf C_2(\beta)\sqrt{\delta^{1/4^k}} &\leq \frac{1}{K}.
}
So we take $K = K(\beta,\eps)$ large enough that 
\eeq{ \label{K_choice}
\frac{1}{\mathbf C_1(\beta,\eps)K} \leq \frac{\eps}{6},
}
and then choose $\delta_0 = \delta_0(\beta,K)$ small enough that 
\eeq{ \label{eps_choice}
\mathbf C_2(\beta)\sqrt{\delta_0^{1/4^K}} \leq \frac{1}{K}.
}
We now have, for all $n\geq n_0$,
\eq{
\E(\llangle \one_{\AA_{\delta_0^{1/4^k},k}}\rrangle_k\one_{B_{\alpha,k}^\mathrm{c}}) &\stackrel{\mbox{\footnotesize\eqref{k_existence}}}{\leq} \frac{1}{\mathbf C_1(\beta,\eps)}\Big(\frac{1}{K} + \mathbf C_2(\beta)\sqrt{\delta_0^{1/4^K}}\Big)\\
&\stackrel{\mbox{\footnotesize\eqref{eps_choice}}}{\leq} \frac{2}{\mathbf C_1(\beta,\eps)K} 
\stackrel{\mbox{\footnotesize\eqref{K_choice}}}{\leq}\frac{\eps}{3}.
}
Combining this bound with \eqref{alpha_choice_real}, we see that
\eeq{ \label{eps_0_bound}
\E \llangle \one_{\AA_{\delta_0^{1/4^k},k}}\rrangle_k 
\leq \E(\llangle \one_{\AA_{\delta_0^{1/4^k},k}}\rrangle_k\one_{B_{\alpha,k}^\mathrm{c}}) + \P(B_{\alpha,k}) \leq \eps \quad \forall\text{ large $n$.}
}
To now complete the proof, we must obtain from this result an analogous one with $k=0$. 

As in the proof of Lemma~\ref{alpha_lemma}, we will write $\beta_n \coloneqq \beta\sqrt{1+\frac{k}{n}}$.
For $\eta>0$, define the set
\eq{
\wt{\AA}_{\eta,k} \coloneqq \Big\{\sigma^1\in\Sigma_n : \sum_i \vphi_i(\sigma^1)\langle \vphi_i(\sigma^2)\rangle_{\beta_n}\leq \eta \Big\}.
}
It follows from \eqref{temperature_equivalence} that
\eeq{ \label{expectations_same}
\llangle\one_{\AA_{\eta,k}}\rrangle_k \stackrel{\text{d}}{=} \langle\one_{\wt{\AA}_{\eta,k}}\rangle_{\beta_n} \quad \text{for any $\eta>0$,}
}
Since $0\leq \beta_n-\beta \leq \frac{\beta K}{n}$, 
Lemma~\ref{connecting_betas}(b) implies
\eq{
\Big|\frac{1}{n}\sum_{i} \vphi_i\langle\vphi_i\rangle_{\beta_n} - \frac{1}{n}\sum_i\vphi_i\langle\vphi_i\rangle_\beta\Big|
&\leq \sqrt{\beta K}\sqrt{F_n'(\beta_n) - F_n'(\beta)}.
}
Denote the right-hand side above by $\Delta_n$.
Take $\delta_* \coloneqq \frac{1}{2}{\delta_0} \leq \frac{1}{2}\delta_0^{1/4^k}$.
From the above display,
$\AA_{\delta_*,0} \subset \wt{\AA}_{\delta_*+\Delta_n,k}.$
Hence
\eq{
\E\langle\one_{\AA_{\delta_*,0}}\rangle_\beta
&\stackrel{\phantom{\mbox{\footnotesize\eqref{expectations_same}}}}{\leq}  \E\langle\one_{\wt{\AA}_{\delta_*+\Delta_n,k}}\rangle_\beta \\
& \stackrel{\phantom{\mbox{\footnotesize\eqref{expectations_same}}}}{\leq}   \P(\Delta_n > \delta_*) + \E \langle\one_{\wt{\AA}_{2\delta_*,k}} \rangle_\beta \\
&\stackrel{\mbox{\footnotesize\eqref{expectations_same}}}{=}  \P(\Delta_n > \delta_*) + \E \langle\one_{\wt{\AA}_{2\delta_*,k}} \rangle_\beta -
\E \langle\one_{\wt{\AA}_{2\delta_*,k}} \rangle_{\beta_n} +\E\llangle\one_{\AA_{2\delta_*,k}}\rrangle_k \\
& \stackrel{\phantom{\mbox{\footnotesize\eqref{expectations_same}}}}{\leq}  \P(\Delta_n > \delta_*) + \E \langle\one_{\wt{\AA}_{2\delta_*,k}} \rangle_\beta -
\E \langle\one_{\wt{\AA}_{2\delta_*,k}} \rangle_{\beta_n} +\E\llangle\one_{\AA_{\delta_0^{1/4^k},k}}\rrangle_k.
}
And by Lemma~\ref{connecting_betas}(a),
\eq{
 | \langle\one_{\wt{\AA}_{2\delta_*,k}} \rangle_{\beta_n}
 -  \langle\one_{\wt{\AA}_{2\delta_*,k}} \rangle_\beta |
  \leq \Delta_n.
}
From the previous two displays and \eqref{eps_0_bound}, we have
\eq{
\E\langle\one_{\AA_{\delta_*,0}}\rangle_\beta \leq \P(\Delta_n > \delta_*) + \E(\Delta_n) + \eps \quad \text{for all large $n$}.
}
Finally, Lemma~\ref{betas_converging} shows that $\Delta_n \to 0$ almost surely and in $L^1$ as $n\to\infty$.
Consequently,
$\limsup_{n\to\infty} \E\langle\one_{\AA_{\delta_*,0}}\rangle_\beta \leq \eps$.
\end{proof}

\section{Proof of equivalence of Theorems~\ref{easy_cor} and~\ref{expected_overlap_thm}} \label{cor_proof}
Theorem~\ref{easy_cor} is implied by Theorem~\ref{expected_overlap_thm} once we establish the following result.
Recall the definitions \eqref{ball_def} and \eqref{A_def}.

\begin{prop}
Suppose $H_n$ is defined by \eqref{field_decomposition},  where $(g_i)_{i=1}^\infty$ are i.i.d. random variables with zero mean and unit variance (not necessarily Gaussian).
Assume \eqref{free_energy_assumption}--\eqref{positive_overlap}. 
Then the following two statements are equivalent:
\begin{itemize}
\item[$\mathrm{(S1)}$] For every $\eps > 0$, there exist integers $k = k(\beta,\eps)$ and $n_0 = n_0(\beta,\eps)$ and a number $\delta = \delta(\beta,\eps)>0$ such that the following is true for all $n\geq n_0$.
With $\P$-probability at least $1-\eps$, there exist $\sigma^1,\dots,\sigma^k\in\Sigma_n$ such that 
\eq{
\mu_{n}^\beta\Big(\bigcup_{j=1}^k \BB(\sigma^j, \delta)\Big) \geq 1 - \eps.
}
\item[$\mathrm{(S2)}$] For every $\eps>0$, there exists $\delta = \delta(\beta,\eps) > 0$ sufficiently small that
\eq{ 
\limsup_{n\to\infty} \E\langle\one_{\AA_{n,\delta}}\rangle \leq \eps.
}
\end{itemize}
\end{prop}

\subsection{Proof of {$\mathrm{(S2)} \Rightarrow \mathrm{(S1)}$}}
Let  $\eps > 0$ be given.
By $\mathrm{(S2)}$, we can choose $\delta > 0$ small enough and $n_0$ large enough  so that for all $n\geq n_0$,
\eq{
\E\langle\one_{\AA_{n,2\delta}}\rangle \leq \frac{\eps^2}{2}.
}
It follows from Markov's inequality that
\eeq{ \label{Markov_for_cor}
\P\Big(\langle\one_{\AA_{n,2\delta}}\rangle > \frac{\eps}{2}\Big) \leq \eps.
}
Now, by the Paley--Zygmund inequality, for any $j\neq k+1$,
\eq{
\givena{\one_{\{\RR_{j,k+1}\geq\delta\}}}{\sigma^{k+1}}\one_{\{\RR(\sigma^{k+1}) > 2\delta\}} 
&\geq\frac{1}{4}\frac{\RR(\sigma^{k+1})^2}{\givena{\RR_{j,k+1}^2}{\sigma^{k+1}}}\one_{\{\RR(\sigma^{k+1}) > 2\delta\}} \\
&\geq\delta^2\one_{\{\RR(\sigma^{k+1})>2\delta\}}.
}
Therefore,
\eq{
\givena{\one_{\bigcap_{j=1}^k\{\RR_{j,k+1}<\delta\}}}{\sigma^{k+1}}\one_{\{\RR(\sigma^{k+1}) > 2\delta\}}  \leq (1 - {\delta^2})^k \leq \e^{-\delta^2 k}.
}
Choosing $k = \ceil{-\delta^{-2}\log(\eps/2)} \vee 0$, we have
\eq{
\langle\one_{\bigcap_{j=1}^k\{\RR_{j,k+1}<\delta\}}\rangle 
\leq \frac{\eps}{2} + \langle\one_{\{\RR(\sigma^{k+1})\leq2\delta\}}\rangle
= \frac{\eps}{2} + \langle \one_{\AA_{n,2\delta}}\rangle.
}
Therefore, 
\eq{
\P\Big(\langle \one_{\bigcup_{j=1}^k\{\RR_{j,k+1}\geq\delta\}}\rangle \geq 1-\eps\Big)
&= \P\Big(\langle \one_{\bigcap_{j=1}^k\{\RR_{j,k+1}<\delta\}}\rangle \leq \eps\Big) \\
&\geq \P\Big(\langle \one_{\AA_{n,2\delta}}\rangle \leq \frac{\eps}{2}\Big) 
\stackrel{\mbox{\footnotesize\eqref{Markov_for_cor}}}{\geq} 1 - \eps.
}
This completes the proof, since
\eq{
\mu_{n}^\beta\Big(\bigcup_{j=1}^k \BB(\sigma^j,\delta)\Big) = \langle \one_{\bigcup_{j=1}^k\{\RR_{j,k+1}\geq\delta\}}\rangle.
}

\subsection{Proof of {$\mathrm{(S1)} \Rightarrow \mathrm{(S2)}$}}

We begin with a lemma that roughly states the following. 
If many 
random variables each have non-negligible positive correlation with a distinguished variable, then at least one pair of these variables has non-negligible positive correlation.

\begin{lemma} \label{one_to_pair}
For any $\delta\in(0,1]$, there exists $N_0 = N_0(\delta)$ such that the following holds for any integer $N\geq N_0$ and any $\sigma^0\in\Sigma_n$.
If $\sigma^1,\dots,\sigma^N\in\BB(\sigma^0,\delta)\subset\Sigma_n$, then
\eeq{ \label{one_to_pair_ineq}
\RR_{j,k} \geq \frac{\delta^2}{2} \quad \text{for some $1\leq j<k\leq N$}.
}
\end{lemma}

\begin{proof}
Consider the $(N+1)\times(N+1)$ matrix $\RR = (\RR_{j,k})_{0\leq i,j\leq N}$, where
\eq{
\RR_{j,k} = \RR(\sigma^j,\sigma^k) = \frac{1}{n}\sum_i \vphi_i(\sigma^j)\vphi_i(\sigma^k).
}
Observe that $\RR$ is positive semi-definite: for any $\vc x \in \R^{N+1}$,
\eq{
\langle \vc x,\RR \vc x\rangle = \sum_{0\leq j,k\leq N} \RR_{j,k}x_jx_k
&= \frac{1}{n}\sum_i \sum_{0\leq j,k\leq N} x_j\vphi_{i}(\sigma^j)x_k\vphi_{i}(\sigma^k) \\
&= \frac{1}{n}\sum_i \bigg(\sum_{j=0}^N x_j\vphi_i(\sigma^j)\bigg)^2 \geq 0.
}
Now let $\eta \coloneqq 0\vee\max_{1\leq j < k\leq N} \RR_{j,k}$. 
For $\vc x = (1,-x,\dots,-x) \in\R^{1+N}$ with $x\geq0$, our assumptions give 
\eq{
0 \leq \iprod{\vc x}{\RR \vc x} \leq 1 + Nx^2 - 2\delta N x + \eta N^2 x^2.
}
We now take $x = \delta/(1+\eta N)$ to obtain
\eq{
0 &\leq 1 + N\Big(\frac{\delta}{1+\eta N}\Big)^2 - 2\delta N\frac{\delta}{1+\eta N} + \eta N^2\Big(\frac{\delta}{1+\eta N}\Big)^2 \\
&= 1 + \frac{\delta^2}{1+\eta N}\Big[\frac{N}{1+\eta N} - 2N + \frac{\eta N^2}{1+\eta N}\Big] = 1-\frac{\delta^2N}{1+\eta N}.
} 
Supposing that $\eta < \delta^2/2$, we further see
\eq{
0 \leq 1 - \frac{\delta^2N}{1+\eta N} \leq 1 - \frac{\delta^2N}{1+\delta^2 N/2},
}
which yields a contradiction as soon as $\frac{\delta^2N}{1+\delta^2 N/2} > 1$.
\end{proof}

We will contrast Lemma~\ref{one_to_pair} with the one below, which says that if $\delta$ is small enough, then any non-negligible subset of $\AA_{n,\delta}$ has many nearly orthogonal elements.

\begin{lemma} \label{many_orthogonal}
For any $\eps_1,\eps_2>0$ and positive integer $N$, there is $\delta = \delta(\eps_1,\eps_2,N) > 0$ such that the following holds. 
If $\AA\subset\AA_{n,\delta}$ with $\langle \one_{\AA}\rangle \geq \eps_1$, then there are $\sigma^1,\dots,\sigma^N\in \AA$ such that
\eq{
\RR_{j,k} < \eps_2 \quad \text{for all $1\leq j<k\leq N$.}
}
\end{lemma}

\begin{proof}
Set $\delta \coloneqq \eps_1\eps_2/N$. 
Observe that for any $\sigma\in\AA$, we have the following implication:
\eeq{\label{high_overlap_set}
\delta \geq \RR(\sigma) \geq \eps_2\langle \one_{\BB(\sigma,\eps_2)}\rangle \quad \Rightarrow \quad \langle\one_{\BB(\sigma,\eps_2)}\rangle \leq \frac{\delta}{\eps_2} = \frac{\eps_1}{N}.
}
Therefore, one can inductively choose
\eq{
\sigma^1 \in \AA, \quad \sigma^2 \in \AA\setminus\BB(\sigma^1,\eps_2), \quad \sigma^3\in\AA\setminus(\BB(\sigma^1,\eps_2)\cup\BB(\sigma^2,\eps_2)), \dots
}
where \eqref{high_overlap_set} guarantees that
\eq{
\mu_n^\beta\big(\AA\setminus(\BB(\sigma^1,\eps_2)\cup\cdots\BB(\sigma^{k-1},\eps_2))\big) \geq \eps_1 - (k-1)\frac{\eps_1}{N}.
}
Hence $\sigma^{k}\in\AA\setminus(\BB(\sigma^1,\eps_2)\cup\cdots\BB(\sigma^{k-1},\eps_2))$ can be found so long as $k\leq N$.

\end{proof}

We can now complete the proof.
Assume that $\mathrm{(S1)}$ holds.
Suppose, contrary to $\mathrm{(S2)}$, that there is some $\eps \in (0,1)$ such that for every $\delta > 0$,
\eeq{ \label{contradiction_setup}
\limsup_{n\to\infty} \E\langle \one_{\AA_{n,\delta}}\rangle > 4\eps.
}
Note that for any $n$ such that $\E\langle \one_{\AA_{n,\delta}}\rangle\geq4\eps$, we have
\eq{
4\eps \leq \E\langle\one_{\AA_{n,\delta}}\rangle &\leq \P(\langle\one_{\AA_{n,\delta}}\rangle \geq 2\eps) +2\eps\P(\langle\one_{\AA_{n,\delta}}\rangle<2\eps) \\
&= (1-2\eps)\P(\langle\one_{\AA_{n,\delta}}\rangle\geq2\eps)+2\eps,
}
and thus $\P(\langle\one_{\AA_{n,\delta}}\rangle \geq2\eps) \geq 2\eps$. 

From $\mathrm{(S1)}$, we choose $k$ and $\delta$ so that for all $n$ large enough (depending on $\eps$ on $\beta$), the following is true with $\P$-probability at least $1-\eps$:
There exist $\sigma^1,\dots,\sigma^k \in \Sigma_n$ such that
\eeq{ \label{balls_covering}
\mu_{n}^\beta\Big(\bigcup_{j=1}^k \BB(\sigma^j, \delta)\Big) \geq 1-\eps.
}
Once $\delta$ has been determined, choose $N$ so that the conclusion of Lemma~\ref{one_to_pair} holds.
Then, given the values of $k$ and $N$, choose $\delta'$ so that the conclusion of Lemma~\ref{many_orthogonal} holds with $\eps_1 = \eps/k$ and $\eps_2 = \delta^2/2$.


In summary, if $n$ is large enough, and $\E\langle \one_{\AA_{n,\delta'}}\rangle\geq4\eps$ 
(by \eqref{contradiction_setup}, there are infinitely many $n$ for which this is the case), the following is true. 
With $\P$-probability at least $2\eps-\eps=\eps$, we have both $\langle \one_{\AA_{n,\delta'}}\rangle \geq 2\eps$ and \eqref{balls_covering} for some $\sigma^1,\dots,\sigma^k\in\Sigma_n$.
In this case, we have
\eq{
\mu_{n}^\beta\bigg(\AA_{n,\delta'} \cap \Big(\bigcup_{j=1}^k \BB(\sigma^j, \delta)\Big)\bigg) \geq 2\eps - \eps = \eps.
}
Therefore, there is some $j$ such that
\eq{
\mu_{n}^\beta\big(\AA_{n,\delta'} \cap \BB(\sigma^j, \delta)\big) \geq \frac{\eps}{k}.
}
By our choice of $\delta'$, we can find $\sigma^1,\dots,\sigma^{N}\in\AA_{n,\delta'} \cap \BB(\sigma^j, \delta)$ satisfying
\eq{
\RR_{j,k} < \frac{\delta^2}{2} \quad \text{for all $1\leq j<k\leq N$.}
}
But $\sigma^1,\dots,\sigma^{N}\in\BB(\sigma^j, \delta)$, and so the above display contradicts \eqref{one_to_pair_ineq}.

\section{Polymer measures are asymptotically non-atomic} \label{no_path_atom}
In this section we prove that directed polymers on the lattice are asymptotically non-atomic.
It is a striking phenomenon that at sufficiently small temperatures, the polymer endpoint distribution places a non-vanishing mass on a single element of $\Z^d$ (which is random and varies with $n$) \cite{comets-shiga-yoshida03}.
The fact that the polymer measures themselves do not share this property, stated below as Theorem~\ref{no_path_atom_thm}, justifies the investigation of replica overlap as an order parameter for path localization.
To emphasize the fact that the Gaussian environment can be replaced by a general one, we reintroduce notation for directed polymers.

Let $(\omega(i,x) : i\geq1, x \in\Z^d)$ be a collection of i.i.d.~random variables.
We will assume that 
\eeq{
\E(\e^{t\omega(i,x)}) < \infty \quad \text{for some $t > 0$}, \label{polymer_exp_moment}
} 
and also that 
\eeq{
\Var(\omega(i,x))> 0 \label{polymer_var}
} 
in order to avoid trivialities. 
Let $\PP_n$ denote the set of nearest-neighbor paths of length $n$ in $\Z^d$ starting at the origin.
Note that $|\PP_n| = (2d)^n$.
To each $\vc x = (0,x_1,\dots,x_n)$ in $\PP_n$ we associate the Hamiltonian energy
\eq{
H_n(\vc x) \coloneqq \sum_{i=1}^n \omega(i,x_i).
}
The polymer measure is then defined by
\eq{
\mu_{n}^\beta(\vc x) \coloneqq \frac{\e^{\beta H_n(\vc x)}}{\sum_{\vc y}\e^{\beta H_n(\vc y)}}, \quad \vc x \in \PP_n.
}

\begin{thm} \label{no_path_atom_thm}
Assume \eqref{polymer_exp_moment}.
Then for any $d\geq1$ and any $\beta\in[0,\infty)$,
\eeq{ \label{no_atom}
\max_{\vc x\in\PP_n} \mu_{n}^\beta(\vc x) = O(n^{-1}) \quad \mathrm{a.s.}\text{ as $n\to\infty$}.
}
\end{thm}

The remainder of Section~\ref{no_path_atom} is to prove Theorem~\ref{no_path_atom_thm}.
We begin by defining the \textit{passage time},
\eq{
L_n \coloneqq \max_{\vc x\in \PP_n} H_n(\vc x).
}
We will denote the set of maximizing paths by
\eeq{
\MM_n \coloneqq \{\vc x \in \PP_n : H_n(\vc x) = L_n\}. \label{maximizing_paths}
}
It is well-known (for instance, see \cite{georgiou-rassoul-seppalainen16}) that there is a finite constant $\lambda$ such that
\eeq{ \label{time_constant}
\lim_{n\to\infty} \frac{L_n}{n} = \sup_{n\geq1} \frac{\E(L_n)}{n} = \lambda \quad \mathrm{a.s.}
} 
The first equality above is a consequence of the superadditivity of $L_n$, and the second equality leads to a short proof of the following standard fact.

\begin{lemma} \label{trivial_lemma}
$\lambda > \E(\omega(i,x))$.
\end{lemma}

\begin{proof}
Let $\vc a = (1,0,\dots,0) \in \Z^d$ and $\vc 0 = (0,\dots,0) \in \Z^d$.
Observe that $L_2 \geq \max\{\omega(1,\vc a) + \omega(2,\vc 0),\, \omega(1,-\vc a) + \omega(2,\vc 0)\}$, and so
\eq{
2\lambda \geq \E(L_2) \geq \E\max\{\omega(1,\vc a) + \omega(2,\vc 0),\omega(1,-\vc a) + \omega(2,\vc 0)\} > 2\E(\omega(i,x)),
}
where the final equality is strict because $\Var(\omega(i,x)^2) > 0$. 
\end{proof}

\begin{defn}
For a nearest-neighbor path $\vc x = (x_0,x_1,\dots,x_n)$ of length $n$ in $\Z^d$, define the \textit{turns} of $\vc x$ to be the following set of indices:
\eeq{ \label{set_of_turns}
T(\vc x) \coloneqq \{1 \leq i \leq n-1 : x_{i+1} - x_i \neq x_i - x_{i-1}\}.
}
The number of turns of $\vc x$ will be denoted $t(\vc x) \coloneqq |T(\vc x)|$.

\end{defn}

\begin{lemma} \label{many_turns}
For any $\eps > 0$, there is $\delta = \delta(\eps,d) > 0$ small enough that
\eq{
|\{\vc x \in \PP_n : t(\vc x) < \delta n\}| \leq C(\eps,d)(1+\eps)^n \quad \text{for all $n\geq1$}.
}
\end{lemma}

\begin{proof}
Given an integer $j$, $0 \leq j \leq n-1$, we count the elements of $\{\vc x\in\PP_n : t(\vc x) = j\}$ as follows.
First, the number of choices for $x_1$ is $2d$.
Next, a turn should occur at exactly $j$ of the coordinates $x_1,\dots,x_{n-1}$.
Moreover, if a turn occurs at $x_i$, then there are $2d-1$ choices for $x_{i+1}-x_i$ (so as to avoid $x_i - x_{i-1}$).
Finally, if a turn does not occur at $x_i$, then there is only one choice for $x_{i+1}-x_i$, namely $x_i - x_{i-1}$.
Therefore, for any positive integer $k \leq \frac{n-1}{2}$, 
\eq{
|\{\vc x \in \PP_n : t(\vc x) < k\}| &= \sum_{j=0}^{k-1} 2d{n-1 \choose j}(2d-1)^j 
\leq 2dk{n-1\choose k}(2d-1)^{k-1}.
}
If $k = \ceil{\delta n}$ for $\delta\in(0,\frac{1}{2})$, then Stirling's approximation gives
\eq{
\lim_{n\to\infty} \frac{1}{n}\log {n - 1 \choose k} = -\delta \log \delta - (1-\delta)\log (1-\delta).
}
Therefore,
\eq{
&\limsup_{n\to\infty}\frac{\log |\{\vc x \in \PP_n : t(\vc x) < \delta n\}|}{n} \\
&\leq -\delta \log \delta - (1-\delta)\log (1-\delta) + \delta \log(2d-1).
}
Now choose $\delta$ sufficiently small that the right-hand side above is strictly less than $\log(1+\eps)$.
Inverting the logarithm and choosing $C$ large enough now yields the desired result.
\end{proof}

\begin{lemma} \label{not_many_gaps}
Let $\{(\omega_i,\omega_i')\}_{i=1}^\infty$ denote a sequence of i.i.d.~pairs of independent random variables.
For any $\eps > 0$ and $\nu > 0$, there exists $D>0$ large enough that
\eq{
\P(|\{1\leq i\leq n-1 : \omega_i > \omega_i' + D\}| > \nu n) \leq \eps^n \quad \text{for all $n\geq1$}.
}
\end{lemma}

\begin{proof}
Choose $D>0$ large enough that $p \coloneqq \P(\{|\omega_i|\geq D/2\}\cup\{|\omega_i'|\geq D/2\})$ satisfies $p^{\nu} \leq \eps/2$.
We then have
\eq{
&\P(|\{1\leq i\leq n : \omega_i > \omega_i'+D\}| > \nu n) \\
&\leq \P(|\{1\leq i\leq n-1: |\omega_i| \geq D/2 \text{ or } |\omega_i'| \geq D/2\}| > \nu n) \\
&\leq \sum_{j = \ceil{\nu n}}^{n-1} {n \choose j}p^j(1-p)^j 
\leq p^{\nu n}2^{n-1} \leq \eps^n.
}
\end{proof}

\begin{proof}[Proof of Theorem~\ref{no_path_atom_thm}]
Let $\omega$ denote a generic copy of $\omega(i,x)$, and $\bar\omega \coloneqq \E(\omega)$.
Set $\kappa \coloneqq (\lambda - \bar\omega)/2$, which is positive by Lemma~\ref{trivial_lemma}.
By assumption, there is $t > 0$ such that $\E(\e^{t \omega}) < \infty$.
Take any $s\in(0,t)$ and observe that for any given $\vc x\in\PP_n$,
\eq{
\P(H_n(\vc x) \geq (\bar\omega +\kappa) n)
&\leq \P(\e^{s(H_n(\vc x)-\bar\omega n)}\geq \e^{s\kappa n})
\leq \e^{-s \kappa n}\E(\e^{s(\omega-\bar\omega)})^n.
}
Using dominated convergence, it is easy to show that
\eq{
\lim_{s\searrow0} \frac{\E(\e^{s(\omega-\bar\omega)})-1}{\e^{s\kappa}-1}
= \lim_{s\searrow0} \frac{\E((\omega-\bar\omega)\e^{s(\omega-\bar\omega)})}{\kappa\e^{s\kappa}}
= 0,
}
and so we may choose $s$ sufficiently small that $\e^{-s\kappa}\E(\e^{s(\omega-\bar\omega)}) < 1$.
Set $\eta \coloneqq 1- \e^{-s\kappa}\E(\e^{s(\omega-\bar\omega)})$, and then choose $\eps > 0$ sufficiently small that $(1+\eps)(1-\eta) < 1$.
With $\delta$ as in Lemma~\ref{many_turns}, we have the union bound
\eq{
\P(\exists\, \vc x\in\PP_n : t(\vc x) < \delta n, H_n(\vc x) \geq (\bar\omega+\kappa)n) \leq C(1+\eps)^n(1-\eta)^n.
}
By our choice of $\eps$, Borel--Cantelli implies that the following statement holds almost surely:
\eq{
\exists\, n_0 : \forall\,  n\geq n_0,\,  \forall\, \vc x\in\PP_n, \quad
t(\vc x) < \delta n \Rightarrow H_n(\vc x) < (\bar\omega + \kappa)n.
}
On the other hand, it is apparent from \eqref{time_constant} and our choice of $\kappa$ that almost surely, we have $L_n > (\bar\omega+\kappa)n$ for all large $n$. 
For any such $n$, we then have $H_n(\vc x) > (\bar\omega+\kappa)n$ for every $\vc x\in\MM_n$, the set of maximizing paths defined in \eqref{maximizing_paths}. 
That is, almost surely:
\eq{
\exists\, n_1 : \forall\,  n\geq n_1,\,  \forall\, \vc x\in\MM_n, \quad
H_n(\vc x) \geq (\bar\omega + \kappa)n.
}
Together, the two previous displays show that almost surely,
\eeq{ \label{statement_1}
\exists\, n_2 : \forall\,  n\geq n_2,\,  \forall\, \vc x\in\MM_n, \quad
t(\vc x) \geq \delta n.
}
Recall from \eqref{set_of_turns} that $T(\vc x)$ denotes the set of turns in the path $\vc x\in\PP_n$.
For a given $\vc x\in\PP_n$ and $i \in T(\vc x)$, 
let $\vc x^{(i)}$ denote the unique element of $\PP_n$ such that $x^{(i)}_i \neq x_i$ but $x^{(i)}_j = x_j$ for all $j\neq i$. 
That is, $x^{(i)}_i - x^{(i)}_{i-1} = x_{i+1}-x_i$ while $x^{(i)}_{i+1} - x^{(i)}_{i} = x_{i}-x_{i-1}$.
Upon taking $\eps = 1/(4d)$ and $\nu = \delta/3$ in Lemma~\ref{not_many_gaps}, a union bound gives
\eq{
\P\Big(\exists\, \vc x\in\PP_n : |\{i \in T(\vc x) : H_n(\vc x) > H_n(\vc x^{(i)}) + D\}| > \frac{\delta}{3}n\Big) \leq 2^{-n}.
}
Therefore, we can again apply Borel--Cantelli to see that almost surely,
\eq{
\exists\, n_3 : \forall\,  n\geq n_3,\,  \forall\, \vc x\in\PP_n, \quad
|\{i \in T(\vc x) : H_n(\vc x) > H_n(\vc x^{(i)}) + D\}| \leq \frac{\delta}{3}n.
}
Now combining this statement with \eqref{statement_1}, we arrive at the following almost sure event:
\eq{   
\exists\, n_4 : \forall\,  n\geq n_4,\,  \forall\, \vc x\in\MM_n, \quad
|\{i \in T(\vc x) : H_n(\vc x) \leq H_n(\vc x^{(i)}) + D\}| \geq \frac{2\delta}{3}n.
}
In particular, since $\MM_n$ has at least one element (call it $\vc y$), we have the following for all $n\geq n_4$:
\eq{
\max_{\vc x\in\PP_n} \mu_n^\beta(\vc x) = \frac{\e^{\beta H_n(\vc y)}}{\sum_{x\in\PP_n}\e^{\beta H_n(\vc x)}}
&\leq \frac{\e^{\beta H_n(\vc y)}}{\sum_{i\in T(\vc y)}\e^{\beta H_n(\vc y^{(i)})}} \\
&\leq \frac{\e^{\beta H_n(\vc y)}}{\frac{2\delta}{3}n \e^{\beta H_n(\vc y)}\e^{-\beta D}} = \frac{3\e^{\beta D}}{2\delta n}.
}
Since $D$ and $\delta$ do not depend on $n$, \eqref{no_atom} follows.
\end{proof}

\section{Acknowledgments}
We are grateful to Francis Comets for valuable feedback and discussion, and to the referees for their beneficial comments, suggestions, and edits.

\bibliography{path_localization}

\end{document}